\pdfoutput=1

\documentclass[english]{jnsao}
\usepackage[capitalise]{cleveref}	
\usepackage{bbold}			
\usepackage{thm-restate}

\newcommand{\N}{\mathbb{N}}

\newcommand{\R}{\mathbb{R}}
\newcommand{\di}{\mathrm{d}}
\newcommand{\Id}{\mathrm{Id}}
\renewcommand{\epsilon}{\varepsilon}
\DeclareMathOperator*{\essinf}{ess\,inf}	
\newcommand{\Exp}{E}			

\newcommand{\norm}[1]{{\lVert #1 \rVert}}
\newcommand{\bignorm}[1]{\left\lVert #1 \right\rVert}

\newcommand{\X}{\mathcal{X}}	
\newcommand{\Y}{\mathcal{Y}}	
\newcommand{\calP}{\mathcal{P}} 
\newcommand{\T}{\mathcal{T}}	
\newcommand{\calH}{\mathcal{H}}
\newcommand{\calR}{\mathcal{R}} 
\newcommand{\calZ}{\mathcal{Z}} 
\newcommand{\CL}{\mathit{CL}}	
\newcommand{\cpt}{\mathrm{c}}	
\newcommand{\dint}{\mathfrak{d}}	
\newcommand{\Normal}{\mathcal{N}}

\usepackage{enumerate,enumitem}
\setenumerate[1]{label=(\roman*)}

\theoremstyle{definition}
\newtheorem{assumption}[theorem]{Assumption}
\numberwithin{equation}{section}
\crefname{assumption}{Assumption}{Assumptions}

\title{Goodness-of-fit testing for nonlinear inverse problems with random observations}
\shorttitle{Goodness-of-fit testing for nonlinear inverse problems}
\author{%
	Remo Kretschmann\thanks{Institute of Mathematics, University of Potsdam, Karl-Liebknecht-Straße 24--25, 14476 Potsdam, Germany\\(\email{remo.kretschmann@uni-potsdam.de}, \email{han.lie@uni-potsdam.de}).}
	\and Han Cheng Lie\footnotemark[1]
}
\shortauthor{R.~Kretschmann, H.C.~Lie}
\date{}

\hypersetup{
    pdftitle = {Goodness-of-fit testing for nonlinear inverse problems with random observations},
    pdfauthor = {Remo Kretschmann, Han Cheng Lie},
}

\begin{document}

\maketitle

\begin{abstract}
This work is concerned with nonparametric goodness-of-fit testing in the context of nonlinear inverse problems with random observations. Bayesian posterior distributions based upon a Gaussian process prior distribution are proven to contract at a certain rate uniformly over a set of true parameters. The corresponding posterior mean is shown to converge uniformly at the posterior contraction rate in the sense of satisfying a concentration inequality. Distinguishability for bounded alternatives separated from a composite null hypothesis at the posterior contraction rate is established using infimum plug-in tests based on the posterior mean and also on maximum a posteriori estimators. The results are applied to a class of inverse problems governed by ordinary differential equation initial value problems that is widely used in pharmacokinetics. For this class, uniform posterior contraction rates are proven and then used to establish distinguishability.
\end{abstract}

\small{
\vskip2ex

}
\normalsize
\section{Introduction}
\label{sec:intro}

In this work, we consider inverse problems where one is interested in a quantity $\theta$ that cannot be observed directly. Instead, only noisy observations of a related quantity $G(\theta)$ are available. We consider the nonparametric case where both $\theta$ and $G(\theta)$ are functions, and where $G(\theta)$ is only observed at $N$ random points $X_1, \dots, X_N$ drawn from a potentially unknown distribution $\mu_X$. In this setting, we perform statistical inference for the quantity $\theta$ based upon the points $X_1, \dots, X_N$ and the corresponding noisy observations
\[ Y_i = [G(\theta)](X_i) + \epsilon_i, \quad i \in \{1, \dots, N\}. \]
Problems of this kind arise for example in pharmacokinetic applications such as model-informed precision dosing, see \cite{HarWahRasHui:2021,Lie:2024}.

We investigate goodness-of-fit testing in this indirect nonlinear nonparametric setting. Here, one assesses whether the function $\theta$ lies in a given class of functions $\Theta_0$ by means of a test $\Psi_N$ which accepts or rejects the hypothesis $H_0$: $\theta \in \Theta_0$ based upon the available observations $(X_i,Y_i)_{i=1}^N$. The term `goodness-of-fit testing' is typically used to highlight the fact that a \emph{composite} hypothesis $\theta \in \Theta_0$ is tested, in contrast with testing a simple hypothesis $\theta = \theta_0$, i.e., testing if $\theta$ has a specific value $\theta_0$.
It is then of interest to determine the separation in a certain metric of an alternative set of functions $\Theta_1$ from the set $\Theta_0$, such that the type $1$ and type $2$ errors of an optimal hypothesis test for testing $H_0$ against $H_1$: $\theta \in \Theta_1$ lie below a certain bound. Considered asympotically, this leads to the question at what rate this separation is allowed to converge to zero as the number of observations increases while $H_0$ and $H_1$ are still distinguishable by an optimal test, i.e., while the errors of an optimal test converge to zero. In addition, it may be of interest to determine how this so-called `minimax separation rate' depends on the smoothness of the functions in $\Theta_0$ and $\Theta_1$, and on the degree of ill-posedness of the inverse problem at hand.

The setting described above presents several challenges for the development of goodness-of-fit tests. These challenges include the possible nonlinearity of the considered composite hypothesis $H_0$, the nonlinearity of the forward model $G$, the ill-posedness of the inverse problem, the randomness of the observation points $(X_i)_{i=1}^{N}$, and the lack of knowledge about the distribution $\mu_X$ of the observation points.

\subsection{Literature review}

An overview of nonparametric goodness-of-fit testing in the direct linear case is given in \cite{IngSus:2003}. Separation rates are derived for the sequence space model, mostly simple hypotheses, and alternatives separated by power norms using optimal $\chi^2$-tests. Composite hypotheses that are considered take the form of shrinking power norm balls; see e.g. \cite[Section 8.5]{IngSus:2003}.
It is, moreover, mentioned that in the nonparametric case, the minimax separation rate is typically faster than the minimax rate for estimation, see \cite[Section 2.10]{IngSus:2003}, so that hypothesis testing may be possible even in settings in which estimation is unfeasible.
In \cite{Ing:1984} and \cite{Ing:1986}, goodness-of-fit testing is considered for composite hypotheses that are contained in a finite-dimensional linear subspace, again for direct linear problems.
In \cite{IngSapSus:2012}, goodness-of-fit testing is extended to linear inverse problems within the framework of the sequence space model, but only for simple hypotheses. In \cite{MarSap:2017}, this framework is further extended to the case of inverse problems with uncertainty in the eigenvalues of the forward model.
The approach that is used in \cite{IngSus:2003} and the subsequent works to prove minimax separation rates either relies heavily on the linearity of the considered problem and the availability of function-valued or infinite data, or on the existence of a spectral decomposition of the forward model. Thus, their approach cannot be easily generalised to the setting of nonlinear forward models and finite random observations.

In \cite{Ray:2013}, posterior contraction rates around a frequentist truth are shown for mildly and severely ill-posed linear inverse problems. Here, the general contraction result \cite[Theorem 2.1]{Ray:2013} for a general prior distribution is proven by proving distinguishability for plug-in tests based upon a truncated singular value decomposition (SVD) estimator in case of a simple hypothesis.
The proof follows that of Theorem 2.1 in \cite{GhoGhoVaa:2000}, with the difference that the existence of tests in the inverse problem case is not established using Theorem 7.1 in \cite{GhoGhoVaa:2000} but shown by direct computations using the consistency of the considered estimator.
The connection with posterior contraction is only made in the proof of Theorem 2.1 in \cite{GhoGhoVaa:2000}.
The results from \cite{GhoGhoVaa:2000} and \cite{Ray:2013} are formulated for alternatives that are separated by the hypothesis in a Hilbert space norm. However, their proofs are specific to linear problems and rely on the use of the truncated SVD estimator and the approximation properties of finite-dimensional projections. Hence, they cannot easily be transferred to other estimators, composite null hypotheses, or nonlinear problems.

In \cite[Theorem 8.13]{GhoVaa:2017}, it is shown that prior concentration and posterior contraction at an exponential speed imply the existence of tests. The proof is very brief but follows the proof of \cite[Theorem 6.22]{GhoVaa:2017}, which is presented in detail.
However, \cite[Theorem 8.13]{GhoVaa:2017} cannot be used to establish distinguishability because the partition of the space constructed therein is simply defined as the values of $\theta$ for which the size of the test is above or below a certain value. A more explicit and direct description of the sets is not given, and the sets do not correspond to an alternative that is separated by the hypothesis in some metric.

To the authors' knowledge, the minimax separation rate for goodness-of-fit testing in a nonparametric nonlinear indirect setting is unknown.
In \cite[Section 6.2.4]{GinNic:2016}, nonparametric goodness-of-fit testing by means of infimum tests is studied. Infimum tests are based upon tests designed for testing simple null hypotheses, and accept a composite null hypothesis $H_0$ as soon as one of the elements of $H_0$ is accepted by the underlying simple test.
In \cite[Section 6.2.1]{GinNic:2016}, some tests for simple hypotheses $\theta = \theta_0$ are constructed---specifically plug-in tests, $\chi^2$-, and $U$-statistic tests---and the minimax separation rates of these tests are studied.
A plug-in test rejects if the distance between an underlying estimator $\widehat{\theta}_N$ for $\theta$ and $\theta_0$ exceed a critical value, and accepts otherwise.
The $\chi^2$- and $U$-statistic tests, on the other hand, reject based upon the distance between a wavelet estimator and a wavelet approximation of $\theta_0$.
While the results for infimum plug-in tests such as in \cite[Proposition 6.2.13]{GinNic:2016} are not limited to the direct or linear case, they require certain uniform convergence properties of the underlying estimators, and these properties are only established for the direct case and a wavelet estimator, see \cite[Proposition 5.1.7]{GinNic:2016}.
In the text after \cite[Proposition 6.2.18]{GinNic:2016}, it is argued that $\chi^2$- and $U$-statistic tests do in fact achieve the minimax separation rate in the direct case if the hypothesis is bounded and consists of functions with sufficient Besov smoothness.
Since $\chi^2$-statistic tests are only defined for function-valued observations, they cannot directly be used in the random observation setting.
The results for $U$-statistic tests, in turn, are formulated specifically for density estimation, so that they do not easily generalise to an inverse problem setting.

In \cite{HarWahRasHui:2021}, the performance of Bayesian goodness-of-fit tests is studied numerically in inference for covariate-to-parameter mappings under the two-compartment model, a model widely used in pharmacokinetics.
In \cite{Nic:2023}, a theory of nonparametric Bayesian inference for nonlinear inverse problems with random observations is presented, including posterior contraction results; see also \cite{vanderVaavanZan:2000,GiorNic:2020,NicvandeGeerWang:2020,MonNicPat:2021,Kekk:2022}, for example.
However, the results in \cite{Nic:2023} are stated only for one arbitrary fixed true parameter $\theta_0$, and thus they cannot be directly combined with the theory from \cite{GinNic:2016} for infimum plug-in tests to obtain distinguishability.
In \cite{Lie:2024}, posterior contraction rates as well as convergence rates for the posterior mean are established for a class of inverse problems governed by ordinary differential equation initial value problems (ODE-IVPs). This class contains in particular inference for covariate-to-parameter mappings and the same two-compartment model as in \cite{HarWahRasHui:2021}.

\subsection{Contributions}

To the authors' knowledge, the minimax separation rate for nonparametric goodness-of-fit testing in a nonlinear inverse problems setting with random observations is unknown. Thus a natural goal is to bound the separation rate by establishing distinguishability for alternatives that are separated from the null hypothesis at a certain rate. In this work, we accomplish this goal using infimum plug-in tests based upon Bayesian estimators, under the assumption that the null hypothesis is contained in a suitably bounded subset of the parameter space.
Our main contributions are the following:

\begin{enumerate}
	\item We establish general quantitative non-asymptotic distinguishability results for infimum plug-in tests and composite null hypotheses under a uniform concentration inequality for the underlying estimator. These results yield distinguishability at the convergence rate of the underlying estimator.
	
	\item We prove uniform posterior contraction and uniform convergence of the posterior mean in $L^2_\zeta$, Sobolev, and supremum norms, for suitable probability measures $\zeta$ and for general nonlinear inverse problems with random observations and Gaussian priors. We then establish distinguishability at the posterior contraction rate for infimum plug-in tests based on the posterior mean or on maximum a posteriori (MAP) estimators. All these results are quantitative and non-asymptotic in the sense that explicit bounds for the involved probabilities are given.
	
	\item We apply our general results to the class of inverse problems governed by linear ODE-IVPs that was investigated in \cite{HarWahRasHui:2021,Lie:2024} to conclude posterior contraction and distinguishability in the $L^2_\zeta$, Sobolev, and supremum norm metrics for this class.
\end{enumerate}

\subsection{Structure}

This article is structured as follows.
First, we motivate the setting of later sections by presenting inference for covariate-to-parameter mappings and the two-compartment model in \cref{sec:motivation}.
In \cref{sec:dist_plug_in}, we prove general distinguishability results for infimum plug-in tests and discuss the concentration inequalities that are known for different estimators in an inverse problem set-up.
In \cref{sec:nonlin_ip}, we consider general nonlinear inverse problems with random observations, and show that uniform convergence of the posterior mean follows from posterior contraction. We establish $L^2_\zeta$, Sobolev, and supremum norm posterior contraction and use it to conclude distinguishability. Moreover, we discuss uniform convergence of MAP estimators and show that it implies distinguishability.
In \cref{sec:app_ode_ivp}, we then turn our attention towards inverse problems governed by ODE-IVPs, prove posterior contraction and distinguishability in this setting, and discuss the obtained rates.
We conclude in \cref{sec:conclusion}.
Longer proofs of results from \Cref{sec:nonlin_ip} are collected in \cref{sec_deferred_proofs}.
We discuss in \cref{sec:dist_simple_hyp} how distinguishability for simple null hypotheses can be derived from proofs in the literature.

\subsection{Notation}
\label{subsection_notation}

For the complement of a set $A$ in a set $B \supseteq A$ we use the notation $A^\complement \coloneqq B \setminus A$.
We denote the natural logarithm of $x > 0$ by $\log x$.
We write $a\vee b\coloneqq \max\{a,b\}$ for $a,b\in\R$.
The relation $a \lesssim b$ denotes an inequality $a \le Cb$ that holds up to a fixed constant $C > 0$.
We denote the substitution of $a$ with $b$ by $a\leftarrow b$. 
For an arbitrary normed vector space $(\calR,\norm{\cdot}_{\calR})$ and $M>0$, $B_{\calR}(M)\coloneqq\{ x\in\calR:\norm{x}_{\calR}\leq M\}$.
We denote the continuous embedding of a normed vector space $V$ into another normed vector space $W$ by $V \hookrightarrow W$.
For a measure $\mu$ on a bounded smooth domain $\X \subset \R^{d_x}$, a normed vector space $V$, and $p\in [1,\infty)$, let $L^p_\mu = L^p_\mu(\X) = L^p_\mu(\X,V)$ denote the space of $p$-integrable functions on $\X$ with respect to the measure $\mu$, with the corresponding norm
\[ \norm{f}_{L^p_\mu} = \left(\int_\X \norm{u(x)}_V^p \di\mu(x)\right)^\frac{1}{p}. \]
We define $L^\infty_\mu(\X,V)$ and $\norm{\cdot}_{L^\infty_\mu}$ in the usual way.
When $\mu$ is the Lebesgue measure on $\R^{d_x}$, we write $L^p = L^p(\X) = L^p(\X,V)$, for $p\in[1,\infty]$, and $\norm{f}_\infty\equiv \sup_{x\in\X}\norm{f(x)}_V$ for the supremum norm.
Let $H^\alpha = H^\alpha(\X) = H^\alpha(\X,V)$ denote the standard Sobolev space of order $\alpha \ge 0$ based upon the Lebesgue space $L^2(\X,V)$. Furthermore, let $H_\cpt^\alpha(\X,V)$ denote the completion of the compactly supported smooth functions $C_\cpt^\infty(\X,V)$ with respect to the $H^\alpha$ norm.
For more details regarding Sobolev spaces, see e.g. \cite[Appendix A.1]{Nic:2023}.

\section{Motivating example}
\label{sec:motivation}

\subsection{Inference for covariate-to-parameter mappings}

Consider the statistical model
\begin{equation}
	\label{eq:forward_model}
	y_i = \widetilde{G}(p_i,x_i) + \epsilon_i, \quad i \in \{1,\dots,N\},
\end{equation}
with i.i.d.~observations $(x_i,y_i)_{i=1}^N$ consisting of \emph{covariates} $x_i \in \X \subseteq \R^{d_x}$ and noisy \emph{observations} $y_i \in \Y \subseteq \R^{d_y}$, unobserved \emph{parameters} $p_i \in \calP \subseteq \R^{d_p}$, a nonlinear mapping $\widetilde{G}:\calP \times \X \to \R^{d_y}$ called the \emph{mechanistic model}, and \emph{noise} $\epsilon_i$.
The relationship between covariates $x$ and parameters $p$ in turn is described by a \emph{covariate-to-parameter mapping} (CPM) or \emph{covariate model} $\theta:\X \to \calP$ via
\begin{equation*}
	p_i = \theta(x_i), \quad i \in \{1,\dots,N\}.
\end{equation*}

We assume that the mechanistic model $\widetilde{G}$ is known. A typical example is the solution operator to a system of ODEs observed at different points in time $t_1,\dots,t_{d_y}$. In this case, each $p_i$ is a vector of unknown parameters of the ODE model. The CPM $\theta$, on the other hand, is assumed to be \emph{unknown} --- it is the quantity of primary interest. Inferring the CPM $\theta$, given observations $(x_i,y_i)_{i=1}^N$, poses an inverse problem.
We moreover assume that the observed covariates are realisations of a random variable $X$ with probability distribution $\mu_X$, i.e., we consider the case of \emph{random design}.

We restrict the admissible CPMs to a subset $\Theta$ of some space of functions from $\X$ to $\calP$.
We can then interpret the mechanistic model $\widetilde{G}$ as a forward model $G$ on $\Theta$ by setting $[G(\theta)](x) = \widetilde{G}(\theta(x),x)$ for every $\theta \in \Theta$ and $x \in \X$.

\subsection{Goodness-of-fit testing}

We aim to evaluate the hypothesis that the true CPM $\theta$ belongs to a given parametric class of functions
\[ \Theta_0 \coloneqq \{\theta_\tau: \tau \in \T\} \subset \Theta \]
with nuisance parameter $\tau \in \T \subseteq \R^{d_\tau}$. This corresponds to the testing problem
\begin{equation}
	\label{eq:testing_problem}
	H_0: \theta \in \Theta_0 \quad \text{vs.} \quad H_1: \theta \in \Theta \setminus \Theta_0.
\end{equation}
Since such an alternative $H_1$ is often asymptotically indistinguishable from the null hypothesis $H_0$, one usually considers an alternative $\theta \in \Theta_1$ with a set $\Theta_1 \subset \Theta$ that is \emph{separated away from} $\Theta_0$ in some metric. In addition, it may be necessary or helpful to further restrict $\Theta_1$ to be contained in a bounded set $S \subset \Theta$ that also contains $\Theta_0$.

Let $\widehat{\theta}$ be an estimator for $\theta$.
We will consider so-called \emph{infimum plug-in tests}
\[ \Psi = \mathbf{1}_{T\,>\,t} = \begin{cases}
	1 & \text{if}~T > t, \\
	0 & \text{if}~T \le t
\end{cases} \]
which reject if the distance
\begin{equation*}
	T = \inf_{\tau \in \T} d(\widehat{\theta},\theta_\tau)
	\label{eq:plug_in_stat}
\end{equation*}
between the estimator and the parametric class $\Theta_0$ in some metric $d$ on $\Theta$ exceeds a critical value $t > 0$.

\subsection{Covariates and mechanistic model}
\label{sec:cov_mech_mod}

We consider a specific example of a goodness-of-fit testing problem from pharmacokinetics that was studied in \cite{HarWahRasHui:2021}. In this example, the pharmacokinetics of a drug that has been administered to a patient are described using a \emph{two-compartment model}. In this model, the body is partitioned into the `central' and `peripheral' compartments that correspond to the bloodstream and tissue respectively. The time evolution of the concentrations $s_1(t)$ and $s_2(t)$ at time $t$ of a drug in the bloodstream and tissue are modelled by the ODE-IVP
\begin{align*}
	\frac{\di}{\di t} \begin{bmatrix} s_1(t,p) \\ s_2(t,p) \end{bmatrix}
	&= \begin{bmatrix} -e^{p_1 - p_2} - e^{p_3 - p_2} & e^{p_3 - p_2} \\
	e^{p_3 - p_4} & -e^{p_3 - p_4} \end{bmatrix}
	\begin{bmatrix} s_1(t,p) \\ s_2(t,p) \end{bmatrix}, \\
	\begin{bmatrix} s_1(0,p) \\ s_2(0,p) \end{bmatrix}
	&= \begin{bmatrix} D_0 w_0 e^{-p_2} \\ 0 \end{bmatrix}.
\end{align*}
Here, $D_0$ is the drug dosage administered at time $0$, and $w_0$ is a reference weight.
The parameters $p = (p_1, \dots, p_4) \in \R^4 =: \calP$ of the ODE are typically \emph{unobservable}. They can be interpreted as
\begin{equation*}
	p_1 = \ln \frac{\CL}{w}, \quad
	p_2 = \ln \frac{V_1}{w}, \quad
	p_3 = \ln \frac{Q}{w}, \quad
	p_4 = \ln \frac{V_2}{w},
\end{equation*}
where $V_1$ and $V_2$ are the volumes of each compartment, $\CL$ is the rate at which the drug is cleared or eliminated from the bloodstream, $Q$ describes inter-compartment drug transport, and $w$ is the weight of the patient.
The quantity 
\[ \widetilde{G}(p,x) = \left(\ln s_1(t_1), \dots, \ln s_1(t_{d_y})\right) \]
is the logarithmic drug concentration in the first compartment at $d_y$ fixed points in time $t_1, \dots, t_{d_y}$.
The overall dimension of each data vector $y_i$ given by \eqref{eq:forward_model} is thus $d_y$.

In \cite{HarWahRasHui:2021}, the covariates are the age $a$ and body weight $w$ of the patient. This results in the covariate vector $x = (a,w) \in (0,\infty)^2$.
The goal is to test whether the weight-normalised elimination clearance rate $\CL^* \coloneqq \CL w^{-\frac34}$ is a function of $a$ of a certain form.
For example, one may test whether $\CL^*$ is a saturable exponential function $\CL^*(a) = (1 - \tau_1 e^{-\tau_2 a})\CL_\text{max}^*$ that is constant in $w$, for some $\tau_1, \tau_2, CL_\text{max}^* \in \R$, and if the weight-normalised quantities $Q^* \coloneqq Qw^{-\frac34}$, $V_1^* \coloneqq V_1w^{-1}$, and $V_2^* \coloneqq V_2w^{-1}$ are constant in both $a$ and $w$.
In this case, the null hypothesis in \eqref{eq:testing_problem} takes the form $\Theta_0\coloneqq \{\theta_\tau: \tau \in \T\}$, where
\begin{equation}
	\label{eq:def_Theta0_sat_exp}
	\theta_\tau(a,w) = \left(\ln \frac{(1 - \tau_1 e^{-\tau_2 a}) \CL_\text{max}^*}{w^\frac14}, \ln V_1^*, \ln \frac{Q^*}{w^\frac14}, \ln V_2^*\right)
\end{equation}
and $\tau = (\tau_1, \tau_2, \CL_\text{max}^*, V_1^*, Q^*, V_2^*) \in \T \subset \R^6$.
For another example, one can replace the condition that $\CL^*$ is a saturable exponential function with the condition that $\CL^*$ is an affine linear function $\CL^*(a) = \tau_1 + \tau_2 a$ that is constant in $w$.
In this case, the null hypothesis $\Theta_0$ in \eqref{eq:testing_problem} is defined by functions $\theta_\tau$ of the form
\begin{equation}
	\label{eq:def_Theta0_aff_lin}
	\theta_\tau(a,w) = \left(\ln \frac{\tau_1 + \tau_2 a}{w^\frac14}, \ln V_1^*, \ln \frac{Q^*}{w^\frac14}, \ln V_2^*\right),
\end{equation}
for $\tau = (\tau_1, \tau_2, \CL_\text{max}^*, V_1^*, Q^*, V_2^*) \in \T \subset \R^6$.
Other examples of $\Theta_0$ are possible; see e.g. \cite[p. 569]{HarWahRasHui:2021}.

\section{Distinguishability for plug-in tests} 
\label{sec:dist_plug_in}

Let us investigate the plug-in testing approach described in sections 6.2.1 
and 6.2.4 of \cite{GinNic:2016} for simple and composite null hypotheses.
While this approach is not expected to yield a separation rate better than the minimax rate of estimation, it is still useful as a way to directly translate established convergence rates of estimators into bounds on separation rates. In addition, the resulting tests can be readily evaluated in practice.

We consider a statistical experiment with outcome $X \in \X$ whose distribution $P_\theta$ is indexed by an infinite-dimensional parameter $\theta$ from a set of admissible parameters $\Theta$.
For a given set $S \subset \Theta$ of candidate parameters, a null hypothesis $H_0 \subset S$, a metric $d$ on $S$, and $\delta > 0$, consider the alternative hypothesis
\begin{equation}
\label{alternative_hypothesis}
 H_1 = H_1(H_0,S,d,\delta) \coloneqq \left\{\theta \in S: d(\theta,H_0) \geq \delta\right\},
\end{equation}
where $d(\theta,H_0) \coloneqq \inf_{h \in H_0} d(\theta,h)$.
We denote the combined type $1$ and type $2$ error of a hypothesis test $\Psi$ for testing $H_0$ against $H_1$ by
\begin{equation}
 \label{gamma}
 \gamma(\Psi,H_0,S,d,\delta) \coloneqq \sup_{\theta \in H_0} P_\theta(\Psi = 1) + \sup_{\theta \in H_1(H_0,S,d,\delta)} P_\theta(\Psi = 0).
\end{equation}
Recall \cite[Definition 6.2.1]{GinNic:2016}: for $H_1$ and $\gamma$ as in \eqref{alternative_hypothesis} and \eqref{gamma}, a sequence $(\delta_N)_{N\in\N}$ of nonnegative real numbers is the \emph{minimax $d$-separation rate} for testing $H_0$ against $H_1$ if 
\begin{equation}
	\label{eq:def_dist}
	\lim_{N \to \infty} \inf_{\Psi} \gamma(\Psi,H_0,S,d,\delta_N) = 0
\end{equation}
and if for any nonnegative real sequence $(\delta_N^\prime)_{N\in\N}$ such that $\delta_N^\prime/\delta_N\to 0$,
\begin{equation*}
	\liminf_{N \to \infty} \inf_{\Psi} \gamma(\Psi,H_0,S,d,\delta_N^\prime) > 0
\end{equation*}
holds, where both infima are taken over all measurable functions $\Psi:\X \to \{0,1\}$.
The minimax $d$-separation rate is the minimal i.e. fastest rate at which a hypothesis and an alternative must be separated such that the combined type $1$ and type $2$ error $\gamma$ from \eqref{gamma} converges to $0$.

Now recall from \cite[pp. 64]{IngSus:2003} that for $H_1$ and $\gamma$ as in \eqref{alternative_hypothesis} and \eqref{gamma}, a sequence $(\delta_N)_{N\in\N}$ of nonnegative real numbers is a \emph{minimax $d$-distinguishability rate} for testing $H_0$ against $H_1$ if \eqref{eq:def_dist} holds, where the infimum is taken over all measurable functions $\Psi:\X \to \{0,1\}$.
By these definitions, a minimax distinguishability rate $(\delta_N)_{N\in\N}$ represents an upper bound for the minimax separation rate $(\epsilon_N)_{N\in\N}$, in the sense that $(\epsilon_N/\delta_N)_{N\in\N}$ is bounded from above.
\begin{definition}
 \label{definition_tests_that_distinguish}
 Let $H_1$ and $\gamma$ be as in \eqref{alternative_hypothesis} and \eqref{gamma}, and let $(\delta_N)_{N\in\N}$ be a sequence of nonnegative real numbers.
 A sequence $(\Psi_N)_{N\in\N}$ of tests is said to \emph{distinguish} between $H_0$ and $H_1$ at the rate $(\delta_N)_{N\in\N}$ if 
\begin{equation*}
  \lim_{N \to \infty} \gamma(\Psi_N,H_0,S,d,\delta_N) = 0.
\end{equation*}
\end{definition}
By the definition of a minimax $d$-distinguishability rate, the existence of tests $(\Psi_N)_{N\in\N}$ that distinguish between $H_0$ and $H_1$ at rate $(\delta_N)_{N\in\N}$ implies that a minimax $d$-distinguishability rate must be at least as fast as the rate $(\delta_N)_{N\in\N}$, because of the infimum in \eqref{eq:def_dist}.

Let $(\widehat{\theta}_N)_{N \in \N}$ be a sequence of measurable functions of $X$ that are estimators for some fixed $\theta$.
For different estimators, concentration inequalities are known in an inverse problems setting; see \Cref{subsection_concentration_inequalities_in_literature} below for some examples. A sequence $(d(\widehat{\theta}_N,\theta))_{N\in\N}$ is said to be \emph{stochastically bounded with respect to $(\delta_N)_{N\in\N}$} 
if for every $\epsilon > 0$, there exist $K > 0$ and $N_0 \in \N$ such that for all $N \ge N_0$,
\begin{equation}
	\label{eq:simple_conc_ineq}
	P_\theta\left(d(\widehat{\theta}_N,\theta) > K \delta_N\right) \le \epsilon.
\end{equation}
Stochastic boundedness is a weaker property than stochastic convergence to zero, because for every given $\epsilon$, only the \emph{existence} of a $K > 0$ for which \eqref{eq:simple_conc_ineq} holds is required. In particular, \eqref{eq:simple_conc_ineq} need not hold for every $K > 0$ and large enough $N$.
Stochastic boundedness with respect to $(\delta_N)_{N\in\N}$ is, moreover, equivalent to the convergence
\[ P_\theta\left(d(\widehat{\theta}_N,\theta) > \delta_N^\prime\right) \to 0 \]
for any rate $(\delta_N^\prime)_{N\in\N}$ that is slower than $(\delta_N)_{N\in\N}$ in the sense that $\delta_N^\prime/\delta_N \to \infty$.
If \eqref{eq:simple_conc_ineq} holds, then we also refer to the sequence $(\delta_N)_{N\in\N}$ as a \emph{convergence rate} for the sequence of estimators $(\widehat{\theta}_N)_{N \in \N}$.
Such convergence rates are shown for the Tikhonov--Phillips or penalised least squares estimator with general Hilbert space norm penalty \cite{AbhHelMue:2023} and with Sobolev norm penalty \cite{Siebel:2024}, as well as for the posterior mean or conditional mean estimator based upon a Gaussian prior; see \cite{Nic:2023} and the references therein for results concerning Bayesian nonlinear inverse problems.

\begin{remark}
\label{remark_growth_rate_K_wrt_epsilon}
A concentration inequality of the form \eqref{eq:simple_conc_ineq} can always be adapted to hold for all $N \in \N$, by choosing $K$ larger if necessary, possibly at the expense of a faster growth rate with respect to $\varepsilon$ of $K$. In many cases, it is of interest to determine the rate at which $K = K(\epsilon)$ in \eqref{eq:simple_conc_ineq} must grow as $\epsilon \to 0$, or whether $K$ can be chosen to be constant. 
\end{remark}

In order to bound the error of plug-in tests, one can assume that a concentration inequality holds \emph{uniformly} on a set $S \subseteq \Theta$ of candidate truths that contains both the null hypothesis $H_0$ and the alternatives $(H_{1,N})_{N \in \N}$. 
\begin{definition}[Uniform stochastic boundedness]
 \label{definition_uniform_stochastic_boundedness}
Let $S$ be nonempty and let $(\widehat{\theta}_N)_{N\in\N}$ be a sequence of estimators. The sequence $(d(\widehat{\theta}_N,\theta))_{N\in\N}$ is said to be \emph{uniformly stochastically bounded} with respect to $(\delta_N)_{N\in\N}$ over $\theta\in S$, if for every $\epsilon > 0$, there exist $K=K(\epsilon) > 0$ and $N_0=N_0(\epsilon) \in \N$ such that
\begin{equation}
	\label{eq_uniform_stochastic_boundedness}
	\sup_{\theta \in S} P_\theta\left(d(\widehat{\theta}_N,\theta) \geq K \delta_N\right) \le \epsilon,\quad \forall N \ge N_0.
\end{equation}
 The sequence $(d(\widehat{\theta}_N,\theta))_{N\in\N}$ is said to be \emph{strongly uniformly stochastically bounded} with respect to $(\delta_N)_{N\in\N}$ over $\theta\in S$, if there exists some probability decay rate $(\epsilon_N)_{N\in\N}$ and some $K>0$ and $N_0\in\N$ that do not depend on $(\delta_N)_{N\in\N}$ or $(\epsilon_N)_{N\in\N}$, such that
\begin{equation}
	\label{eq_strong_uniform_stochastic_boundedness}
	\sup_{\theta \in S} P_\theta\left(d(\widehat{\theta}_N,\theta) \geq  K \delta_N\right) \le \epsilon_N, \quad \forall N \ge N_0.
\end{equation}
 If \eqref{eq_uniform_stochastic_boundedness} or \eqref{eq_strong_uniform_stochastic_boundedness} holds, then we refer to $(\delta_N)_{N\in\N}$ as a \emph{uniform convergence rate} for $(\widehat{\theta}_N)_{N \in \N}$.
 \end{definition}
Uniform stochastic boundedness is a stronger assumption than stochastic boundedness for each $\theta \in S$, because $K$ in \eqref{eq_uniform_stochastic_boundedness} is not allowed to depend on $\theta$, whereas $K$ may depend on $\theta$ in \eqref{eq:simple_conc_ineq}.
One motivation for considering strong uniform stochastic boundedness is that it controls the growth rate of $K(\epsilon)$ with respect to $\epsilon$; cf. \Cref{remark_growth_rate_K_wrt_epsilon}. In \cref{corollary_uniform_L2_convergence_posterior_mean} and \cref{corollary_uniform_Sobolev_convergence_posterior_mean}, we shall see instances of strong uniform stochastic boundedness.

One sufficient condition for uniform stochastic boundedness is that the expected error of the estimators satisfy
\begin{equation}
	\label{eq:conv_rate_exp}
	\sup_{\theta \in S} \Exp_\theta [d(\widehat{\theta}_N,\theta)] \le \delta_N, \quad \forall N \in \N,
\end{equation}
see e.g. equation (6.25) and equation (6.52) in \cite{GinNic:2016}. If \eqref{eq:conv_rate_exp} holds, then by Markov's inequality,
\[ \sup_{\theta \in S} P_\theta\left(d(\widehat{\theta}_N,\theta) \geq  K\delta_N\right) \le \sup_{\theta \in S} \frac{1}{K\delta_N} \Exp_\theta  [d(\widehat{\theta}_N,\theta)] \le \frac{1}{K} = \epsilon \]
for $K=K(\epsilon) \coloneqq \epsilon^{-1}$ and all $N \in \N$. This gives an example in which $K(\epsilon)$ in \Cref{definition_uniform_stochastic_boundedness} is inversely proportional to $\epsilon$; see \Cref{remark_growth_rate_K_wrt_epsilon}. The condition \eqref{eq:conv_rate_exp} was used to prove bounds on the combined type $1$ and type $2$ errors in Proposition 6.2.2 and Proposition 6.2.13 in \cite{GinNic:2016}, for example.
In \Cref{errors_plug_in_general,corollary_distinguishability_at_slower_rates}, we shall prove bounds on the type $1$ and type $2$ errors under hypotheses that are weaker than \eqref{eq:conv_rate_exp}.

We shall now use concentration inequalities for estimator errors to obtain distinguishability results. For a set $H_0$ to be determined, a sequence of nonnegative numbers $(t_N)_{N\in\N}$, and a sequence of estimators $(\widehat{\theta}_N)_{N\in\N}$, define the \emph{infimum plug-in tests}
\begin{equation}
\label{eq_infimum_plug_in_tests}
  \Psi_N \coloneqq \mathbf{1}_{T_N\,>\,t_N},\quad   T_N \coloneqq d(\widehat{\theta}_N,H_0) = \inf_{h \in H_0} d(\widehat{\theta}_N,h),\qquad N\in\N .
\end{equation}
We shall use the following lemma for our distinguishability results.
\begin{lemma}[Bound on type $1$ and type $2$ errors]
	\label{errors_plug_in_general}
	Let $S\subseteq \Theta$ be nonempty, let $(t_N)_{N\in\N}$ be a sequence of nonnegative numbers, and let $(\widehat{\theta}_N)_{N\in\N}$ be a sequence of estimators.
	Then the type $1$ and type $2$ errors of the tests $(\Psi_N)_{N\in\N}$ from \eqref{eq_infimum_plug_in_tests} for testing
	\begin{equation}   
	\label{eq_null_hypothesis_alternative_hypothesis}
	 H_0 \subseteq S \quad \text{vs.} \quad 
	H_{1,N} \subseteq \left\{\theta \in S: d(\theta,H_0) \ge 2 t_N\right\}
	\end{equation}
	satisfy
    \begin{equation}
    \label{eq_bound_on_errors}
     	 \sup_{\theta \in H_0} P_\theta(\Psi_N = 1)\bigvee \sup_{\theta \in H_{1,N}} P_\theta(\Psi_N = 0) \le \sup_{\theta\in S} P_\theta\left(d(\widehat{\theta}_N,\theta)\geq t_N\right),\qquad \forall N\in\N.
    \end{equation}
\end{lemma}
\begin{proof}[Proof of \Cref{errors_plug_in_general}]
The proof uses the ideas from the proof of \cite[Proposition 6.2.2]{GinNic:2016}. Let $N\in\N$ be arbitrary.
By the definitions \eqref{eq_infimum_plug_in_tests} of $\Psi_N$ and $T_N$, we have $P_\theta (\Psi_N=1)=P_\theta(\inf_{h\in H_0}d(\widehat{\theta}_N,h)>t_N)$.
Now, if $\inf_{h\in H_0}d(\widehat{\theta}_N,h)>t_N$, then for every $\theta\in H_0$, $d(\widehat{\theta}_N,\theta)>t_N$.
By the preceding argument and by the containment $H_0\subseteq S$, the type $1$ error bound in \eqref{eq_bound_on_errors} follows from
\begin{align*}
 \sup_{\theta \in H_0} P_\theta \left(\Psi_N=1\right)=\sup_{\theta \in H_0}P_\theta\left(\inf_{h\in H_0}d(\widehat{\theta}_N,h)>t_N\right)\leq \sup_{\theta \in H_0}P_\theta \left( d(\widehat{\theta}_N,\theta)>t_N\right).
\end{align*}
Again by the definitions \eqref{eq_infimum_plug_in_tests} of $\Psi_N$ and $T_N$, we have $P_\theta (\Psi_N=0)=P_\theta(\inf_{h\in H_0}d(\widehat{\theta}_N,h)\leq t_N)$. By the triangle inequality, $d(\theta,h) \le d(\theta,\widehat{\theta}_N) + d(\widehat{\theta}_N,h)$ holds for all $\theta,h \in \Theta$.
By taking the infimum of both sides of this inequality over $h\in H_0$ and subtracting $d(\theta,\widehat{\theta}_N)$, we obtain 
\begin{equation*}
  \inf_{h \in H_0} d(\theta,h) - d(\theta,\widehat{\theta}_N) \leq \inf_{h \in H_0} d(\widehat{\theta}_N,h),\quad \forall \theta\in\Theta.
\end{equation*}
Thus, if $ \inf_{h \in H_0} d(\widehat{\theta}_N,h)\leq t_N$, then $\inf_{h \in H_0} d(\theta,h) - d(\theta,\widehat{\theta}_N) \leq t_N$.
Now, $\inf_{h \in H_0} d(\theta,h)\geq 2t_N$ holds for every $\theta \in H_{1,N}$, because $H_{1,N}\subseteq \{\theta \in S: d(\theta,H_0)\geq 2t_N\}$.
By the preceding observations, it follows that $2t_N- d(\theta,\widehat{\theta}_N) \leq t_N$ for every $\theta\in H_{1,N}$. Thus, the type $2$ error for $\theta \in H_{1,N}$ is bounded by
\begin{equation*}
	P_\theta(\Psi_N = 0) = P_\theta\left(\inf_{h \in H_0} d(\widehat{\theta}_N,h) \le t_N\right) \le P_\theta\left(2 t_N - d(\widehat{\theta}_N,\theta) \le t_N\right)= P_\theta\left(d(\widehat{\theta}_N,\theta) \ge t_N\right).
\end{equation*}
Taking the supremum over $\theta\in H_{1,N}$ completes the proof of \eqref{eq_bound_on_errors}.
\end{proof}
We now use \Cref{errors_plug_in_general} to show that uniform stochastic boundedness as defined by \eqref{eq_uniform_stochastic_boundedness} implies that the plug-in tests in \eqref{eq_infimum_plug_in_tests} distinguish between $H_0$ and $H_1$; cf. \Cref{definition_tests_that_distinguish}.
\begin{theorem}[Distinguishability under uniform stochastic boundedness]
\label{theorem_distinguishability_under_uniform_stochastic_boundedness}
Let $S\subseteq \Theta$ be nonempty. Let $(\widehat{\theta}_N)_{N\in\N}$ be a sequence of estimators and $(\delta_N)_{N\in\N}$ be a sequence of nonnegative numbers, such that $(d(\widehat{\theta}_N,\theta))_{N\in\N}$ is uniformly stochastically bounded with respect to $(\delta_N)_{N\in\N}$ over $\theta\in S$. 
Then for an arbitrary sequence $(\epsilon_N)_{N\in\N}$ of nonnegative numbers converging to zero and for the corresponding sequences $(K(\epsilon_N))_{N\in\N}$ and $(N_0(\epsilon_N))_{N\in\N}$ given by \eqref{eq_uniform_stochastic_boundedness}, the type $1$ and type $2$ errors of the tests $(\Psi_N)_{N\in\N}$ from \eqref{eq_infimum_plug_in_tests} with $t_N\leftarrow K(\epsilon_N)\delta_N$, $N\in\N$, for testing
	\begin{equation}   
	\label{eq_null_hypothesis_alternative_hypothesis_specific}
	 H_0 \subseteq S \quad \text{vs.} \quad 
	H_{1,N} \subseteq \left\{\theta \in S: d(\theta,H_0) \ge 2K(\epsilon_N)\delta_N\right\} 
	\end{equation}
	satisfy
\begin{equation}
\label{eq_type1_type2_errors_converge_to_zero}
  \sup_{\theta \in H_0} P_\theta(\Psi_N = 1) \bigvee \sup_{\theta \in H_{1,N}} P_\theta(\Psi_N = 0)\leq \epsilon_M,\qquad \forall N\geq N_0(\epsilon_M),\ M\in\N.
\end{equation}
In particular, $\lim_{N\to\infty}\gamma (\Psi_N,H_0,S,d,K(\epsilon_N)\delta_N)=0$, and the tests $(\Psi_N)_{N\in\N}$ distinguish between $H_0$ and $(H_{1,N})_{N\in\N}$ at the rate $(K(\epsilon_N)\delta_N)_{N\in\N}$.
\end{theorem}
Note that \Cref{theorem_distinguishability_under_uniform_stochastic_boundedness} does not impose any conditions on the growth rate with respect to $\epsilon$ of $K(\epsilon)$; cf. \Cref{remark_growth_rate_K_wrt_epsilon}.
\begin{proof}[Proof of \Cref{theorem_distinguishability_under_uniform_stochastic_boundedness}]
Let $(\delta_N)_{N\in\N}$, $(\epsilon_N)_{N\in\N}$, and $(K(\epsilon_N))_{N\in\N}$ be as in the statement of the theorem.
Since $K(\epsilon_N)\delta_N$ is nonnegative for every $N\in\N$, we may apply \Cref{errors_plug_in_general} with the substitution $t_N\leftarrow K(\epsilon_N)\delta_N$ in \eqref{eq_null_hypothesis_alternative_hypothesis} and \eqref{eq_bound_on_errors}, to conclude that the type $1$ and type $2$ errors of the tests $(\Psi_N)_{N\in\N}$ for the testing problem \eqref{eq_null_hypothesis_alternative_hypothesis_specific} satisfy
\begin{equation*}
 \sup_{\theta\in H_0}P_\theta(\Psi_N=1)\bigvee \sup_{\theta\in H_{1,N}}P_\theta(\Psi_N=0)\leq \sup_{\theta\in S}P_\theta\left(d(\widehat{\theta}_N,\theta)\geq K(\epsilon_N)\delta_N\right).
\end{equation*}
Let $M\in\N$ be arbitrary. By the hypothesis of uniform stochastic boundedness \eqref{eq_uniform_stochastic_boundedness} and by the choice of sequence $(\epsilon_N)_{N\in\N}$, there exists some $N_0(\epsilon_M)\in\N$ such that
\begin{equation*}
 \sup_{\theta \in S} P_\theta\left( d(\widehat{\theta}_N,\theta)\geq K(\epsilon_N)\delta_N\right)\leq \epsilon_M,\qquad \forall N\geq N_0(\epsilon_M).
\end{equation*}
Combining the two math displays above yields \eqref{eq_type1_type2_errors_converge_to_zero}. Next, note that by the definition \eqref{gamma} of the combined type $1$ and type $2$ error function $\gamma$ and by \eqref{eq_type1_type2_errors_converge_to_zero}, we have $\gamma(\Psi_N,H_0,S,d,K(\epsilon_N)\delta_N)\leq 2\epsilon_M$, for every $N\geq N_0(\epsilon_M)$. 
Letting $M\to\infty$ and using the hypothesis that $(\epsilon_N)_{N\in\N}$ converges to zero completes the proof.
\end{proof}

The following result shows that for \Cref{theorem_distinguishability_under_uniform_stochastic_boundedness}, one can ensure that $(K(\epsilon_N))_{N\in\N}$ is constant for any choice of $(\epsilon_N)_{N\in\N}$, provided that one uses an alternative rate $(\delta_N^\prime)_{N\in\N}$ that is slower than the rate $(\delta_N)_{N\in\N}$.

\begin{corollary}[Distinguishability at slower rates]
	\label{corollary_distinguishability_at_slower_rates}
	Let $S\subseteq \Theta$ be nonempty, let $(\widehat{\theta}_N)_{N\in\N}$ be a sequence of estimators, and let $(\delta_N)_{N\in\N}$ be a sequence of nonnegative numbers, such that $(d(\widehat{\theta}_N,\theta))_{N\in\N}$ is uniformly stochastically bounded with respect to $(\delta_N)_{N\in\N}$ over $\theta\in S$.
	If $(\delta_N^\prime)_{N\in\N}$ satisfies
	\begin{equation}
		\label{eq:cond_rate}
		\frac{\delta_N^\prime}{\delta_N} \to \infty \quad \text{as}~N \to \infty,
	\end{equation}
	then the combined type $1$ and type $2$ error of the tests $(\Psi_N)_{N\in\N}$ from \eqref{eq_infimum_plug_in_tests} with $t_N \leftarrow \delta_N^\prime$ for testing
	\begin{equation*}   
	  H_0 \subseteq S \quad \text{vs.} \quad 
	H_{1,N}^\prime \subseteq \left\{\theta \in S: d(\theta,H_0) \ge 2\delta_N^\prime\right\}
	\end{equation*}
	satisfies $\lim_{N\to\infty}\gamma(\Psi_N,H_0,S,D,\delta_N^\prime)= 0$. In particular, the tests $(\Psi_N)_{N\in\N}$ distinguish between $H_0$ and $(H_{1,N})_{N\in\N}$ at the rate $(\delta_N^\prime)_{N\in\N}$.
\end{corollary}
\begin{remark}
\label{remark_limit_case}
 \Cref{corollary_distinguishability_at_slower_rates} does not include the limiting case of alternatives $(H_{1,N})_{N\in\N}$ that are separated from the null hypothesis $H_0$ at precisely the rate $(\delta_N)_{N\in\N}$.  This is because of the constraint \eqref{eq:cond_rate} on the sequence $(\delta_N^\prime)_{N\in\N}$.
\end{remark}

\begin{proof}[Proof of \Cref{corollary_distinguishability_at_slower_rates}]
Let $\epsilon>0$ be arbitrary. By the hypothesis of uniform stochastic boundedness, there exist $K(\epsilon)>0$ and $N_0(\epsilon)\in\N$ such that $\sup_{\theta\in S}P_\theta(d(\widehat{\theta}_N,\theta)\geq K(\epsilon)\delta_N)\leq \epsilon$ for every $N\geq N_0(\epsilon)$.
By the hypothesis \eqref{eq:cond_rate}, there exists $N_1(K(\epsilon))\in\N$ such that $\delta_N^\prime \geq K(\epsilon) \delta_N$ for every $N\geq N_1(K(\epsilon))$. 
Now
\begin{equation*}
 \sup_{\theta \in S}P_\theta\left(d(\widehat{\theta}_N,\theta)\geq \delta_N^\prime\right)\leq \sup_{\theta\in S}P_\theta\left(d(\widehat{\theta}_N,\theta)\geq K(\epsilon)\delta_N\right) \leq  \epsilon,\quad \forall N\geq N_0(\epsilon)\vee N_1(K(\epsilon)).
\end{equation*}
Hence, $(d(\widehat{\theta}_N,\theta))_{N\in\N}$ satisfies the condition \eqref{eq_uniform_stochastic_boundedness} for being uniformly stochastically bounded with respect to the slower rate $(\delta_N^\prime)_{N\in\N}$, with $N_0(\epsilon)\leftarrow N_0(\epsilon)\vee N_1(K(\epsilon))$ and constant $K(\epsilon)\leftarrow 1$ for every $\epsilon$ in \eqref{eq_uniform_stochastic_boundedness}.
Thus, the hypotheses of \Cref{theorem_distinguishability_under_uniform_stochastic_boundedness} are satisfied.
By replacing $K(\epsilon_N)$ and $\delta_N$ in the statement of  \Cref{theorem_distinguishability_under_uniform_stochastic_boundedness} with $1$ and $\delta_N^\prime$ respectively, we obtain $\lim_{N\to\infty}\gamma(\Psi_N,H_0,S,D,\delta_N^\prime)= 0$.
\end{proof}
Recall the criterion \eqref{eq_strong_uniform_stochastic_boundedness} for strong uniform stochastic boundedness from \Cref{definition_uniform_stochastic_boundedness}. If we strengthen the hypothesis of uniform stochastic boundness in \Cref{theorem_distinguishability_under_uniform_stochastic_boundedness} to strong uniform stochastic boundedness, then we obtain the following result.
\begin{corollary}[Distinguishability under strong uniform stochastic boundedness]
 \label{corollary_distinguishability_under_strong_uniform_stochastic_boundedness}
 Let $S\subseteq \Theta$ be nonempty. Let $(\widehat{\theta}_N)_{N\in\N}$ be a sequence of estimators and $(\delta_N)_{N\in\N}$ be a sequence of nonnegative numbers, such that $(d(\widehat{\theta}_N,\theta))_{N\in\N}$ is strongly uniformly stochastically bounded with respect to $(\delta_N)_{N\in\N}$ over $\theta\in S$, with probability decay rate $(\epsilon_N)_{N\in\N}$ and scalars $K>0$ and $N_0\in\N$ that do not depend on $(\delta_N)_{N\in\N}$ or $(\epsilon_N)_{N\in\N}$ as in \eqref{eq_strong_uniform_stochastic_boundedness}. Then the type $1$ and type $2$ errors of the tests $(\Psi_N)_{N\in\N}$ from \eqref{eq_infimum_plug_in_tests} with $t_N\leftarrow K\delta_N$, $N\in\N$, for testing
	\begin{equation}   
	\label{eq_null_hypothesis_alternative_hypothesis_specific_strong_uniform_stochastic_boundedness}
	 H_0 \subseteq S \quad \text{vs.} \quad 
	H_{1,N} \subseteq \left\{\theta \in S: d(\theta,H_0) \ge 2K\delta_N\right\} 
	\end{equation}
	satisfy
\begin{equation}
\label{eq_type1_type2_errors_bounds_strong_uniform_stochastic_boundedness}
  \sup_{\theta \in H_0} P_\theta(\Psi_N = 1) \bigvee \sup_{\theta \in H_{1,N}} P_\theta(\Psi_N = 0)\leq \epsilon_N,\qquad \forall N\geq N_0.
\end{equation}
In particular, $\gamma (\Psi_N,H_0,S,d,K\delta_N)\leq 2\epsilon_N$ for every $N\geq N_0$, and the tests $(\Psi_N)_{N\in\N}$ distinguish between $H_0$ and $(H_{1,N})_{N\in\N}$ at the rate $(K\delta_N)_{N\in\N}$.
\end{corollary}
\begin{proof}[Proof of \Cref{corollary_distinguishability_under_strong_uniform_stochastic_boundedness}]
By \Cref{definition_uniform_stochastic_boundedness}, strong uniform stochastic boundedness implies uniform stochastic boundedness.
Thus, we may apply \Cref{theorem_distinguishability_under_uniform_stochastic_boundedness} with the sequences $(K(\epsilon_N))_{N\in\N}$ and $(N_0(\epsilon_N))_{N\in\N}$ in the statement of \Cref{theorem_distinguishability_under_uniform_stochastic_boundedness} replaced with the constants $K$ and $N_0$ given in the hypothesis of strong uniform stochastic boundedness. Given this replacement, the condition \eqref{eq_null_hypothesis_alternative_hypothesis_specific} in the statement of \Cref{theorem_distinguishability_under_uniform_stochastic_boundedness} on the alternative hypothesis $H_{1,N}$ yields the analogous condition \eqref{eq_null_hypothesis_alternative_hypothesis_specific_strong_uniform_stochastic_boundedness} in the statement of \Cref{corollary_distinguishability_under_strong_uniform_stochastic_boundedness}.
The type $1$ and type $2$ error bounds \eqref{eq_type1_type2_errors_converge_to_zero} in the statement of \Cref{theorem_distinguishability_under_uniform_stochastic_boundedness} yield the analogous bounds \eqref{eq_type1_type2_errors_bounds_strong_uniform_stochastic_boundedness} in the statement of \Cref{corollary_distinguishability_under_strong_uniform_stochastic_boundedness}, after replacing $\epsilon_M$ and the conditions $N\geq N_0(\epsilon_M)$, $M\in\N$ with $\epsilon_N$ and $N\geq N_0$, $N\in\N$ respectively.
The final statement of \Cref{corollary_distinguishability_under_strong_uniform_stochastic_boundedness} follows from \eqref{eq_type1_type2_errors_bounds_strong_uniform_stochastic_boundedness} and the definition \eqref{gamma} of the combined type $1$ and type $2$ error $\gamma$, and from \Cref{definition_tests_that_distinguish} of tests that distinguish between a null hypothesis and a family of alternatives.
\end{proof}
In \Cref{sec:nonlin_ip}, we shall use \Cref{corollary_distinguishability_under_strong_uniform_stochastic_boundedness} to prove distinguishability for posterior means in \Cref{corollary_L2_distinguishability_posterior_mean} and \Cref{corollary_Linfty_distinguishability_posterior_mean}, and also for MAP estimators in \Cref{corollary_distinguishability_MAP}.

\subsection{Concentration inequalities in literature}
\label{subsection_concentration_inequalities_in_literature}

Concentration inequalities are known for several estimators in a nonlinear inverse problems setting with random observations.

In \cite[Theorem 2.3.2]{Nic:2023}, stochastic boundedness of $(d(\widehat{\theta}_N,\theta))_{N\in\N}$ for the sequence of posterior means $\widehat{\theta}_N = \Exp^{\Pi_N}[\theta|(Y_k,X_k)_{k=1}^N]$ is shown, with $d$ being the metric induced by the $L^2_{\mu_X}$ and with the rate 
\begin{equation*}
	\delta_N = N^{-\frac{\alpha+\kappa}{2(\alpha+\kappa) + d}},\quad N\in\N
\end{equation*}
where $\alpha>0$ denotes some smoothness parameter and $d$ denotes the dimension of a common domain for a set of functions, for the special case where the set $S$ is a singleton set. An examination of the proof shows that the constant $K$ in the concentration inequality depends on $\theta$. In \Cref{sec:nonlin_ip}, we will prove that in fact strong uniform stochastic boundedness holds, for a suitably bounded set $S$ of admissible parameters.

In \cite{AbhHelMue:2023}, the penalised least squares estimators $(\widehat{\theta}_N)_{N\in\N}$ with squared norm penalty are studied, for general nonlinear inverse problems with certain properties.
Corollary 4.2 in \cite{AbhHelMue:2023} asserts stochastic boundedness of $(d(\widehat{\theta}_N,\theta))_{N\in\N}$ with $d$ being the metric induced by the norm of some separable Hilbert space $\calH$ of mappings from $\X$ to $V$ that is continuously embedded in $L_{\lambda}^2(\X,V)$, and with the rate 
\[ \delta_N = N^{-\frac{\alpha}{2\alpha + b + 1}},\quad N\in\N \]
where $\alpha>0$ again denotes some smoothness parameter, and $0<b < 1$ as given in \cite[Assumption 5]{AbhHelMue:2023} describes the degree of ill-posedness with respect to the design measure via the decay of the eigenvalues of a covariance operator. 
A concentration inequality is stated for $K \propto \ln \frac{4}{\epsilon}$, but it is unclear if $K$ depends on $\theta$, and if $K$ remains bounded under the source condition \cite[Assumption 11]{AbhHelMue:2023}.

Corollary 2.10 in \cite{Siebel:2024} establishes stochastic boundedness of $(d(\widehat{\theta}_N,\theta))_{N\in\N}$ for penalised least squares estimators $(\widehat{\theta}_N)_{N\in\N}$ with respect to
\[ \delta_N = N^{-\frac{\alpha + \kappa}{2(\alpha + \kappa) + d_z}},\quad N\in\N \]
under $P_{\theta}$, for the metric induced by the standard $L^2$ norm and for the squared $H^\alpha$ norm penalty, where $\alpha>0$ denotes some smoothness parameter, $\kappa$ describes the regularity of the forward mapping according to \cite[Condition 2.3]{Siebel:2024}, and $d_z$ denotes the dimension of a common domain for a family of functions.
The constant $K$ in the concentration inequality explicitly depends on $\norm{\theta}_{H^\alpha}$, but can be chosen identically for all $\theta \in S$ if $S$ is bounded in $H^\alpha$.
In \cref{sec:map_estimator}, we will derive upper bounds for the separation rate from this result.

\section{Nonlinear inverse problems with random observations}
\label{sec:nonlin_ip}

In this section, we consider goodness-of-fit testing for nonlinear inverse problems with random observations, using the same setting as in \cite[Section 1.2]{Nic:2023}. We now describe this setting.
Let $(\calZ,\zeta)$ and $(\X,\lambda)$ be probability spaces, where $\calZ \subseteq \R^{d_z}$ and $\X \subseteq \R^{d_x}$. Furthermore, let $V$ and $W$ be normed spaces of finite dimension $d_V$ and $d_W$.
We consider parameters $\theta$ from a Borel-measurable subspace $\Theta \subseteq L^2_\zeta(\calZ,W)$.
We assume that we are given indirect noisy observations
\begin{equation}
	Y_i = [G(\theta_0)](X_i) + \epsilon_i, \quad i \in \{1, \dots, N\}, \label{eq:inverse_problem}
\end{equation}
at random points $X_i$ of a function $G(\theta_0)$ generated by a true parameter $\theta_0 \in \Theta$, where $X_i \sim \lambda$ are i.i.d. $\X$-valued random variables, $\epsilon_i \sim \Normal(0,\sigma^2\Id_V)$ for some $\sigma>0$ are i.i.d. $V$-valued random variables, and $G$ is a measurable nonlinear mapping from $\Theta$ to $L^2_\lambda(\X,V)$.
In addition, we shall make use of a normed space $(\calR,\norm{\cdot}_{\calR})$, where $\calR$ is a subspace of $\Theta$; we shall refer to a $\calR$ as the `regularisation space'. We will specify $\calR$ further in the results below.
Note that although the covariance operator of the noise $\epsilon_i$ is assumed to be the identity on $V$, nonisotropy can be modelled by a suitable choice of the norm on $V$.

We first state our key assumptions regarding the forward model.
The first assumption expresses the boundedness and Lipschitz continuity of the forward model $G$ when it is restricted to bounded subsets of the regularisation space $\calR$. 
\begin{assumption}[{\cite[Condition 2.1.1]{Nic:2023}}]
 \label{assumption_Nickl_Condition_2_1_1}
 Let $(\X,\lambda)$, $(\calZ,\zeta)$, $V$, $W$, $\Theta$, $G$, and $\calR$ be as described above.
There exists $\kappa\geq 0$ such that for every $M>0$ there exist finite scalars $U=U(M)\geq 1$ and $L=L(M)>0$, such that 
 \begin{align}
  \label{eq_Nickl_2_3}
  \sup_{\theta\,\in\,\Theta\,\cap\,B_{\calR}(M)}\norm{G(\theta)}_\infty\leq &U
\\
 \label{eq_Nickl_2_4}
 \norm{G(\theta_1)-G(\theta_2)}_{L^2_\lambda(\X,V)}\leq& L\norm{\theta_1-\theta_2}_{(H^\kappa(\calZ))^\ast},\quad \forall\theta_1,\theta_2\in \Theta\cap B_{\calR}(M).
\end{align}
\end{assumption}

The index $\kappa$ in \eqref{eq_Nickl_2_4} describes how smoothing the operator $G$ is and thus quantifies the ill-posedness of the inverse problem defined by $G$.
Note that $H^\kappa$ refers to the standard Sobolev space based upon the Lebesgue measure, so that if $\kappa = 0$, then the norm of the space $(H^0(\calZ))^* = L^2(\calZ)$ on the right hand side of \eqref{eq_Nickl_2_4} is the \emph{standard} i.e. Lebesgue $L^2$ norm; cf. \Cref{subsection_notation}.
Note that in comparison with \cite{Siebel:2024}, where similar but global assumptions on the forward mapping $G$ are used to prove convergence of MAP estimators (cf. \cref{assumptions_G_global}), \cref{assumption_Nickl_Condition_2_1_1} only requires local boundedness and local Lipschitz continuity.

Next, we consider a so-called stability estimate condition on the forward model.
\begin{assumption}[Stability estimate]
\label{assumption_local_Lipschitz_continuity_of_restricted_inverse_forward_map}
	Let $(\X,\lambda)$, $(\calZ,\zeta)$, $V$, $W$, $G$, and $\calR$ be as in \Cref{assumption_Nickl_Condition_2_1_1}.
	There exists some $\eta > 0$ such that for every $M > 0$, there exist constants $L^\prime, \delta_0 > 0$ such that
	\begin{equation}
		\label{eq_stability_estimate}
		\sup \left\{\norm{\theta_1 - \theta_2}_{L_\zeta^2(\calZ,W)}: \theta_1, \theta_2 \in B_\calR(M), \norm{G(\theta_1) - G(\theta_2)}_{L_\lambda^2(\X,V)} \le \delta\right\} \le L^\prime\delta^\eta
	\end{equation}
	holds for all $\delta \in (0,\delta_0]$.
\end{assumption}
An example of \Cref{assumption_local_Lipschitz_continuity_of_restricted_inverse_forward_map} with $\eta=1$ appears in \cite[Proposition 3.3]{Lie:2024}. 
\Cref{assumption_local_Lipschitz_continuity_of_restricted_inverse_forward_map} implies \cite[Condition 2.1.4]{Nic:2023}. 
An important distinction between \Cref{assumption_local_Lipschitz_continuity_of_restricted_inverse_forward_map} and \cite[Condition 2.1.4]{Nic:2023} is that \cite[Condition 2.1.4]{Nic:2023} depends on a given $\theta_0$, whereas \Cref{assumption_local_Lipschitz_continuity_of_restricted_inverse_forward_map} does not. 
We shall exploit this distinction below, in \Cref{theorem_uniform_posterior_contraction_L2}.
Note that the stability condition \eqref{eq_stability_estimate} is formulated using the $L^2_\zeta$ norm in the parameter space and not the standard $L^2$ norm.

\subsection{Uniform convergence of posterior mean under posterior contraction}
\label{sec:uniform_conv_post_mean}

As we have seen in \cref{sec:dist_plug_in}, uniform convergence of an estimator implies distinguishability. 
In this section, we will establish uniform convergence for the posterior mean as a consequence of posterior contraction.
To this end, we state and prove \Cref{theorem_uniform_posterior_contraction_implies_uniform_convergence_posterior_mean}, which shows that uniform posterior contraction and a continuous embedding of the regularisation space $\calR$ into a separable Banach space $\mathbb{B}$ implies uniform convergence of the posterior mean in the $\mathbb{B}$ norm.

We denote by $P^N_\theta=\bigotimes^N_{i=1} P_\theta$ the joint law of $(Y_i,X_i)_{i=1}^{N}$, where $P_\theta$ denotes the law of a generic pair $(Y_i,X_i)$ that satisfies \eqref{eq:inverse_problem} with $\theta_0\leftarrow\theta$, for $\theta\in\Theta$.
Let $\mu$ denote the product of the Lebesgue measure on $V$ with $\lambda$. Then $P_\theta\ll \mu$, and
\begin{equation}
\label{eq_Nickl_1_10}
 p_\theta(y,x)\coloneqq \frac{\di P_\theta}{\di \mu}(y,x)=\frac{1}{(2\pi\sigma^2)^{{d_V}/2}}\exp\left(-\frac{1}{2\sigma^2}\norm{y-G(\theta)(x)}_V^2\right);
\end{equation}
cf. \cite[equation (1.10)]{Nic:2023}.
This yields the definition \cite[equation (1.13)]{Nic:2023} of the log-likelihood function of the data $D_N=(Y_i,X_i)_{i=1}^{N}$ under the additive independent Gaussian observation noise model \eqref{eq:inverse_problem}:
\begin{equation}
 \label{eq_Nickl_1_13}
 \Theta\ni\theta\mapsto\ell_N(\theta)\coloneqq -\frac{1}{2\sigma^2}\sum_{i=1}^{N}\norm{Y_i-G(\theta)(X_i)}^2_V.
\end{equation}
For a prior $\Pi$ on $\Theta$, the corresponding posterior given the data $D_N$ is a probability measure on $\Theta$ that satisfies
\begin{equation}
 \label{eq_Nickl_1_12}
 \mathrm{d}\Pi(\theta\vert D_N)=\frac{e^{\ell_N}(\theta)\mathrm{d}\Pi(\theta)}{\int_{\Theta} e^{\ell_N(\theta')}\mathrm{d}\Pi(\theta')},\quad \theta\in\Theta.
\end{equation}

We shall use centred Gaussian priors on $\theta$ that satisfy the following conditions.

\begin{assumption}[{\cite[Condition 2.2.1]{Nic:2023}}]
 \label{assumption_Nickl_Condition_2_2_1}
 Let $(\X,\lambda)$, $(\calZ,\zeta)$ and $G$ be as in \Cref{assumption_Nickl_Condition_2_1_1}.
 Let $\Pi^\prime $ be a centred Gaussian Borel probability measure on the linear space $\Theta\subseteq L^2_\zeta(\calZ,W)$ with reproducing kernel Hilbert space (RKHS) $\calH$. Suppose further that $\Pi^\prime (\calR)=1$ for some separable normed linear subspace $(\calR,\norm{\cdot}_{\calR})$ of $\Theta$.
\end{assumption}
Recall that if $\mathbb{B}$ is a separable Banach space and $X$ is a centred $\mathbb{B}$-valued Gaussian random variable, then the RKHS $\calH$ of $X$ is a dense measurable subspace of $\mathbb{B}$ that is continuously embedded in $\mathbb{B}$; see e.g. \cite[Proposition 2.6.9]{GinNic:2016} or the discussion after \cite[equation (I.3)]{GhoVaa:2017}.
Therefore, \cref{assumption_Nickl_Condition_2_2_1} implies that if $\calR$ is a separable Banach space, then $\calH \subseteq \calR$.
The following definition of a family of rescaled Gaussian priors is the same as that given by \cite[equation (2.18)]{Nic:2023}.
\begin{definition}
\label{definition_rescaled_Gaussian_priors}
Let $\calZ$ be a bounded, smooth domain in $\mathbb{R}^{d_z}$, $\kappa$ be as given in \Cref{assumption_Nickl_Condition_2_1_1}, and let $\Pi^\prime $ satisfy \Cref{assumption_Nickl_Condition_2_2_1}. For a chosen $\alpha>0$, define $\Pi_N$ for every $N\in\mathbb{N}$ to be the law of $\theta\coloneqq N^{-d_z/(4\alpha+4\kappa+2d_z)}\theta'$, where $\theta'\sim\Pi^\prime $.
\end{definition}
Given $\alpha$ from \Cref{definition_rescaled_Gaussian_priors} and $\kappa$ from \Cref{assumption_Nickl_Condition_2_1_1}, we shall use the following sequence from \cite[equation (2.19)]{Nic:2023} to determine our posterior contraction rate:
\begin{equation}
 \label{eq_Nickl_2_19}
 \delta_N\coloneqq N^{-\frac{\alpha+\kappa}{2\alpha+2\kappa+d_z}}.
\end{equation}
For every $M>0$, for $\kappa$ and $\calR$ from \Cref{assumption_Nickl_Condition_2_1_1} and \Cref{assumption_Nickl_Condition_2_2_1}, and for an arbitrary sequence $(\delta_N)_{N\in\N}$ converging to zero, we shall use the \emph{regularisation sets} from \cite[equation (2.20)]{Nic:2023}:
\begin{equation}
 \label{eq_Nickl_2_20}
 \Theta_N=\Theta_N(M)\coloneqq \{\theta\in\calR:\theta=\theta_1+\theta_2,\norm{\theta_1}_{(H^\kappa)^\ast}\leq M\delta_N,\ \norm{\theta_2}_{\calH}\leq M,\ \norm{\theta}_{\calR}\leq M\}.
\end{equation}
These sets can also be expressed using a Minkowski sum as
\[ \Theta_N = \left(B_{(H^\kappa)^*}(M\delta_N) + B_\calH(M)\right) \cap B_\calR(M). \]

For a subset $\mathcal{M}'\subseteq \mathcal{M}$ of a semimetric space $(\mathcal{M},d)$ and $\epsilon>0$, we denote the corresponding covering number by $N(\mathcal{M}',d,\epsilon)$.
The condition \eqref{eq_Nickl_1_26} in \Cref{theorem_Nickl_2_2_2} below is expressed in terms of the semimetric $d_G$, defined by 
\begin{equation}
 \label{semimetric_dG}
 (\theta_1,\theta_2)\mapsto d_G(\theta_1,\theta_2)\coloneqq \norm{G(\theta_1)-G(\theta_2)}_{L^2_\lambda(\X,V)}.
\end{equation}
For $\delta_N$ in \eqref{eq_Nickl_2_19}, $\theta_0\in\Theta$ and $\mathcal{U}>0$, define
 \begin{equation}
 \label{eq_Nickl_1_23}
\mathcal{B}_N(\theta_0,\mathcal{U})\coloneqq \{\theta\in\Theta:d_G(\theta,\theta_0)\leq \delta_N,\norm{G(\theta)}_\infty\leq \mathcal{U}\},
 \end{equation}
and observe that $\mathcal{B}_N(\theta_0,\mathcal{U})$ is increasing in the second argument, i.e. if $\mathcal{U}_1\leq \mathcal{U}_2$ then for every $\theta_0$, $\mathcal{B}_N(\theta_0,\mathcal{U}_1)\subseteq \mathcal{B}_N(\theta_0,\mathcal{U}_2)$. 
We also define the following function.
\begin{align}
 \label{eq_C_of_U_function}
 (0,\infty)\ni u\mapsto C_u\coloneqq &\frac{1-e^{-u^2/2}}{2u^2}
\end{align}
We now state a forward posterior contraction result that will form the basis for all the subsequent posterior contraction results.
Below, $\Pi_N(\cdot\vert D_N)$ denotes the posterior corresponding to \eqref{eq_Nickl_1_12} with the substitution $\Pi\leftarrow\Pi_N$ for $\Pi_N$ in \cref{definition_rescaled_Gaussian_priors}.
\begin{restatable}{theorem}{forwardPosteriorContraction}
 \label{theorem_Nickl_2_2_2}
 Suppose \Cref{assumption_Nickl_Condition_2_1_1} holds for the forward map $G$, $\kappa\geq 0$, $M'\mapsto U(M')<\infty$, regularisation space $\calR$, and bounded smooth domain $\calZ\subset \mathbb{R}^{d_z}$.
Let the base prior $\Pi^\prime $ satisfy \Cref{assumption_Nickl_Condition_2_2_1} for the given $\calR$, and let $\alpha>0$ be such that the RKHS $\calH$ of $\Pi^\prime $ satisfies the continuous embedding
\begin{equation}
 \calH\hookrightarrow H^\alpha_\cpt(\calZ)\text{ if }\kappa\geq 1/2,\qquad\text{or}\qquad \calH\hookrightarrow H^\alpha (\calZ)\text{ if }\kappa<1/2. \label{eq:RKHS_Sob_emb}
\end{equation}
Let $(\Pi_N)_{N\in\mathbb{N}}$ and $(\delta_N)_{N\in\mathbb{N}}$ be as in \Cref{definition_rescaled_Gaussian_priors} and \eqref{eq_Nickl_2_19} respectively.
Let $\theta_0\in\calH\cap\calR$. 
Then there exists $0<M<\infty$ sufficiently large, such that $\theta_0\in B_{\calR}(M)$, and there exists $A=A(U)$ for $U=U(M)$ defined by \eqref{eq_Nickl_2_3} and $M$, such that for all sufficiently large $N$,
\begin{equation}
 \label{eq_Nickl_1_24}
 \Pi_N(\mathcal{B}_N(\theta_0,U))\geq e^{-AN\delta_N^2}.
\end{equation}
Furthermore, by increasing $M$ further if necessary, there exists some $B>A+2$ such that the sets $(\Theta_N(M))_{N\in\mathbb{N}}$ defined in \eqref{eq_Nickl_2_20} satisfy 
\begin{subequations}
\begin{align}
 \label{eq_Nickl_1_25}
 \Pi_N(\Theta^\complement_N(M)) &\leq e^{-BN\delta_N^2}
 \\
 \label{eq_Nickl_1_26}
 \log N(\Theta_N(M),d_G,\overline{m}\delta_N) &\leq N\delta_N^2
\end{align}
\end{subequations}
for all sufficiently large $N\in\N$, where $\overline{m}=\overline{m}(M,L)>0$ in \eqref{eq_Nickl_1_26} is large enough and for $L=L(M)$ defined in condition \eqref{eq_Nickl_2_4} of \Cref{assumption_Nickl_Condition_2_1_1}. 
In particular, for any $0<b<B-(A+2)$, and for any $0<c<1$, we can choose $m=m(A,b,\overline{m})$ sufficiently large and independent of $c$, such that for $U$ as in \eqref{eq_Nickl_1_24}, for $C_U$ given by \eqref{eq_C_of_U_function} with $u\leftarrow U$, for some $N^*(c)\in\N$, and for every $N\geq N^*(c)$,
\begin{equation}
\label{eq_Nickl_1_27}
 P^N_{\theta_0}\left(\Pi_N(\{\theta\in\Theta_N: d_G(\theta,\theta_0)\leq mC_{U}^{-1/2}\delta_N \}\vert D_N)\leq 1-e^{-bN\delta_N^2}\right)\leq (1+c)\frac{2(U^2+1)}{N\delta_N^2}.
\end{equation}
\end{restatable}
  \Cref{theorem_Nickl_2_2_2} is an extended version of \cite[Theorem 2.2.2]{Nic:2023}.
  The conclusion \eqref{eq_Nickl_1_27} is a stronger version of \cite[equation (1.27)]{Nic:2023}, which only states that 
  \begin{equation*}
   \lim_{N\to\infty} P^N_{\theta_0}\left(\Pi_N(\{\theta\in\Theta_N: d_G(\theta,\theta_0)\leq m\delta_N C_{v}^{-1}\}\vert D_N)\leq 1-e^{-bN\delta_N^2}\right)=0,
  \end{equation*}
  while the proof of \cite[Theorem 2.2.2]{Nic:2023} yields \eqref{eq_Nickl_1_27}.
  See \cref{sec_proof_forward_posterior_contraction_result} for a proof of \Cref{theorem_Nickl_2_2_2}.
  
  In \Cref{Whittle_Matern_prior} below, we give an example for a Gaussian base prior distribution $\Pi^\prime $ that satisfies the embedding condition \eqref{eq:RKHS_Sob_emb} of \cref{theorem_Nickl_2_2_2}. This example was already mentioned in the last paragraph of \cite[Section 2.2]{Nic:2023}. 
  \begin{proposition}
	\label{Whittle_Matern_prior}
	Let $\calZ \subset \R^{d_z}$ be a bounded domain with smooth boundary, and let $\zeta$ be a probability measure on $\calZ$ with bounded Lebesgue density $p_\zeta$.
	Then, for any $\alpha > d_z/2$ and $\kappa \ge 0$, there exists a Whittle--Mat\'{e}rn-type Gaussian process $\theta = \{\theta(z): z \in \calZ\}$ such that the law $\Pi^\prime $ of $\theta$ satisfies \cref{assumption_Nickl_Condition_2_2_1} with $\calR \leftarrow H^\beta(\calZ)$ for any $\beta \in [0, \alpha - d_z/2)$.
	Furthermore, the RKHS $\calH$ satisfies the continuous embedding \eqref{eq:RKHS_Sob_emb}.
\end{proposition}

\begin{proof}[Proof of \cref{Whittle_Matern_prior}]
Let $\theta$ be a Whittle--Mat\'{e}rn-type Gaussian process on $\calZ$ with covariance function $K_\alpha$; see e.g. the construction in \cite[Section B.1.3]{Nic:2023}.
As stated in \cite[Theorem B.1.3]{Nic:2023}, the RKHS of $\theta$ is then given by $\calH = H^\alpha(\calZ)$ with $\alpha > d_z/2$, and the process $\theta$ almost surely takes values in $\calR = H^\beta(\calZ)$ for every $\beta \in [0, \alpha - d_z/2)$.
Now $H^\beta(\calZ)\subset L^2(\calZ)$. Since $p_\zeta$ is bounded, we may apply statement \cref{embedding_L2_L2zeta} of \Cref{lemma_embedding} to conclude that $ L^2(\calZ) \hookrightarrow L^2_\zeta(\calZ)$. Thus, there exists a linear space $\Theta$ that satisfies $\calR \subseteq \Theta \subseteq L^2_\zeta(\calZ)$, and \Cref{assumption_Nickl_Condition_2_2_1} is satisfied.

Next, we verify that the RKHS $\calH$ satisfies the continuous embedding \eqref{eq:RKHS_Sob_emb}.
Since the RKHS is $H^\alpha(\calZ)$ for $\alpha>d_z/2$ by the paragraph above, the criterion that $\calH\hookrightarrow H^\alpha(\calZ)$ if $\kappa<\frac12$ is satisfied.
If $\kappa \ge \frac12$, then by multiplying the Whittle--Mat\'{e}rn process $\theta$ with a cutoff-function $\xi \in C^\infty(\calZ)$ that is compactly supported in $\calZ$, we obtain a process $\theta^\prime \coloneqq \{\xi(z) \theta(z):z \in \calZ\}$ for which the RKHS is continuously embedded in $H^\alpha_c(\calZ)$; see the math display above \cite[eq. (B.1)]{Nic:2023}. Substituting $\theta$ with $\theta_\xi$ yields the desired result for the case $\kappa \geq \frac12$.
\end{proof}

 We now show that if $\mathbb{B}$ is a Banach space such that $\Pi^\prime (\mathbb{B})=1$ and if uniform posterior contraction with respect to the $\mathbb{B}$ metric holds with some rate $(\delta_N^\eta)_{N\in\N}$ and some suitably bounded set $S$ of candidate truths, then the sequence of posterior means is strongly uniformly stochastically bounded in the $\mathbb{B}$ metric and with the same rate $(\delta_N^\eta)_{N\in\N}$; cf. \eqref{eq_strong_uniform_stochastic_boundedness} in \Cref{definition_uniform_stochastic_boundedness}.
\begin{theorem}[Uniform posterior contraction implies uniform convergence of posterior mean]
 \label{theorem_uniform_posterior_contraction_implies_uniform_convergence_posterior_mean}
 Suppose that the forward model $G$ satisfies \eqref{eq_Nickl_2_3} from \Cref{assumption_Nickl_Condition_2_1_1} and that the base prior $\Pi^\prime $, $\calR$ and $\calH$ satisfy \Cref{assumption_Nickl_Condition_2_2_1}.
 Let $(\Pi_N)_{N\in\mathbb{N}}$ satisfy \Cref{definition_rescaled_Gaussian_priors}, let $\delta_N$ satisfy \eqref{eq_Nickl_2_19}, and let $(\mathbb{B},\norm{\cdot}_{\mathbb{B}})$ be a separable Banach space such that $\Pi^\prime (\mathbb{B})=1$.
Suppose that $S\subset\mathbb{B}\cap\calR$ is bounded in the $\mathbb{B}$ norm and the $\calR$ norm, and that for sufficiently large $b>0$, there exist $0<M,C,C_1,\eta<\infty$ sufficiently large, $0<\mathbf{q}<\infty$ and $N^*\in\N$, such that for every $N\geq N^*$, and for every $\theta_0\in S$,
\begin{equation}
\label{eq_hypothesis_uniform_posterior_contraction_in_Banach_norm}
 P^N_{\theta_0}\left( \Pi_N(\{\theta\in\Theta:\norm{\theta}_{\calR}\leq M,\norm{\theta-\theta_0}_{\mathbb{B}}\leq C \delta_N^\eta\} \vert D_N)\leq 1-e^{-bN\delta_N^2}\right)\leq \frac{C_1}{(N\delta_N^2)^\mathbf{q}},
\end{equation}
and suppose that there exists $0<\mathbf{A}=\mathbf{A}(S,U)<\infty$ such that for the envelope constant $U=U(M)$ as in \eqref{eq_Nickl_2_3},
\begin{equation}
\label{eq_uniform_constant_A}
   \forall \theta_0\in S, N\in\N,\quad \Pi_N(\mathcal{B}_N(\theta_0,U))\geq e^{-\mathbf{A}N\delta_N^2}.
\end{equation}
Then for $N\geq N^*$,
\begin{equation}
\label{eq_concentration_for_posterior_mean}
 \sup_{\theta_0\in S}P^N_{\theta_0}\left(\norm{\Exp^{\Pi_N}[\theta\vert D_N]-\theta_0}_{\mathbb{B}}>C\delta^\eta_N\right)\leq \frac{C_1+2(U^2+1)}{(N\delta_N^2)^{1\wedge \mathbf{q}}}.
\end{equation}
In particular, $(\norm{\Exp^{\Pi_N}[\theta\vert D_N]-\theta_0}_{\mathbb{B}})_{N\in\N}$ is strongly uniformly stochastically bounded with respect to $(\delta_N^\eta)_{N\in\N}$ over $\theta_0 \in S$ and probability decay rate proportional to $(1/(N\delta_N^2)^{1\wedge \mathbf{q}})_{N\in\N}$.
\end{theorem}
For the proof of \Cref{theorem_uniform_posterior_contraction_implies_uniform_convergence_posterior_mean}, see \Cref{sec_proof_uniform_posterior_contraction_implies_uniform_convergence_posterior_mean}.

The hypothesis \eqref{eq_uniform_constant_A} is a condition on the collection $(\mathcal{B}_N(\theta_0,U))_{\theta_0\in S}$ of sets defined by \eqref{eq_Nickl_1_23}.
It is a uniform version of \eqref{eq_Nickl_1_24}, in the sense that the possibly $\theta_0$-dependent exponent $A$ in \eqref{eq_Nickl_1_24} is replaced by an exponent $\mathbf{A}$ that depends only on $S$.

In \Cref{theorem_uniform_posterior_contraction_implies_uniform_convergence_posterior_mean}, two key hypotheses on the set $S$ of candidate true parameters $\theta_0$ are that $S$ is bounded in the $\calR$ norm, and that there exists a scalar $\mathbf{A}$ such that \eqref{eq_uniform_constant_A} holds.
In \cref{proposition_existence_of_uniform_U_and_uniform_A}, we show that if the set $S$ of candidate true parameters is bounded in the $\calH$ norm, then these two key hypotheses are satisfied.

\begin{restatable}{proposition}{existenceOfUniformUandA}
 \label{proposition_existence_of_uniform_U_and_uniform_A}
  Suppose that the hypotheses of \Cref{theorem_Nickl_2_2_2} hold, with the mapping $M'\mapsto U(M')$ as in \eqref{eq_Nickl_2_3}.
If $S\subset\calR$ is bounded in the $\calH$ norm, then $S$ is bounded in the $\calR$ norm.
For every $0<M_0<\infty$, and for $U=U(M_0+\sup_{\theta_0\in S}\norm{\theta_0}_{\calR})$, there exists $0<\mathbf{A}=\mathbf{A}(S,U)<\infty$ that satisfies \eqref{eq_uniform_constant_A}.
\end{restatable}
For the proof of \Cref{proposition_existence_of_uniform_U_and_uniform_A}, see \Cref{sec_proof_of_existence_of_uniform_U_and_uniform_A}.

We conclude this section with an observation that we shall use in later proofs.
\begin{remark}
 \label{remark_hypotheses}
The hypotheses that are common to both \Cref{theorem_Nickl_2_2_2} and \Cref{theorem_uniform_posterior_contraction_implies_uniform_convergence_posterior_mean} are that  \Cref{assumption_Nickl_Condition_2_1_1,assumption_Nickl_Condition_2_2_1} hold, and that $(\Pi_N)_{N\in\N}$ and $(\delta_N)_{N\in\N}$ satisfy \Cref{definition_rescaled_Gaussian_priors} and \eqref{eq_Nickl_2_19} respectively. 
The hypotheses of \Cref{theorem_uniform_posterior_contraction_implies_uniform_convergence_posterior_mean} that are not included in the hypotheses of \Cref{theorem_Nickl_2_2_2} are as follows. First, for some separable Banach space $\mathbb{B}$, we have $\Pi^\prime(\mathbb{B})=1$ and that $S\subset \mathbb{B}\cap\calR$ is bounded in the $\mathbb{B}$ norm and the $\calR$ norm. Second, the forward contraction result \eqref{eq_hypothesis_uniform_posterior_contraction_in_Banach_norm} is assumed to hold. Third, the uniform prior mass lower bound condition \eqref{eq_uniform_constant_A} is assumed to hold.
\end{remark}

\subsection{Uniform posterior contraction and distinguishability in \footnotesize{$L^2_\zeta$} \normalsize{norm}}
\label{subsection_uniform_posterior_contraction_L2}

In the previous section, we showed in  \Cref{theorem_uniform_posterior_contraction_implies_uniform_convergence_posterior_mean} that uniform posterior contraction in some Banach space norm implies uniform convergence of the posterior mean in the same norm.
In this section, we focus on the case where the Banach space is the space $L^2_\zeta(\calZ,W)$.
In \Cref{theorem_uniform_posterior_contraction_L2}, we state a version of the $L^2_\zeta$ posterior contraction result \cite[Theorem 2.3.1]{Nic:2023}, which was stated for a single true parameter $\theta_0$. Our version holds uniformly for all candidate true parameters $\theta_0$ belonging to a suitably bounded set $S$.
Combining \Cref{theorem_uniform_posterior_contraction_L2} and \Cref{theorem_uniform_posterior_contraction_implies_uniform_convergence_posterior_mean}, we obtain uniform $L^2_\zeta$ convergence of the posterior mean; see \Cref{corollary_uniform_L2_convergence_posterior_mean}.
The final result of this section, \Cref{corollary_L2_distinguishability_posterior_mean}, then uses \cref{corollary_distinguishability_under_strong_uniform_stochastic_boundedness} to obtain $L^2_\zeta$ distinguishability from the uniform $L^2_\zeta$ convergence of the posterior mean.

For \Cref{theorem_uniform_posterior_contraction_L2} below, recall the definitions \eqref{eq_Nickl_2_19}, \eqref{eq_Nickl_2_20}, and \eqref{eq_C_of_U_function} of $\delta_N$, $\Theta_N(M)$, and the map $u\mapsto C_u$.
The importance of \Cref{theorem_uniform_posterior_contraction_L2} in the context of this paper is to provide sufficient conditions for the hypothesis \eqref{eq_hypothesis_uniform_posterior_contraction_in_Banach_norm} of \Cref{theorem_uniform_posterior_contraction_implies_uniform_convergence_posterior_mean} to hold.

\begin{theorem}[Uniform $L^2_\zeta$ posterior contraction]
 \label{theorem_uniform_posterior_contraction_L2}
Suppose the hypotheses of \Cref{theorem_Nickl_2_2_2} hold. Suppose that $S\subset\calR\cap\calH$ is bounded in the $\calH$ norm.
Then for every $b>0$, we can choose $\sup_{\theta_0\in S}\norm{\theta_0}_{\calR}<M<\infty$ and $0<m<\infty$ large enough, such that for some $U=U(M)$, some $N^*\in\N$, every $N\geq N^*$ and every $\theta_0\in S$, 
\begin{equation*}
   P^N_{\theta_0}\left(\Pi_N(\{\theta\in\Theta_N(M): d_G(\theta,\theta_0)\leq mC_{U}^{-1/2}\delta_N \}\vert D_N)\leq 1-e^{-bN\delta_N^2}\right)\leq \frac{4(U^2+1)}{N\delta_N^2}.
\end{equation*}
If in addition \Cref{assumption_local_Lipschitz_continuity_of_restricted_inverse_forward_map} holds for $\calR$, some $\eta>0$, and some $L^\prime=L^\prime(M)$, then for every $N\geq N^*$ and every $\theta_0\in S$,
\begin{equation}
 \label{eq_Nickl_2_26}
  P^N_{\theta_0}(\Pi_N(\{\theta\in\Theta_N(M):\norm{\theta-\theta_0}_{L^2_\zeta(\calZ,W)}\leq L^\prime(mC_U^{-1/2}\delta_N)^\eta\}\vert D_N)\leq 1-e^{-bN\delta_N^2})\leq \frac{4(U^2+1)}{N\delta_N^2}.
\end{equation}
\end{theorem}
For the proof of \Cref{theorem_uniform_posterior_contraction_L2}, see \Cref{sec_proof_uniform_posterior_contraction_distinguishability_L2}.

The posterior contraction rate $\delta_N^\eta = N^{-\eta(\alpha + \kappa)/(2\alpha + 2\kappa + d_z)}$ established in \cref{theorem_uniform_posterior_contraction_L2} is the same as the rate in \cite[Theorem 2.3.1]{Nic:2023}. The key difference between \cref{theorem_uniform_posterior_contraction_L2} and \cite[Theorem 2.3.1]{Nic:2023} is that in \cref{theorem_uniform_posterior_contraction_L2}, the conclusion of posterior contraction holds uniformly for all $\theta_0$ from the set $S$, whereas \cite[Theorem 2.3.1]{Nic:2023} is stated for only a single parameter $\theta_0$. 
	By \eqref{eq_Nickl_2_19}, the rate of convergence of the described probabilities in both results is
	\[ \left(N\delta_N^2\right)^{-1} = N^{-\frac{d_z}{2\alpha + 2\kappa + d_z}}. \]

For the following result, recall the definition of strong uniform stochastic boundedness from \Cref{definition_uniform_stochastic_boundedness}, the definition \eqref{eq_Nickl_2_19} of $\delta_N$, and the parameter $\eta>0$ from \Cref{assumption_local_Lipschitz_continuity_of_restricted_inverse_forward_map}.
\begin{corollary}[Uniform $L^2_\zeta$ convergence of posterior mean]
 \label{corollary_uniform_L2_convergence_posterior_mean}
 Suppose that all the hypotheses of \Cref{theorem_uniform_posterior_contraction_L2} hold, for a given set $S$.
Then, for $M$, $m$, $U=U(M)$, $N^*\in\N$, and $L^\prime =L^\prime (M)$ as in \Cref{theorem_uniform_posterior_contraction_L2}, we have
 \begin{equation}
 	\label{eq:conc_ineq_post_mean}
\sup_{\theta_0\in S} P^N_{\theta_0} \left(\norm{\Exp^{\Pi_N}[\theta\vert D_N]-\theta_0}_{L^2_\zeta} > L^\prime(mC_U^{-1/2})^\eta\delta_N^\eta\right) \leq \frac{6(U^2+1)}{N\delta_N^2},\qquad\forall N\geq N^*.
 \end{equation}
In particular, $(\norm{\Exp^{\Pi_N}[\theta \vert D_N] - \theta_0}_{L^2_\zeta})_{N\in\N}$ is strongly uniformly stochastically bounded with respect to $(\delta_N^\eta)_{N\in\N}$ over $\theta_0\in S$, with probability decay rate proportional to $(1/N\delta_N^2)_{N\in\N}$.
\end{corollary}

\begin{remark}
\label{remark_uniform_L2_convergence_posterior_mean}
 Note that if $\zeta$ admits a Lebesgue density $p_\zeta$ with strictly positive essential infimum, then it follows from the proof of statement \cref{embedding_L2zeta_L2} of \Cref{lemma_embedding} that 
\begin{equation*}
 \norm{\Exp^{\Pi_N}[\theta\vert D_N]-\theta_0}_{L^2_\zeta} > L^\prime(mC_U^{-1/2})^\eta\delta_N^\eta \Longleftarrow \norm{\Exp^{\Pi_N}[\theta\vert D_N]-\theta_0}_{L^2}\essinf_{z\in\calZ}\{p_\zeta(z)\} > L^\prime(mC_U^{-1/2})^\eta\delta_N^\eta .
\end{equation*}
By the above display, \eqref{eq:conc_ineq_post_mean} implies 
\begin{equation*}
 \sup_{\theta_0\in S} P^N_{\theta_0} \left(\norm{\Exp^{\Pi_N}[\theta\vert D_N]-\theta_0}_{L^2} > \bigr(\essinf_{z\in\calZ}\{p_\zeta(z)\}\bigr)^{-1} L^\prime(mC_U^{-1/2})^\eta\delta_N^\eta\right) \leq \frac{6(U^2+1)}{N\delta_N^2},\qquad\forall N\geq N^*,
\end{equation*}
which is precisely uniform $L^2$ convergence of the posterior mean, instead of $L^2_\zeta$ convergence as in \Cref{corollary_uniform_L2_convergence_posterior_mean}.
\end{remark}

\begin{proof}[Proof of \Cref{corollary_uniform_L2_convergence_posterior_mean}]
 The proof consists in applying \Cref{theorem_uniform_posterior_contraction_implies_uniform_convergence_posterior_mean}. 
 First, recall that the hypotheses of \Cref{theorem_uniform_posterior_contraction_L2} are that the hypotheses of \Cref{theorem_Nickl_2_2_2} hold and that $S$ is bounded in the $\calH$ norm.
 Thus, it suffices to show that the remaining hypotheses of \Cref{theorem_uniform_posterior_contraction_implies_uniform_convergence_posterior_mean} that are stated in \Cref{remark_hypotheses} are satisfied.
 First, since $\calR\subseteq\Theta\subseteq L^2_\zeta(\calZ)$ and since $\Pi^\prime (\calR)=1$ by \Cref{assumption_Nickl_Condition_2_2_1}, we have $\Pi^\prime (L^2_\zeta(\calZ))=1$, so we may set $\mathbb{B}\leftarrow L^2_\zeta(\calZ)$ in our application of \Cref{theorem_uniform_posterior_contraction_implies_uniform_convergence_posterior_mean}.
 By combining $\Pi^\prime (L^2_\zeta(\calZ))=1$ and \cite[Proposition 2.6.9]{GinNic:2016}, it follows that the RKHS $\calH$ of $\Pi^\prime $ is continuously embedded in $L^2_\zeta(\calZ)$, and thus the boundedness of $S$ in the $\calH$ norm implies the boundedness of $S$ in both the $\calR$ norm and the $L^2_\zeta$ norm.
 Second, the conclusion \eqref{eq_Nickl_2_26} of \Cref{theorem_uniform_posterior_contraction_L2} implies that the hypothesis \eqref{eq_hypothesis_uniform_posterior_contraction_in_Banach_norm} of \Cref{theorem_uniform_posterior_contraction_implies_uniform_convergence_posterior_mean} holds, with the substitutions $C\leftarrow L^\prime(mC_U^{-1/2})^\eta$ for $\eta$ as in \Cref{assumption_local_Lipschitz_continuity_of_restricted_inverse_forward_map}, $C_1\leftarrow 4(U^2+1)$, and exponent $\mathbf{q}\leftarrow 1$.
 Third, the hypothesis that $S$ is bounded in the $\calH$ norm implies that there exists for every sufficiently large $M$ some $\mathbf{A}=\mathbf{A}(S,U(M))$ such that \eqref{eq_uniform_constant_A} holds, by \Cref{proposition_existence_of_uniform_U_and_uniform_A}. 
 Thus, all the hypotheses of \Cref{theorem_uniform_posterior_contraction_implies_uniform_convergence_posterior_mean} are satisfied, and the bound \eqref{eq_concentration_for_posterior_mean} of \Cref{theorem_uniform_posterior_contraction_implies_uniform_convergence_posterior_mean} with the substitutions stated above yields the desired conclusion.

\end{proof}

	The convergence rate $\delta_N^\eta = N^{-\eta(\alpha + \kappa)/(2\alpha + 2\kappa + d_z)}$ for the posterior mean established in \cref{corollary_uniform_L2_convergence_posterior_mean} is the same as \cite[Theorem 2.3.2]{Nic:2023}, with the difference that \eqref{eq:conc_ineq_post_mean} holds uniformly for all $\theta_0$ from the set $S$.
	The rate $\delta_N^\eta$ is also the same convergence rate established for MAP estimators in \cite[Corollary 2.10]{Siebel:2024}, which however imposes stronger, \emph{global} assumptions on $G$, cf.~\cref{conv_MAP_est} below.

Next, we use \cref{corollary_uniform_L2_convergence_posterior_mean} to prove $L^2_\zeta$ distinguishability of infimum plug-in tests based on the posterior mean.
\begin{corollary}[$L^2_\zeta$ distinguishability for the posterior mean] 
	\label{corollary_L2_distinguishability_posterior_mean}
	Suppose that the assumptions of \cref{theorem_uniform_posterior_contraction_L2} hold, for a given set $S$ and for $(\delta_N)_{N\in\N}$ as in \eqref{eq_Nickl_2_19}. Then, for $M$, $m$, $U=U(M)$, $N^*\in\N$, and $L^\prime =L^\prime (M)$ as in \Cref{theorem_uniform_posterior_contraction_L2}, the type $1$ and type $2$ errors of the plug-in tests $\Psi_N = \mathbf{1}_{T_N\,>\,t_N}$ with
	\[ T_N \coloneqq \inf_{h \in H_0} \norm{\Exp^{\Pi_N}[\theta|D_N] - h}_{L_\zeta^2(\calZ)} 
		\qquad \text{and} \qquad
		t_N \coloneqq L'(mC_U^{-1/2})^\eta\delta_N^\eta \]
	for testing $H_0\subseteq S$ against 
	\[ H_{1,N} \subseteq \left\{ \theta_0 \in S: \inf_{h \in H_0} \norm{\theta_0 - h}_{L_\zeta^2(\calZ)} \ge 2 L'(mC_U^{-1/2})^\eta \delta_N^\eta \right\} \]	
	satisfy
	\[ 	\sup_{\theta_0 \in H_0} P_{\theta_0}^N(\Psi_N = 1) \le \frac{6(U^2+1)}{N\delta_N^2}
		\qquad \text{and} \qquad
		\sup_{\theta_0 \in H_{1,N}} P_{\theta_0}^N(\Psi_N = 0) \le \frac{6(U^2+1)}{N\delta_N^2} \]
	for all $N \ge N^*$.
\end{corollary}
Under an additional assumption on $\zeta$, we can obtain $L^2$ distinguishability for the posterior mean; see \Cref{remark_uniform_L2_convergence_posterior_mean}.
\begin{proof}[Proof of \Cref{corollary_L2_distinguishability_posterior_mean}]
	Given the hypotheses,  \cref{corollary_uniform_L2_convergence_posterior_mean} holds, and the conclusion \eqref{eq:conc_ineq_post_mean} of \cref{corollary_uniform_L2_convergence_posterior_mean} implies that \eqref{eq_strong_uniform_stochastic_boundedness} holds with the metric $d$, the estimators $(\widehat{\theta}_N)_{N\in\N}$, the distinguishability rate $(\delta_N)_{N\in\N}$, the probability decay rate $(\epsilon_N)_{N\in\N}$, and the scalars $K>0$ and $N_0\in\N$ in \eqref{eq_strong_uniform_stochastic_boundedness} replaced with the $L^2_\zeta$ metric, the posterior means $(\Exp^{\Pi_N}[\theta|D_N])_{N\in\N}$, the modified contraction rate $(\delta_N^\eta)_{N\in\N}$, the probability decay rate $(6(U^2+1)/(N\delta_N^2))_{N\in\N}$, and the constants $L'(mC_U^{-1/2})^\eta$ and $N^*$ from \Cref{corollary_uniform_L2_convergence_posterior_mean} respectively.
	The conclusion follows by applying \cref{corollary_distinguishability_under_strong_uniform_stochastic_boundedness} with the above-mentioned replacements.
\end{proof}

\subsection{Uniform posterior contraction and distinguishability in Sobolev and supremum norms}
\label{subsection_uniform_posterior_contraction_Linfty}

In this section, we show that under stronger assumptions than those in \Cref{subsection_uniform_posterior_contraction_L2}, we can use \Cref{theorem_uniform_posterior_contraction_L2}, which yields uniform posterior contraction in $L^2_\zeta$, to obtain uniform posterior contraction in the standard Sobolev norms for a suitable range of Sobolev index, and thus also in the supremum norm; see \Cref{theorem_uniform_posterior_contraction_Linfty} below.
Combining \Cref{theorem_uniform_posterior_contraction_Linfty} with \Cref{theorem_uniform_posterior_contraction_implies_uniform_convergence_posterior_mean} yields uniform convergence of the posterior mean in the same class of norms; see \Cref{corollary_uniform_Sobolev_convergence_posterior_mean}.
In \Cref{corollary_Linfty_distinguishability_posterior_mean}, we combine \Cref{corollary_uniform_Sobolev_convergence_posterior_mean} with \cref{corollary_distinguishability_under_strong_uniform_stochastic_boundedness} to obtain distinguishability in the same class of norms.

For the following result, recall that $\calZ\subset \R^{d_z}$ is a bounded, smooth domain; see \Cref{definition_rescaled_Gaussian_priors}.
\begin{theorem}[Uniform Sobolev and supremum norm posterior contraction]
 \label{theorem_uniform_posterior_contraction_Linfty}
Suppose that all the hypotheses of \Cref{theorem_uniform_posterior_contraction_L2} hold for a set $S$ of candidate truths $\theta_0$ and for $\calR$ that is continuously embedded in $H^\beta(\calZ)$ for some $\beta>d_z/2$, and let $b$, $M$, $m$, $N^*$, and $U$ be such that \eqref{eq_Nickl_2_26} holds. 
Let $\zeta$ be a probability measure on $\calZ$ with a Lebesgue density $p_\zeta$ that has strictly positive essential infimum.
Then for the same choice of $b$ and $M$, for $\delta_N$ as in \eqref{eq_Nickl_2_19}, for $\eta$ as in \Cref{assumption_local_Lipschitz_continuity_of_restricted_inverse_forward_map}, and for any fixed $\beta'\in[0,\beta)$, there exists some $0<C<\infty$ such that the following holds for every $N\geq N^*$ and $\theta_0\in S$:
\begin{equation}
 \label{eq_Sobolev_posterior_contraction_uniform_version}
 P^N_{\theta_0}\left( \Pi_N(\{\theta\in\Theta:\norm{\theta}_{\calR}\leq M,\norm{\theta-\theta_0}_{H^{\beta^\prime}}\leq C \delta_N^{\eta(\beta-\beta')/\beta}\} \vert D_N)\leq 1-e^{-bN\delta_N^2}\right)\leq \frac{4(U^2+1)}{N\delta_N^2}.
\end{equation}
In particular, for any fixed $\beta^\prime \in (d_z/2,\beta)$, there exists some $0<C^\prime<\infty$, such that the following holds for every $N\geq N^*$ and $\theta_0\in S$:
\begin{equation}
 \label{eq_Linfty_posterior_contraction_uniform_version}
 P^N_{\theta_0}\left( \Pi_N(\{\theta\in\Theta:\norm{\theta}_{\calR}\leq M,\norm{\theta-\theta_0}_\infty\leq C^\prime \delta_N^{\eta(\beta-\beta')/\beta} \}\vert D_N)\leq 1-e^{-bN\delta_N^2}\right)\leq \frac{4(U^2+1)}{N\delta_N^2}.
\end{equation}
\end{theorem}
For the proof of \Cref{theorem_uniform_posterior_contraction_Linfty}, see \Cref{sec_proof_uniform_posterior_contraction_distinguishability_Sobolev_Linfty}.

For a given $\beta > d_z/2$, the posterior contraction rate $\delta_N^{\eta(\beta - \beta')/\beta}$ in the supremum norm stated in \eqref{eq_Linfty_posterior_contraction_uniform_version} is constrained by the lower bound $\beta^\prime > d_z/2$, whereas the posterior contraction rate $\delta_N^{\eta(\beta - \beta')/\beta}$ in the $H^{\beta^\prime}$ norm can be increased by choosing $\beta^\prime \in [0,d_z/2]$, at the cost of posterior contraction in a weaker norm.
In comparison with \cite[Proposition 4.1.3]{Nic:2023}, the statement of \cref{theorem_uniform_posterior_contraction_Linfty} is not limited to the case where $\zeta$ is the uniform distribution.

The next result is analogous to \Cref{corollary_uniform_L2_convergence_posterior_mean}. Recall the condition \eqref{eq_strong_uniform_stochastic_boundedness} from \Cref{definition_uniform_stochastic_boundedness} of strong uniform stochastic boundedness.
\begin{corollary}[Uniform Sobolev and supremum norm convergence of posterior mean]
 \label{corollary_uniform_Sobolev_convergence_posterior_mean}
 Suppose all the hypotheses of \Cref{theorem_uniform_posterior_contraction_Linfty} hold with some scalars $\beta$ and $\eta$ and some set $S$ of candidate truths $\theta_0$. Then the following statements hold:
\begin{enumerate}
 \item \label{item_Sobolev_convergence_posterior_mean} For any $\beta'\in (0,\beta)$, $C > 0$ can be chosen large enough such that for some $0<C_1<\infty$ and $N_0 \in \N$, and for all $N \ge N_0$,
 \begin{equation*}
	\sup_{\theta_0\in S} P^N_{\theta_0} \left(\norm{\Exp^{\Pi_N}[\theta\vert D_N]-\theta_0}_{H^{\beta^\prime}} > C\delta_N^{\eta(\beta - \beta')/\beta}\right) \leq \frac{C_1}{N\delta_N^2},
 \end{equation*}
 i.e. $(\norm{\Exp^{\Pi_N}[\theta \vert D_N] - \theta_0}_{H^{\beta^\prime}})_{N\in\N}$ is strongly uniformly stochastically bounded with respect to the rate $(\delta_N^{\eta(\beta-\beta^\prime)/\beta})_{N\in\N}$ over $\theta_0\in S$ and with probability decay rate proportional to $(1/N\delta_N^2)_{N\in\N}$.
 \item \label{item_Linfty_convergence_posterior_mean} For any $\beta'\in (d_z/2,\beta)$, $C^\prime > 0$ can be chosen large enough such that for some $0<C_2<\infty$ and $N_0 \in \N$, and for all $N \ge N_0$,
 \begin{equation*}
	\sup_{\theta_0\in S} P^N_{\theta_0} \left(\norm{\Exp^{\Pi_N}[\theta\vert D_N]-\theta_0}_{\infty} > C^\prime\delta_N^{\eta(\beta - \beta')/\beta}\right) \leq \frac{C_2}{N\delta_N^2},
 \end{equation*}
i.e. $(\norm{\Exp^{\Pi_N}[\theta \vert D_N] - \theta_0}_\infty)_{N\in\N}$ is strongly uniformly stochastically bounded with respect to the rate $(\delta_N^{\eta(\beta-\beta^\prime)/\beta})_{N\in\N}$ over $\theta_0\in S$, with probability decay rate proportional to $(1/N\delta_N^2)_{N\in\N}$.
\end{enumerate}
\end{corollary}
\begin{proof}[Proof of \Cref{corollary_uniform_Sobolev_convergence_posterior_mean}]
The proof is analogous to the proof of \Cref{corollary_uniform_L2_convergence_posterior_mean} and consists in verifying that the hypotheses of \Cref{theorem_uniform_posterior_contraction_implies_uniform_convergence_posterior_mean} hold, for the Sobolev norm case and the supremum norm case.
The hypotheses of \Cref{theorem_uniform_posterior_contraction_Linfty} are stronger than the hypotheses of \Cref{theorem_uniform_posterior_contraction_L2}, because in \Cref{theorem_uniform_posterior_contraction_Linfty} we also require that $\calR$ is continuously embedded in $H^\beta(\calZ)$ for some $\beta>d_z/2$.
In particular, the hypothesis in \Cref{theorem_uniform_posterior_contraction_L2} that $S\subset \calR\cap\calH$ is bounded in the $\calH$ norm represents a stronger condition on $S$ when it is implicitly stated as a hypothesis of \Cref{theorem_uniform_posterior_contraction_Linfty}, due to the additional hypothesis that $\calR$ is continuously embedded in $H^\beta(\calZ)$ for some $\beta>d_z/2$.
By \Cref{proposition_existence_of_uniform_U_and_uniform_A}, the hypothesis that $S$ is bounded in the $\calH$ norm implies that the uniform prior mass lower bound condition \eqref{eq_uniform_constant_A} of \Cref{theorem_uniform_posterior_contraction_implies_uniform_convergence_posterior_mean} is satisfied.
It thus suffices to verify the remaining two hypotheses of \Cref{theorem_uniform_posterior_contraction_implies_uniform_convergence_posterior_mean}; cf. \Cref{remark_hypotheses}.

\ref{item_Sobolev_convergence_posterior_mean}: The hypothesis in \Cref{theorem_uniform_posterior_contraction_Linfty} that $\calR$ is continuously embedded in $H^\beta(\calZ)$ for $\beta>d_z/2$ implies that $\calR$ is continuously embedded in $H^{\beta^\prime}(\calZ)$ for any $\beta'<\beta$.
Thus, if $\mathbb{B}\leftarrow H^{\beta^\prime}(\calZ)$, then the hypothesis of \Cref{theorem_uniform_posterior_contraction_implies_uniform_convergence_posterior_mean} that $\Pi^\prime(\mathbb{B})=1$ and that $S\subset \mathbb{B}\cap\calR$ is bounded in both the $\mathbb{B}$ norm and the $\calR$ norm is satisfied, given that $S$ is assumed to be bounded in the $\calH$ norm, cf. \Cref{proposition_existence_of_uniform_U_and_uniform_A}.
Given that all the hypotheses of \Cref{theorem_uniform_posterior_contraction_Linfty} hold, we may apply the conclusion \eqref{eq_Sobolev_posterior_contraction_uniform_version} to ensure that the hypothesis \eqref{eq_hypothesis_uniform_posterior_contraction_in_Banach_norm} of \Cref{theorem_uniform_posterior_contraction_implies_uniform_convergence_posterior_mean} holds, with the substitutions $\mathbb{B}\leftarrow H^{\beta^\prime}(\calZ)$, $C$ in \eqref{eq_hypothesis_uniform_posterior_contraction_in_Banach_norm} substituted with the constant $C$ in \eqref{eq_Sobolev_posterior_contraction_uniform_version}, $\eta$ in \eqref{eq_hypothesis_uniform_posterior_contraction_in_Banach_norm} substituted with the exponent $\eta(\beta-\beta')/\beta$ in \eqref{eq_Sobolev_posterior_contraction_uniform_version}, and $\mathbf{q}\leftarrow 1$.
Thus, the remaining two hypotheses of \Cref{theorem_uniform_posterior_contraction_implies_uniform_convergence_posterior_mean} are satisfied, and by the conclusion \eqref{eq_concentration_for_posterior_mean} of \Cref{theorem_uniform_posterior_contraction_implies_uniform_convergence_posterior_mean} with the substitutions stated above, we obtain the desired conclusion for the Sobolev norm case.

\ref{item_Linfty_convergence_posterior_mean}: As shown in the proof of \ref{item_Sobolev_convergence_posterior_mean} above, $\Pi^\prime(H^{\beta^\prime}(\calZ))=1$ and $S$ is bounded in both the $H^{\beta^\prime}$ and $\calR$ norm, for any $\beta^\prime<\beta$.  
By the Sobolev embedding $\norm{\cdot}_\infty \lesssim \norm{\cdot}_{H^\beta}$, see e.g. \eqref{eq_Nickl_A_4}, it follows that $\Pi^\prime(L^\infty)=1$ and that $S$ is bounded in the supremum norm and $\calR$ norm. 
Given that all the hypotheses of \Cref{theorem_uniform_posterior_contraction_Linfty} hold, we may apply \eqref{eq_Linfty_posterior_contraction_uniform_version} to conclude that the hypothesis \eqref{eq_hypothesis_uniform_posterior_contraction_in_Banach_norm} of \Cref{theorem_uniform_posterior_contraction_implies_uniform_convergence_posterior_mean} holds, with the substitutions $\mathbb{B}\leftarrow L^\infty(\calZ)$, $C$ in \eqref{eq_hypothesis_uniform_posterior_contraction_in_Banach_norm} substituted with the constant $C^\prime$ in \eqref{eq_Linfty_posterior_contraction_uniform_version}, $\eta$ in \eqref{eq_hypothesis_uniform_posterior_contraction_in_Banach_norm} substituted with the exponent $\eta(\beta-\beta')/\beta$ in \eqref{eq_Linfty_posterior_contraction_uniform_version}, and $\mathbf{q}\leftarrow 1$. 
Thus, the remaining two hypotheses of \Cref{theorem_uniform_posterior_contraction_implies_uniform_convergence_posterior_mean} are satisfied, and by the conclusion \eqref{eq_concentration_for_posterior_mean} of \Cref{theorem_uniform_posterior_contraction_implies_uniform_convergence_posterior_mean} with the substitutions stated above, we obtain the desired conclusion for the supremum norm case.
\end{proof}

We now combine \Cref{corollary_distinguishability_under_strong_uniform_stochastic_boundedness} and \Cref{corollary_uniform_Sobolev_convergence_posterior_mean} to obtain distinguishability in Sobolev norms and the supremum norm.

\begin{corollary}[Sobolev and supremum norm distinguishability of the posterior mean]
	\label{corollary_Linfty_distinguishability_posterior_mean}
	Suppose that the assumptions of \cref{theorem_uniform_posterior_contraction_Linfty} hold with $\beta$ and $\eta$.
	Let $\delta_N$ be as in \eqref{eq_Nickl_2_19} for all $N \in \N$. Moreover, suppose that either 
	\begin{enumerate}   
	 \item \label{item_Sobolev_distinguishability_posterior_mean} $\mathbb{B} = H^{\beta^\prime}(\calZ)$ for an arbitrary $\beta^\prime \in [0,\beta)$, or
	 \item \label{item_Linfty_distinguishability_posterior_mean} $\mathbb{B} = L^\infty(\calZ)$ and $\beta^\prime \in (d_z/2, \beta)$.
	\end{enumerate}
	Then, $K, K^\prime > 0$ can be chosen large enough such that the type $1$ and type $2$ errors of the plug-in tests $\Psi_N = \mathbf{1}_{T_N\,>\,t_N}$ with
	\[ T_N \coloneqq \inf_{h \in H_0} \norm{\Exp^{\Pi_N}[\theta|D_N] - h}_\mathbb{B}
		\qquad \text{and} \qquad
		t_N \coloneqq K\delta_N^{\eta(\beta - \beta')/\beta} \]
	for testing $H_0\subseteq S$ against
		\[ H_{1,N} \subseteq \left\{ \theta_0 \in S: \inf_{h \in H_0} \norm{\theta_0 - h}_\mathbb{B} \ge 2K\delta_N^{\eta(\beta - \beta')/\beta} \right\} \]
	satisfy
	\[ 	\sup_{\theta_0 \in H_0} P_{\theta_0}^N(\Psi_N = 1) \le \frac{K^\prime}{N\delta_N^2}
		\qquad \text{and} \qquad
		\sup_{\theta_0 \in H_{1,N}} P_{\theta_0}^N(\Psi_N = 0) \le \frac{K^\prime}{N\delta_N^2} \]
	for all $N \ge N_0$, where $N_0 \in \N$ is chosen as in \cref{corollary_uniform_Sobolev_convergence_posterior_mean}.
\end{corollary}
\begin{proof}[Proof of \Cref{corollary_Linfty_distinguishability_posterior_mean}]
	\ref{item_Sobolev_distinguishability_posterior_mean}: Suppose $\mathbb{B}= H^{\beta^\prime}(\calZ)$. 
	Given the hypotheses, \cref{corollary_uniform_Sobolev_convergence_posterior_mean} holds. 
	Thus, we have strong uniform stochastic boundedness, with the metric $d$, the estimators $(\widehat{\theta}_N)_{N\in\N}$, the distinguishability rate $(\delta_N)_{N\in\N}$, the probability decay rate $(\epsilon_N)_{N\in\N}$, and the scalars $K>0$ and $N_0\in\N$ in \eqref{eq_strong_uniform_stochastic_boundedness} replaced by the $H^{\beta^\prime}$ metric, the posterior means $(\Exp^{\Pi_N}[\theta|D_N])_{N\in\N}$, the modified contraction rate $(\delta_N^{\eta(\beta-\beta^\prime)/\beta})_{N\in\N}$, the probability decay rate $(C_1/(N\delta_N^2))_{N\in\N}$, and the constants $C$ and $N_0$ from statement \ref{item_Sobolev_convergence_posterior_mean} of \Cref{corollary_uniform_Sobolev_convergence_posterior_mean} respectively.
	Applying \cref{corollary_distinguishability_under_strong_uniform_stochastic_boundedness} with these replacements yields the desired conclusion for $\mathbb{B}= H^{\beta^\prime}(\calZ)$, with $K\leftarrow C$ and $K^\prime\leftarrow C_1$.
	
	\ref{item_Linfty_distinguishability_posterior_mean}: Suppose $\mathbb{B}= L^\infty(\calZ)$. 	Given the hypotheses, \cref{corollary_uniform_Sobolev_convergence_posterior_mean} holds. 
	Thus, we have strong uniform stochastic boundedness, with the metric $d$, the estimators $(\widehat{\theta}_N)_{N\in\N}$, the distinguishability rate $(\delta_N)_{N\in\N}$, the probability decay rate $(\epsilon_N)_{N\in\N}$, and the scalars $K>0$ and $N_0\in\N$ in \eqref{eq_strong_uniform_stochastic_boundedness} replaced with the metric induced by the supremum norm, the posterior means $(\Exp^{\Pi_N}[\theta|D_N])_{N\in\N}$, the modified contraction rate $(\delta_N^{\eta(\beta-\beta^\prime)/\beta})_{N\in\N}$, the probability decay rate $(C_2/(N\delta_N^2))_{N\in\N}$, and the constants $C^\prime$ and $N_0$ from statement \ref{item_Linfty_convergence_posterior_mean} of \Cref{corollary_uniform_Sobolev_convergence_posterior_mean} respectively.
	Applying \cref{corollary_distinguishability_under_strong_uniform_stochastic_boundedness} with these replacements yields the desired conclusion for $\mathbb{B}= L^\infty(\calZ)$, with $K\leftarrow C$ and $K^\prime\leftarrow C_2$.
\end{proof}

\subsection{Distinguishability in {\footnotesize $L^2$} for tests based on MAP estimators}
\label{sec:map_estimator}

Recall that in \eqref{eq:inverse_problem}, $G$ is a nonlinear mapping from $\Theta$ to $L^2_\lambda(\X,V)$, where $\Theta\subseteq L^2_\zeta(\calZ,W)$, $\X \subseteq \R^{d_x}$, $V$ and $W$ are normed spaces of finite dimension $d_V$ and $d_W$, and $\epsilon_i \sim \Normal(0,\sigma^2\Id_V)$ for some $\sigma\ne 0$.
For a given RKHS $\calH$ of functions from $\calZ$ to $W$, and for a given collection $(y_i)_{i=1}^N\subset V$, we define the \emph{maximum a posteriori (MAP)} or \emph{penalised least squares estimator} $\widehat{\theta}_{\mathrm{MAP},N}$ as a minimiser of a nonparametric Tikhonov-Phillips functional $J_{r,\calH}^N:\calH\to \R$, where
\begin{equation}
	\label{eq:def_MAP_estimator}
	J_{r,\calH}^N(\theta) \coloneqq \frac{1}{2\sigma^2N} \sum_{i=1}^N \bignorm{y_i - \left([G(\theta)](x_i)\right)}_V^2 + \frac{r^2}{2} \norm{\theta}_\calH^2, \qquad \forall \theta \in \calH,
\end{equation}
with regularisation parameter $r > 0$.
See e.g. \cite[Proposition 2.5]{Siebel:2024} for a result that ensures the existence of the MAP estimator $\widehat{\theta}_{\mathrm{MAP},N}$.
In the context of model \eqref{eq:inverse_problem}, the MAP estimator can be interpreted as a point of maximal posterior probability density for a posterior distribution based upon a centred Gaussian prior distribution $\Pi$ with Cameron--Martin space $\calH$.
In this section, we consider the case $\calH = H^\alpha(\calZ,W)$.
We first recall \cite[Condition 2.3]{Siebel:2024}:
\begin{assumption}[Global boundedness and Lipschitz continuity]
	\label{assumptions_G_global}
	\begin{enumerate}
		\item There exist $\alpha, \gamma_1, \kappa \ge 0$ and $C_{\mathrm{Lip},2} > 0$ such that
		\[ \norm{G(\theta_1) - G(\theta_2)}_{L_\lambda^2} \le C_{\mathrm{Lip},2} \Big(1 + \norm{\theta_1}_{H^\alpha}^{\gamma_1}\vee \norm{\theta_2}_{H^\alpha}^{\gamma_1}\Big) \norm{\theta_1 - \theta_2}_{(H^\kappa)^*} \]
		for all $\theta_1, \theta_2 \in \Theta \cap \widetilde{\calH}$, where
		\[ \widetilde{\calH} = \begin{cases}
			H^\alpha(\calZ,W) & \text{if}~\kappa < \frac12, \\
			H_\cpt^\alpha(\calZ,W) & \text{if}~\kappa \ge \frac12.
		\end{cases} \]
		\item There exists $U > 0$ such that
		\[ \sup_{\theta\,\in\,\Theta} \norm{G(\theta)}_\infty \le U. \]
		\item There exist $\alpha, \gamma_2, \eta_1, \eta_2 \ge 0$ and $C_{\mathrm{Lip},\infty} > 0$ such that $\alpha > \max \{\eta_1,\eta_2\} + \frac{d_z}{2}$ and that
		\[ \norm{G(\theta_1) - G(\theta_2)}_\infty \le C_{\mathrm{Lip},\infty} \Big(1 + \max \left\{\norm{\theta_1}_{C^{\eta_1}}^{\gamma_2}, \norm{\theta_2}_{C^{\eta_1}}^{\gamma_2}\right\}\Big) \norm{\theta_1 - \theta_2}_{C^{\eta_2}} \]
		for all $\theta_1, \theta_2 \in H^\alpha(\calZ,W)$.
	\end{enumerate}
\end{assumption}
In contrast with \cref{assumption_Nickl_Condition_2_1_1}, which involves one Lipschitz continuity condition and one boundedness condition on the intersection of $\Theta$ with balls defined by the $\calR$-norm, \Cref{assumptions_G_global} involves two Lipschitz continuity conditions and one boundedness condition $G$ on the smaller spaces $\widetilde{\calH}$ or $H^\alpha$, without localisation to balls in the $\calR$ norm.

\begin{theorem}[Uniform $L^2$ convergence of MAP estimators]
	\label{conv_MAP_est}
	Let $\calZ \subset \R^{d_z}$ be a bounded set with smooth boundary.
	Suppose that \cref{assumption_local_Lipschitz_continuity_of_restricted_inverse_forward_map} holds with $\calR = H^\alpha(\calZ,W)$ and $\eta \in (0,1]$, and suppose that \cref{assumptions_G_global} holds, with
	\[ \alpha \ge \max \left\{ \eta_1\vee \eta_2 + d_z\vee \frac{d_z}{2}(1 + \gamma_2), \frac{d_z}{2}\gamma_1 - \kappa \right\}. \]
	Then for every $\overline{c}, M > 0$, there exist $c_7,c_8 > 0$ such that a minimiser $\widehat{\theta}_{\mathrm{MAP},N}$ of $J_{\delta_N,H^\alpha}^N$ satisfies
	\begin{equation}
		\label{eq:conc_ineq_Siebel_simplified}
		\sup_{\theta_0\,\in\,B_{H^\alpha}(M)} P_{\theta_0}^N \left(\norm{\widehat{\theta}_{\mathrm{MAP},N} - \theta_0}_{L^2} > c_7 \delta_N^\eta\right) < c_8 \exp \left(-\overline{c}N\delta_N^2\right),\quad \forall N\in\N
	\end{equation}
	with
	\[ \delta_N = N^{-\frac{\alpha + \kappa}{2\alpha + 2\kappa + d_z}} \]
	as in \eqref{eq_Nickl_2_19}.
	In particular, $(\widehat{\theta}_{\mathrm{MAP},N})_{N\in\N}$ is strongly uniformly stochastically bounded with respect to the rate $(\delta_N^\eta)_{N\in\N}$ over $\theta_0\in B_{H^\alpha}(M)$, with probability decay rate proportional to $(\exp(-\overline{c}N\delta_N^2))_{N\in\N}$.
\end{theorem}

\begin{proof}[Proof of \Cref{conv_MAP_est}]
	First, note that \cite[Condition 2.3]{Siebel:2024} is satisfied by \cref{assumptions_G_global}.
	Next, let us show that the stability estimate \cite[Condition 2.9]{Siebel:2024} is satisfied.
	To this end, note that for all $\theta_1, \theta_2 \in B_{H^\alpha}(M)$, we have
	$\norm{\theta_1}_{H^\alpha} + \norm{\theta_2}_{H^\alpha} \le 2M$,
	so that \cite[Condition 2.9]{Siebel:2024} follows from \cref{assumption_local_Lipschitz_continuity_of_restricted_inverse_forward_map} with $M$ replaced by $2M$.
	Thus, we may apply \cite[Corollary 2.10]{Siebel:2024}, which yields the existence of the stated scalars $c_7,c_8>0$, such that
	\begin{align*}
		&\hspace{14pt}P_{\theta_0}^N \left(\norm{\widehat{\theta}_{\mathrm{MAP},N} - \theta_0}_{L^2} > c_7 \delta_N^\eta\right) \\
		&\le P_{\theta_0}^N \left(\norm{\widehat{\theta}_{\mathrm{MAP},N} - \theta_0}_{L^2} > c_7 \delta_N^\eta~\text{or}~\norm{\widehat{\theta}_{\mathrm{MAP},N}}_{H^\alpha} > c_7\right) < c_8 \exp \left(-\overline{c}N\delta_N^2\right)
	\end{align*}
	for all $\theta_0 \in B_{H^\alpha}(M)$ and every $N\in\N$.
	The final statement follows since \eqref{eq:conc_ineq_Siebel_simplified} implies that the condition \eqref{eq_strong_uniform_stochastic_boundedness} from \Cref{definition_uniform_stochastic_boundedness} for strong uniform stochastic boundedness is satisfied with $d$, rate $(\delta_N)_{N\in\N}$, set $S$, probability decay rate $(\epsilon_N)_{N\in\N}$, and scalars $K>0$ and $N_0\in\N$ replaced with the $L^2(\calZ,W)$ metric, the rate $(\delta_N^\eta)_{N\in\N}$, the set $B_{H^\alpha}(M)$, probability decay rate $(c_8\exp(-\overline{c}N\delta_N^2))_{N\in\N}$, and scalars $c_7$ and $1$ respectively.
\end{proof}

In \Cref{conv_MAP_est}, $\zeta$ is chosen as the uniform distribution on the bounded set $\calZ \subset \R^{d_z}$. Thus, $\zeta$ is proportional to $\mathbf{1}_\calZ\di z$, where $\di z$ denotes the Lebesgue measure on $\R^{d_z}$.
Moreover, the regularisation parameter is chosen as $r = \delta_N$, since we consider the minimisers of $J_{\delta_N,H^\alpha}^N$; cf. the definition \eqref{eq:def_MAP_estimator} of $J_{r,H^\alpha}^N$.
The functional $J_{r,\calH}^N$ is related to the log-likelihood $\ell_N$ via
\[ NJ_{r,\calH}^N(\theta) = -\ell_N(\theta) + \frac{Nr^2}{2}\norm{\theta}_{H^\alpha}^2, \quad \forall \theta \in H^\alpha(\calZ,W). \]
Therefore, the choice of $r$ can be interpreted as choosing a Gaussian prior with Cameron--Martin space $H^\alpha$ that is scaled by a constant multiple of
\[ \left(Nr^2\right)^{-\frac12} = N^{-\frac12}\delta_N^{-1} = N^{-\frac{d_z}{2\alpha + 2\kappa + d_z}}, \]
which is the same scaling used in \cref{definition_rescaled_Gaussian_priors}.

\begin{corollary}[$L^2$ distinguishability for MAP estimators]
	\label{corollary_distinguishability_MAP}
	Let $M > 0$.
	Under the assumptions of \cref{conv_MAP_est}, for every $\overline{c} > 0$, there exists $c_7,c_8 > 0$ such that the type $1$ and type $2$ errors of the tests $\Psi_N \coloneqq \mathbf{1}_{T_N > t_N}$ for 
	\[ T_N \coloneqq \inf_{h \in H_0} \norm{\widehat{\theta}_{\mathrm{MAP},N} - h}_{L^2(\calZ,W)},\qquad t_N \coloneqq c_7\delta_N^\eta \]
	and $\delta_N$ as in \eqref{eq_Nickl_2_19} for testing $H_0 \subseteq B_{H^\alpha}(M)$ against
	\[ 
		H_{1,N} \subset \left\{\theta \in B_{H^\alpha}(M): \inf_{h \in H_0} \norm{\theta - h}_{L^2} \ge 2c_7\delta_N^\eta\right\} 
	\]
	satisfy
	\begin{align*}
		\sup_{\theta_0 \in H_0} P_{\theta_0}^N(\Psi_N = 1) \le c_8 \exp(-\overline{c}N\delta_N^2)\qquad\text{and}\qquad
		\sup_{\theta_0 \in H_{1,N}} P_{\theta_0}^N(\Psi_N = 0) \le c_8 \exp(-\overline{c}N\delta_N^2)
	\end{align*}
	for all $N \in \N$.
\end{corollary}

\begin{proof}[Proof of \Cref{corollary_distinguishability_MAP}]
	The conclusion \eqref{eq:conc_ineq_Siebel_simplified} of strong uniform stochastic boundedness from \cref{conv_MAP_est} implies that the hypothesis of \cref{corollary_distinguishability_under_strong_uniform_stochastic_boundedness} holds with the set $S\leftarrow B_{H^\alpha}(M)$, the estimators $(\widehat{\theta}_N)_{N\in\N}\leftarrow (\widehat{\theta}_{\mathrm{MAP},N})_{N\in\N}$, the rate $( \delta_N^\eta)_{N\in\N}$, the probability decay rate $(c_8\exp(-\overline{c}N\delta_N^2))_{N\in\N}$, and the scalars $K\leftarrow c_7$ and $N_0\leftarrow 1$.
	The stated error bounds for the tests $(\Psi_N)_{N\in\N}$ thus follow from the error bounds \eqref{eq_type1_type2_errors_bounds_strong_uniform_stochastic_boundedness} in \cref{corollary_distinguishability_under_strong_uniform_stochastic_boundedness}.
\end{proof}

\section{Inverse problems governed by linear ordinary differential equations}
\label{sec:app_ode_ivp}

In this section, we consider a class of linear ODE-IVPs that include the two-compartment model described in \cref{sec:cov_mech_mod} and define a nonlinear inverse problem that can be analysed within the framework of \Cref{sec:nonlin_ip}.

Let $\calP$ be a nonempty open subset of $\R^{d_p}$. For every $p\in\calP$, consider the corresponding time-homogeneous, linear ODE-IVP
\begin{equation}
	\label{eq:ode_ivp}
	\begin{aligned}
		\frac{\di}{\di t} s(t,p) &= A(p)s(t,p) \quad \text{for}~t \in [0,T], \\
		s(0,p) &= s_0(p)
	\end{aligned}
\end{equation}
in $\R^{d_s}$, $d_s \in \N$, where both the matrix $A: \calP \to \R^{d_s \times d_s}$ and the initial condition $s_0:\calP \to \R^{d_s}$ depend on a parameter vector $p \in \calP$, and $0<T <\infty$ is fixed.
Both $A(p)$ and $s_0(p)$ are assumed to be locally bounded and locally Lipschitz continuous functions of $p$, and $A(p)$ is assumed to be diagonalisable over $\R$.

\begin{assumption}[{\cite[Assumption 2.1]{Lie:2024}}]
	\label{Lie_assumption_2_1}
	For every $M > 0$, there exists $C_1(M) > 0$ such that
	\begin{align*}
		\norm{A(p)}_2 &\le C_1(M), & \norm{s_0(p)} &\le C_1(M), \\
		\norm{A(p) - A(q)}_2 &\le C_1(M)\norm{p - q}_2, & \norm{s_0(p) - s_0(q)}_2 &\le C_1(M)\norm{p - q}_2
	\end{align*}
	for all $p, q \in B_2(0,M)\coloneqq \{p^\prime\in \R^{d_p}: \norm{p^\prime}_2\leq M\}$.
\end{assumption}
For every $d\in\N$, $\mathrm{GL}(d,\R)$ denotes the set of invertible elements of $\R^{d\times d}$.
\begin{assumption}[{\cite[Assumption 2.3]{Lie:2024}}]
	\label{Lie_assumption_2_3}
	For every $p \in \calP$, there exist $\Lambda(p), V(P) \in \mathrm{GL}(d_s,\R)$ with diagonal $\Lambda(p)$, such that $A(p) = V(p)\Lambda(p)V^{-1}(p)$.
\end{assumption}

If \cref{Lie_assumption_2_3} holds, then the first component $s_1$ of the solution $s$ to the ODE-IVP \eqref{eq:ode_ivp} is of the form
\[ s_1(t,p) = \sum_{i=1}^{\dint} a_i(p) e^{\lambda_i(p)t} \quad \text{for all}~t \in [0,T] \]
with $\N\ni\dint \le d_s$ distinct real-valued functions $\lambda_1, \dots, \lambda_\dint$ and $\dint$ real-valued functions $a_1, \dots, a_\dint$ of $p$, see page 6 in \cite{Lie:2024}. Here, we refer to $\calP \ni p\mapsto  (a_i(p), \lambda_i(p))_{i=1}^p\in \R^{2\dint}$ as
 the \emph{coefficient map}, and to $2\dint$ as the \emph{intrinsic dimension} of $s_1$.

 For the following assumption, we denote by $A^\top$ the transpose of $A\in\R^{m\times n}$, $m,n\in\N$.
\begin{assumption}
	\label{assumptions_coefficient_map}
	The coefficent map of $s_1$ belongs to $C^1(\calP,\R^{2\dint})$, $a_i(p) > 0$ for all $p \in \calP$ and $i \in \{1,\ldots,\dint\}$, and for every $q \in \calP$, the Jacobian of the coefficient map evaluated at $q$, i.e.
	\begin{equation}
		\left.
		\mathcal{J}(q) = \begin{bmatrix}
			\nabla_p a_1(p) & \hdots & \nabla_p a_\dint(p) & 
			\nabla_p \lambda_1(p) & \hdots & \nabla_p \lambda_\dint(p)
		\end{bmatrix}^\top
		\right\vert_{p = q} \in \R^{2\dint \times d_p},
	\end{equation}
	has full rank.
\end{assumption}
\Cref{assumptions_coefficient_map} implies that both Assumptions 2.5 and 3.1 in \cite{Lie:2024} hold. In particular, \Cref{assumptions_coefficient_map} implies that the first component of the solution $s_1(t,p)$ is strictly positive, for every $t \in [0,T]$ and $p \in \calP$. This is important for the definition \eqref{eq:def_G_ode_ivp} of $G$ below.

Recall the discussion from \Cref{sec:cov_mech_mod} regarding the set $\X$ of `covariates' in the context of pharmacokinetic models. Fix a deterministic collection $(t_j)_{j =1}^{d_y} \subset [0,T]$ of $d_y\in\N$ distinct observation times, a set of admissible covariates $\X\subseteq \R^{d_x}$, a probability measure $\mu_X$ on $\X$, and a set of admissible CPMs $\Theta \subseteq L_{\mu_X}^2(\X,\calP)$.
Define the forward model $G:\Theta \to L_{\mu_X}^2(\X,\R^{d_y})$ by
\begin{equation}
	\label{eq:def_G_ode_ivp}
	[G(\theta)](x) \coloneqq \left(\ln s_1(t_j,\theta(x))\right)_{j=1}^{d_y}, \quad \forall \theta \in \Theta, x \in \X.
\end{equation}
Consider the inverse problem defined by \eqref{eq:inverse_problem} and \eqref{eq:def_G_ode_ivp}.
In \eqref{eq:def_G_ode_ivp}, we assume that only the first component of the solution is observable and store the logarithms of its values at the observation times $(t_j)_{j=1}^{d_y}$. 

First, we show that the two-compartment model described in \Cref{sec:cov_mech_mod} belongs to the considered class of ODE-IVPs.
\begin{proposition}
\label{proposition_two_compartment_model_class_ODEs}
	The two-compartment model satisfies \cref{Lie_assumption_2_1,Lie_assumption_2_3,assumptions_coefficient_map} with $2\dint = d_p = 4$.
\end{proposition}
\begin{proof}[Proof of \Cref{proposition_two_compartment_model_class_ODEs}]
	It follows from Proposition 5.3 in \cite{Lie:2024} that \cref{Lie_assumption_2_1,Lie_assumption_2_3} hold.
	Moreover, Lemma B.2 in \cite{Lie:2024} shows that \cref{assumptions_coefficient_map} is satisfied with $\dint = 2$.
\end{proof}

Next, we show that for the considered class of ODE-IVPs, the operator $G$ in \eqref{eq:def_G_ode_ivp} satisfies the key assumptions stated in \Cref{sec:nonlin_ip} on the forward model.

\begin{proposition}
	\label{ode_ivp_satisfies_ip_assumptions}
	Suppose that the ODE-IVP \eqref{eq:ode_ivp} satisfies \cref{Lie_assumption_2_1,Lie_assumption_2_3,assumptions_coefficient_map}, that $\X\subset\R^{d_x}$ is a bounded, smooth domain, that $\mu_X$ is a probability measure on $\X$ with a bounded Lebesgue density $p_X$, that $\calR$ is continuously embedded in $L^\infty(\X,\calP)$, and that $d_y \ge d_p = 2\dint$. Then, the operator $G$ defined in \eqref{eq:def_G_ode_ivp} satisfies \cref{assumption_Nickl_Condition_2_1_1,assumption_local_Lipschitz_continuity_of_restricted_inverse_forward_map} with $\kappa \leftarrow 0$, $\eta \leftarrow 1$, $\calZ \leftarrow \X$, $\zeta \leftarrow \mu_X$, $\lambda \leftarrow \mu_X$, $W \leftarrow \R^{d_p}$, $V \leftarrow \R^{d_y}$,
	\begin{equation}
		\label{eq:U_L_ode_ivp}
		U = \sqrt{d_y}C_2, \quad
		L = \sqrt{d_y}L_s e^{C_2}\norm{p_X}_{L^\infty(\X)}^{1/2}, \quad \text{and} \quad
		L^\prime = C_3,
	\end{equation}
	for  $L_s = L_s(M,T)$ as in \cite[Proposition 2.2]{Lie:2024}, $C_2 = C_2(M,T)$ as in \cite[Proposition 3.2]{Lie:2024}, and $C_3 = C_3(M,T,(t_j)_{j = 1}^{d_y})$ as in \cite[Proposition 3.3]{Lie:2024}.
\end{proposition}
\begin{proof}[Proof of \cref{ode_ivp_satisfies_ip_assumptions}]
	\Cref{assumptions_coefficient_map} implies that both Assumptions 2.5 and 3.1 in \cite{Lie:2024} hold.
	Furthermore, \Cref{Lie_assumption_2_1} and \Cref{Lie_assumption_2_3} correspond to Assumptions 2.1 and 2.3 in \cite{Lie:2024} respectively. 
	The hypothesis that $\calR$ is continuously embedded in $L^\infty(\X,\calP)$ corresponds to the hypothesis \cite[equation (3.2)]{Lie:2024} on the regularisation space $\calR$ made in \cite[Proposition 3.2]{Lie:2024}. Given these observations, we may apply \cite[Proposition 3.2]{Lie:2024}, in particular the conclusion \cite[equation (3.5)]{Lie:2024} with $\mu\leftarrow \mu_X$ and $q \leftarrow 2$, to obtain that for every $M>0$,
	\begin{equation*} 
	 \norm{G(\theta_1) - G(\theta_2)}_{L^2_{\mu_X}} 
	\le \sqrt{d_y}L_s e^{C_2}\norm{\theta_1 - \theta_2}_{L^2_{\mu_X}},\quad \forall \theta_1,\theta_2\in  B_{\calR}(M).
	\end{equation*}
    By the hypothesis that $\mu_X$ admits a bounded Lebesgue density $p_X$, we may apply \cref{embedding_L2_L2zeta} from \Cref{lemma_embedding} with $\zeta\leftarrow\mu_X$ to further bound the right-hand side of the preceding inequality, thus obtaining
	\[ \norm{G(\theta_1) - G(\theta_2)}_{L^2_{\mu_X}} 
	\le \sqrt{d_y}L_s e^{C_2}\norm{p_X}_{L^\infty(\X)}^{1/2} \norm{\theta_1 - \theta_2}_{L^2} \]
	for all $\theta_1, \theta_2 \in B_\calR(M)$. 
	The inequality
	\[ \sup_{\theta \in B_\calR(M)} \norm{G(\theta)}_\infty \le \sqrt{d_y}C_2 < \infty \]
	follows from the conclusion \cite[equation (3.3)]{Lie:2024} of \cite[Proposition 3.2]{Lie:2024}. Thus, \cref{assumption_Nickl_Condition_2_1_1} is satisfied with $\kappa \leftarrow 0$, with $\calZ\leftarrow \X$, $\zeta\leftarrow \mu_X$, $\lambda\leftarrow \mu_X$, $W\leftarrow \R^{d_p}$, $V\leftarrow \R^{d_y}$, and with  $U$ and $L$ as in \eqref{eq:U_L_ode_ivp}.
	Finally, it follows from \cite[Proposition 3.3]{Lie:2024} with $q \leftarrow 2$ that \cref{assumption_local_Lipschitz_continuity_of_restricted_inverse_forward_map} is satisfied with $\eta \leftarrow 1$ and $L^\prime$ as in \eqref{eq:U_L_ode_ivp}.
\end{proof}

The computations verifying that the two-compartment model satisfies \cref{assumption_Nickl_Condition_2_1_1} suggest that the two-compartment model does not satisfy the global conditions of \cref{assumptions_G_global}, because the constants $U$ and $L$ for which it satisfies \cref{assumption_Nickl_Condition_2_1_1} depend exponentially on the radius $M$;
see Lemma B.1 and Proposition 3.2 in \cite{Lie:2024}.
In particular, \cite[equation (3.7)]{Lie:2024} shows that the scalar $C_2$ that determines $U$ and $L$ as shown in \eqref{eq:U_L_ode_ivp}, depends on $C_1$ in \Cref{Lie_assumption_2_1} according to
\begin{equation*}
 C_2(M,T)=C_1(M)T+\vert \log C_1(M)\vert+C^\prime(M) C_1(M)T
\end{equation*}
for some $C^\prime(M)>0$ and $T$. On the other hand, \cite[Lemma B.1]{Lie:2024} states that $C_1(M)\lesssim e^{2M}$.
Therefore, we cannot apply \cref{corollary_distinguishability_MAP} to conclude distinguishability for plug-in tests based upon the MAP estimator.

\subsection{Uniform posterior contraction and distinguishability in {\footnotesize $L^2_{\mu_X}$}}

In this section, we apply the results from \cref{sec:nonlin_ip} to prove posterior contraction and distinguishability in $L^2$, for the inverse problem defined by \eqref{eq:inverse_problem} and \eqref{eq:def_G_ode_ivp}.

\begin{theorem}[Uniform $L^2_{\mu_X}$ posterior contraction]
	\label{posterior_contraction_L2_ode_ivp}
	Suppose that the ODE-IVP \eqref{eq:ode_ivp} satisfies \cref{Lie_assumption_2_1,Lie_assumption_2_3,assumptions_coefficient_map}, that $\X\subset\R^{d_x}$ is a bounded, smooth domain, that $\mu_X$ is a probability measure on $\X$ with a bounded Lebesgue density $p_X$, and that $d_y \ge d_p = 2\dint$.
	Let the base prior distribution $\Pi^\prime $ satisfy \cref{assumption_Nickl_Condition_2_2_1}, such that its RKHS $\calH$ is continuously embedded in $H^\alpha(\X,\calP)$ for some $\alpha>0$ and such that $\calR$ is continuously embedded in $L^\infty(\X,\calP)$, and let $(\Pi_N)_{N\in\N}$ be as in \Cref{definition_rescaled_Gaussian_priors}. Moreover, let $S \subset \calH$ be bounded in the $\calH$ norm. Then, for every $b > 0$ and sufficiently large $M, m > 0$, there exists $N_0 \in \N$ such that for every $\theta_0 \in S$, the posterior distribution $\Pi_N(\cdot|D_N)$ for the inverse problem defined by \eqref{eq:inverse_problem} and \eqref{eq:def_G_ode_ivp} satisfies
	\begin{equation}
		\label{eq:posterior_contraction_L2_ode_ivp}
		P_{\theta_0}^N \left(\Pi_N \Big(\{\theta \in \Theta_N(M): \norm{\theta - \theta_0}_{L^2_{\mu_X}(\X,\calP)} \le L^\prime m C_U^{-1/2}\delta_N \}\Big\vert D_N\Big) \le 1 - e^{-bN\delta_N^2}\right) \le 4\frac{U^2 + 1}{N\delta_N^2},
	\end{equation}
	for all $N \ge N_0$, where $\delta_N \coloneqq N^{-\alpha/(2\alpha + d_x)}$ for all $N \in \N$, $U$ and $L^\prime$ are defined by \eqref{eq:U_L_ode_ivp}, and $C_U$ by \eqref{eq_C_of_U_function}.
\end{theorem}

\begin{remark}
\label{remark_Whittle_Matern_prior_L2_ode_ivp}
A specific example of a Gaussian base prior $\Pi^\prime $ that satisfies the hypotheses of \Cref{posterior_contraction_L2_ode_ivp} is given by the law of a Whittle--Mat\'{e}rn process with RKHS $\calH=H^\alpha(\X,\calP)$ for some $\alpha>d_x$. 
Given the hypotheses on $\X$ and $\mu_X$ in \Cref{posterior_contraction_L2_ode_ivp},  we may apply \cref{Whittle_Matern_prior} to conclude that a Whittle--Mat\'{e}rn process with this regularity exists.  
\end{remark}

The posterior contraction rate $\delta_N=N^{-\alpha/(2\alpha + d_x)}$ established in \cref{posterior_contraction_L2_ode_ivp} is the same as the posterior contraction rate in \cite[Theorem 4.2]{Lie:2024}, with the difference that \eqref{eq:posterior_contraction_L2_ode_ivp} holds uniformly for every $\theta_0$ from the set $S$.
The probability decay rate in \eqref{eq:posterior_contraction_L2_ode_ivp} is
\[ \left(N\delta_N^2\right)^{-1} = N^{-\frac{d_x}{2\alpha + d_x}}. \]

\begin{proof}[Proof of \Cref{posterior_contraction_L2_ode_ivp}]
	Under the stated hypotheses, we may apply \cref{ode_ivp_satisfies_ip_assumptions} to conclude that $G$ as defined in \eqref{eq:def_G_ode_ivp} satisfies \cref{assumption_Nickl_Condition_2_1_1,assumption_local_Lipschitz_continuity_of_restricted_inverse_forward_map} with $\calZ\leftarrow \X$, $\zeta\leftarrow \mu_X$, $\lambda\leftarrow \mu_X$, $V\leftarrow \R^{d_y}$, $W\leftarrow \R^{d_p}$, $\kappa\leftarrow 0$, $\eta\leftarrow 1$, and $U$ and $L^\prime$ as in \eqref{eq:U_L_ode_ivp}. The stated hypotheses on the base prior $\Pi^\prime $ coincide with the hypotheses on the base prior $\Pi^\prime $ stated in \Cref{theorem_Nickl_2_2_2} since $\kappa=0$; cf. \eqref{eq:RKHS_Sob_emb}. 
	Thus, all the hypotheses of \Cref{theorem_Nickl_2_2_2} hold.
	By the hypothesis in \Cref{posterior_contraction_L2_ode_ivp} that $S$ is bounded in the $\calH$ norm, it follows from \Cref{proposition_existence_of_uniform_U_and_uniform_A} that $S$ is also bounded in the $\calR$ norm. 
	Hence, all the hypotheses of \Cref{theorem_uniform_posterior_contraction_L2} hold, and by applying its conclusion \eqref{eq_Nickl_2_26} with $\eta\leftarrow 1$ and with $U$ and $L^\prime$ as in \eqref{eq:U_L_ode_ivp}, we obtain the desired conclusion \eqref{eq:posterior_contraction_L2_ode_ivp}.
\end{proof}

\begin{corollary}[$L^2_{\mu_X}$ distinguishability]
	\label{distinguishability_L2_ode_ivp}
	Suppose the hypotheses of \cref{posterior_contraction_L2_ode_ivp} hold. Then, for $M$, $m$, and $N_0\in\N$ as in \Cref{posterior_contraction_L2_ode_ivp}, for $U$ and $L^\prime$ as in \eqref{eq:U_L_ode_ivp}, and $C_U$ as in \eqref{eq_C_of_U_function}, the type $1$ and type $2$ errors of the plug-in tests $\Psi_N = \mathbf{1}_{T_N > t_N}$ with
	\[ T_N \coloneqq \inf_{h \in H_0} \norm{\Exp^{\Pi_N}[\theta|D_N] - h}_{L^2_{\mu_X}(\X,\calP)},
		\qquad
		t_N \coloneqq L^\prime (m C_U^{-1/2})\delta_N, \]
	and $\delta_N \coloneqq N^{-\alpha/(2\alpha + d_x)}$ for testing $H_0\subseteq S$ against
	\[ H_{1,N} \coloneqq \left\{ \theta_0 \in S: \inf_{h \in H_0} \norm{\theta_0 - h}_{L^2_{\mu_X}(\X,\calP)} \ge 2L^\prime (m C_U^{-1/2})\delta_N \right\} \]
	satisfy
	\[ 	\sup_{\theta_0 \in H_0} P_{\theta_0}^N(\Psi_N = 1) \le \frac{6 (U^2+1)}{N\delta_N^2}
		\qquad \text{and} \qquad
		\sup_{\theta_0 \in H_{1,N}} P_{\theta_0}^N(\Psi_N = 0) \le \frac{6(U^2+1)}{N\delta_N^2} \]	
	for all $N \ge N_0$.
\end{corollary}
If the Lebesgue density $p_X$ of $\mu_X$ has a strictly positive essential infimum, then we can obtain $L^2$ distinguishability from the $L^2_{\mu_X}$ distinguishability statement of \Cref{distinguishability_L2_ode_ivp}; see \Cref{remark_uniform_L2_convergence_posterior_mean}.

\begin{proof}[Proof of \Cref{distinguishability_L2_ode_ivp}]
In the proof of \Cref{posterior_contraction_L2_ode_ivp}, we showed that the hypotheses of \Cref{theorem_uniform_posterior_contraction_L2} hold with $\calZ\leftarrow \X$, $\zeta\leftarrow \mu_X$, $\lambda\leftarrow \mu_X$, $V\leftarrow \R^{d_y}$, $W\leftarrow \R^{d_p}$, $\kappa\leftarrow 0$, $\eta\leftarrow 1$, and $U$ and $L^\prime$ as in \eqref{eq:U_L_ode_ivp}.
Thus, we may apply \Cref{corollary_L2_distinguishability_posterior_mean} with the same choices to obtain the desired conclusion.
\end{proof}

\begin{remark}
\label{rem:limit_rate_alpha}
If $H_0$ is a class of infinitely smooth functions that is bounded in $H^\alpha(\X)$ for every $\alpha > 0$, then $\alpha$ can be chosen arbitrarily large in \cref{distinguishability_L2_ode_ivp}. Given the rate $\delta_N\coloneqq N^{-\alpha/(2\alpha+d_x)}$ from \Cref{distinguishability_L2_ode_ivp}, this observation implies $L^2_{\mu_X}$ distinguishability at any rate $\delta_N = N^{-\gamma}$ with $\gamma \in (\frac13,\frac12)$ arbitrarily close to the `parametric rate' of $\frac12$.
This observation is relevant for goodness-of-fit testing for parametric classes of covariate-parameter relationships in pharmacokinetics. Recall the class \eqref{eq:def_Theta0_sat_exp} of exponential functions and the class \eqref{eq:def_Theta0_aff_lin} of affine linear functions $\theta_\tau$ that were described in \cref{sec:cov_mech_mod}, and recall that the elements in each class are parametrised by a finite-dimensional nuisance parameter $\tau$. If the classes are defined by values of $\tau$ that belong to a suitably bounded set $\T$, then by the preceding observation, one obtains $L^2_{\mu_X}$ distinguishability for these classes at an almost parametric rate.
\end{remark}

\subsection{Uniform posterior contraction and distinguishability in Sobolev and supremum norms}

In this section, we apply the results from \cref{sec:nonlin_ip} to prove posterior contraction and distinguishability in certain Sobolev  metrics and in the metric induced by the supremum norm, for the inverse problem defined by \eqref{eq:inverse_problem} and \eqref{eq:def_G_ode_ivp}.

\begin{theorem}[Uniform Sobolev and supremum norm posterior contraction]
	\label{posterior_contraction_Linfty_ode_ivp}
	Suppose that the ODE-IVP \eqref{eq:ode_ivp} satisfies \cref{Lie_assumption_2_1,Lie_assumption_2_3,assumptions_coefficient_map}, that $\mu_X$ has a bounded Lebesgue density with strictly positive essential infimum over $\X$, and that $d_y \ge d_p = 2\dint$.
	Let the base prior distribution $\Pi^\prime $ satisfy \cref{assumption_Nickl_Condition_2_2_1}, let its RKHS $\calH$ be continuously embedded in $H^\alpha(\X,\calP)$ for some $\alpha > 0$, and let $\calR$ be continuously embedded in $H^\beta(\X,\calP)$ for some $\beta > d_x/2$. Moreover, let $S \subset \calH$ be bounded in $\calH$ norm. Then, for every $b > 0$, $\beta^\prime \in [0,\beta)$, and sufficiently large $M > 0$, there exists $C > 0$ and $N_0 \in \N$ such that for every $\theta_0 \in S$ and $N \ge N_0$, the posterior distribution $\Pi_N(\cdot|D_N)$ for the inverse problem defined by \eqref{eq:inverse_problem} and \eqref{eq:def_G_ode_ivp} satisfies
	\begin{equation*}
		P_{\theta_0}^N \left(\Pi_N \Big(\{\theta \in \Theta: \norm{\theta}_\calR \le M, \norm{\theta - \theta_0}_{H^{\beta^\prime}} \le C\delta_N^{(\beta - \beta')/\beta} \}\Big\vert D_N\Big) \le 1 - e^{-bN\delta_N^2}\right) \le \frac{4(U^2 + 1)}{N\delta_N^2},
	\end{equation*}
	where $\delta_N \coloneqq N^{-\alpha/(2\alpha + d_x)}$ for all $N \in \N$ and $U$ is given by \eqref{eq:U_L_ode_ivp}.
	In particular, for every $b > 0$, $\beta^\prime \in (d_x/2,\beta)$, and sufficiently large $M > 0$, there exists $C^\prime > 0$ and $N_0 \in \N$ such that for every $\theta_0 \in S$ and $N \ge N_0$,
	\begin{equation*}
		P_{\theta_0}^N \left(\Pi_N \Big(\{\theta \in \Theta: \norm{\theta}_\calR \le M, \norm{\theta - \theta_0}_{\infty} \le C^\prime\delta_N^{(\beta - \beta')/\beta}\} \Big\vert D_N\Big) \le 1 - e^{-bN\delta_N^2}\right) \le \frac{4(U^2 + 1)}{N\delta_N^2}.
	\end{equation*}
\end{theorem}
In \Cref{remark_Whittle_Matern_prior_L2_ode_ivp}, we showed that there exist Whittle--Mat\'{e}rn Gaussian processes whose laws yield base priors $\Pi^\prime$ that satisfy the hypotheses of \Cref{posterior_contraction_L2_ode_ivp}. 
Similarly, under the hypotheses on $\X$ and $\mu_X$ stated in \Cref{posterior_contraction_Linfty_ode_ivp}, there exists by \Cref{Whittle_Matern_prior} a Whittle--Mat\'{e}rn Gaussian process whose law $\Pi^\prime$ satisfies $\Pi^\prime(H^\beta(\X,\calP))=1$ for any $\beta\in [0,\alpha-d_x/2)$, such that its RKHS $\calH$ satisfies \eqref{eq:RKHS_Sob_emb}. In particular, by the condition $\alpha>d_x$ stated in \Cref{posterior_contraction_Linfty_ode_ivp}, it follows that $(d_x/2,\alpha-d_x/2)$ is nonempty, and hence the condition that $\calR$ is continuously embedded in $H^\beta(\X,\calP)$ for $\beta >d_x/2$ can be satisfied.

\begin{proof}[Proof of \Cref{posterior_contraction_Linfty_ode_ivp}]
The hypotheses of \Cref{posterior_contraction_Linfty_ode_ivp} are strictly stronger than the hypotheses of \Cref{posterior_contraction_L2_ode_ivp}. 
As shown in the proof of \Cref{posterior_contraction_L2_ode_ivp}, the hypotheses of \Cref{posterior_contraction_L2_ode_ivp} imply that all the hypotheses of \Cref{theorem_uniform_posterior_contraction_L2} hold, with the substitutions $\calZ\leftarrow \X$, $V\leftarrow \R^{d_y}$, $W\leftarrow \R^{d_p}$, $\kappa\leftarrow 0$, $\eta\leftarrow 1$, and $U$ and $L^\prime$ as in \eqref{eq:U_L_ode_ivp}. 
The differences of the hypotheses of \Cref{posterior_contraction_Linfty_ode_ivp} with respect to the hypotheses of \Cref{posterior_contraction_L2_ode_ivp} are that the Lebesgue density $p_X$ of $\mu_X$ has strictly positive essential infimum, and that the regularisation space $\calR$ is assumed to be continuously embedded in $H^\beta(\X,\calP)$ for some $\beta>d_x/2$.
These are precisely the differences of the hypotheses of \Cref{theorem_uniform_posterior_contraction_Linfty} with respect to the hypotheses of \Cref{theorem_uniform_posterior_contraction_L2}.
Hence, all the hypotheses of \Cref{theorem_uniform_posterior_contraction_Linfty} hold, and the desired inequalities follow from \eqref{eq_Sobolev_posterior_contraction_uniform_version} and \eqref{eq_Linfty_posterior_contraction_uniform_version}.
\end{proof}

\begin{corollary}[Sobolev and supremum norm distinguishability]
	\label{distinguishability_Linfty_ode_ivp}
Suppose that the hypotheses of \Cref{posterior_contraction_Linfty_ode_ivp} hold, and suppose that either
	\begin{enumerate}   
	 \item \label{item_Sobolev_distinguishability_posterior_mean_ode_ivp} $\mathbb{B} = H^{\beta^\prime}(\X,\calP)$ for some $\beta^\prime \in [0,\beta)$, or
	 \item \label{item_Linfty_distinguishability_posterior_mean_ode_ivp} $\mathbb{B} = L^\infty(\X,\calP)$ with $\beta^\prime \in (d_x/2, \beta)$.
	\end{enumerate}
	Then, there exist $K, K^\prime > 0$ and $N_0 \in \N$ such that the type $1$ and type $2$ errors of the plug-in tests $\Psi_N = \mathbf{1}_{T_N > t_N}$ with
	\[ T_N \coloneqq \inf_{h \in H_0} \norm{\Exp^{\Pi_N}[\theta|D_N] - h}_{\mathbb{B}},
		\qquad
		t_N \coloneqq K\delta_N^{(\beta - \beta')/\beta}, \]
	and $\delta_N \coloneqq N^{-\alpha/(2\alpha + d_x)}$ for testing $H_0\subseteq S$ against
	\[ H_{1,N} \coloneqq \left\{ \theta_0 \in S: \inf_{h \in H_0} \norm{\theta_0 - h}_{\mathbb{B}} \ge 2K\delta_N^{(\beta - \beta')/\beta} \right\} \]
	satisfy
	\[ 	\sup_{\theta_0 \in H_0} P_{\theta_0}^N(\Psi_N = 1) \le \frac{K^\prime}{N\delta_N^2}
		\qquad \text{and} \qquad
		\sup_{\theta_0 \in H_{1,N}} P_{\theta_0}^N(\Psi_N = 0) \le \frac{K^\prime}{N\delta_N^2} \]	
	for all $N \ge N_0$.
\end{corollary}

\begin{proof}[Proof of \Cref{distinguishability_Linfty_ode_ivp}]
As shown in the proof of \Cref{posterior_contraction_Linfty_ode_ivp}, the hypotheses of \Cref{posterior_contraction_Linfty_ode_ivp} imply that the hypotheses of \Cref{theorem_uniform_posterior_contraction_Linfty} hold with the substitutions $\calZ\leftarrow \X$, $V\leftarrow \R^{d_y}$, $W\leftarrow \R^{d_p}$, $\kappa\leftarrow 0$, $\eta\leftarrow 1$, and $U$ and $L^\prime$ as in \eqref{eq:U_L_ode_ivp}. 
Thus, the desired conclusion follows from \Cref{corollary_Linfty_distinguishability_posterior_mean}. 
\end{proof}

By choosing $\beta \in (d_x/2, \alpha - d_x/2)$ arbitrarily close to $\alpha - d_x/2$ and $\beta^\prime \in (d_x/2, \beta)$ arbitrarily close to $d_x/2$ in \cref{distinguishability_Linfty_ode_ivp}, supremum norm distinguishability at the rate $\delta_N = N^{-\gamma}$ can be achieved with $\gamma \in (0, \gamma_0)$ arbitrarily close to
\[ \gamma_0 = N^{-\frac{\alpha}{2\alpha + d_x}\cdot\frac{\alpha - d_x}{\alpha - d_x/2}}
	= N^{-\frac{2\alpha(\alpha - d_x)}{(2\alpha + d_x)(2\alpha - d_x)}}. \]
Here, $\gamma_0$ also approaches the parametric rate $\frac12$ as the prior smoothness parameter $\alpha$ tends to $\infty$, cf.~\cref{rem:limit_rate_alpha}.

\section{Conclusion and outlook}
\label{sec:conclusion}

In this work, we have established sufficient conditions for distinguishability with respect to the $L^2_\zeta$, Sobolev, and supremum norms, in the context of nonparametric goodness-of-fit testing for nonlinear inverse problems with random observations.
These results yield upper bounds on the corresponding minimax separation rates in the same setting.
In order to obtain these results, we proved in \Cref{sec:nonlin_ip} the corresponding posterior contraction rates for suitably regular Gaussian priors, in which the posterior contraction holds uniformly for sets of true parameters that are bounded in the corresponding RKHS norms.
In \Cref{sec:app_ode_ivp}, we applied our results to inverse problems governed by parametrised ODE-IVPs, and derived posterior contraction rates and conditions for distinguishability for this class of inverse problems.
The distinguishability results use tools that we developed to obtain distinguishability for infimum plug-in tests based upon a freely chosen uniformly convergent estimator in a general setting; see \Cref{sec:dist_plug_in}.
All these results are quantitative and non-asymptotic in nature and yield explicit bounds on the errors of the considered tests.

The plug-in tests that we use cannot provide minimax distinguishability for separation rates that are faster than the minimax rate for estimation. On the other hand, it is known that the minimax rate for testing may be lower in a nonparametric setting; see e.g. \cite[Section 2.10]{IngSus:2003}.
In order to further bound the minimax separation rate in the considered setting, one could study infimum $\chi^2$- and $U$-statistic tests, which are known to be minimax optimal in certain cases, and compare their performance to that of the infimum plug-in tests we have presented here.

\section*{Acknowledgements}

We thank Alexandra Carpentier and Niklas Hartung (University of Potsdam), Abhishake (Lappeenranta University of Technology), Markus Rei\ss~ (Humboldt University, Berlin), and Maximilian Siebel (University of Heidelberg) for helpful comments.
The research of the authors has been partially funded by the Deutsche Forschungsgemeinschaft (DFG) --- \href{https://gepris.dfg.de/gepris/projekt/318763901}{Project-ID 318763901} --- SFB1294 ``Data Assimilation''.

\bibliography{references}

@book{IngSus:2003,
	author = {Ingster, Yu. I. and Suslina, Irina A.},
	title = {Nonparametric Goodness-of-Fit Testing Under Gaussian Models},
	publisher = {Springer New York, NY},
	year = 2003,
	series = {Lecture Notes in Statistics},
	doi = {10.1007/978-0-387-21580-8},
}

@article{Ing:1984,
	author = {Ingster, Yu. I.},
	title = {Asymptotically Optimal {B}ayes Tests for Composite Hypotheses},
	fjournal = {Theory of Probability \& Its Applications},
	journal = {Theory Probab. Appl.},
	volume = {28},
	number = {4},
	pages = {775--794},
	year = {1984},
	doi = {10.1137/1128075},
}

@article{Ing:1986,
	author = {Ingster, Yu. I.},
	title = {Asymptotically Optimal Tests for Verifying Composite Finite-Parameter Hypotheses},
	fjournal = {Theory of Probability \& Its Applications},
	journal = {Theory Probab. Appl.},
	volume = {30},
	number = {2},
	pages = {311--332},
	year = {1986},
	doi = {10.1137/1130038},
}

@article{IngSapSus:2012,
	author = {Yuri I. Ingster and Theofanis Sapatinas and Irina A. Suslina},
	title = {Minimax signal detection in ill-posed inverse problems},
	volume = {40},
	fjournal = {The Annals of Statistics},
	journal = {Ann. Statist.},
	number = {3},
	pages = {1524--1549},
	year = {2012},
	doi = {10.1214/12-AOS1011},
}

@article{MarSap:2017,
	author = {Cl\'{e}ment Marteau and Theofanis Sapatinas},
	title = {Minimax goodness-of-fit testing in ill-posed inverse problems with partially unknown operators},
	volume = {53},
	fjournal = {Annales de l'Institut Henri Poincar\'{e}, Probabilit\'{e}s et Statistiques},
	journal = {Ann. Inst. Henri Poincar\'{e} Probab. Stat.},
	number = {4},
	publisher = {Institut Henri Poincar\'{e}},
	pages = {1675--1718},
	year = {2017},
	doi = {10.1214/16-AIHP768},
}

@article{Ray:2013,
	author = {Kolyan Ray},
	title = {{B}ayesian inverse problems with non-conjugate priors},
	volume = {7},
	fjournal = {Electronic Journal of Statistics},
	journal = {Elec. J. Stat.},
	number = {none},
	pages = {2516--2549},
	year = {2013},
	doi = {10.1214/13-EJS851},
}

@article{GhoGhoVaa:2000,
	 author = {Subhashis Ghosal and Jayanta K. Ghosh and Aad W. van der Vaart},
	 fjournal = {The Annals of Statistics},
	 journal = {Ann. Statist.},
	 number = {2},
	 pages = {500--531},
	 publisher = {Institute of Mathematical Statistics},
	 title = {Convergence rates of posterior distributions},
	 volume = {28},
	 year = {2000},
	 doi = {10.1214/aos/1016218228},
}

@book{GinNic:2016,
	author = {Evarist Gin\'{e} and Richard Nickl},
	title = {Mathematical Foundations of Infinite-Dimensional Statistical Models},
	year = 2016,
	publisher = {Cambridge University Press},
	series = {Cambridge Series in Statistical and Probabilistic Mathematics},
	doi = {10.1017/CBO9781107337862},
}

@unpublished{AbhHelMue:2023,
	title = {Statistical inverse learning problems with random observations},
	author = {Abhishake and Tapio Helin and Nicole Mücke},
	year = 2023,
	eprint = {2312.15341},
	eprinttype = {arxiv},
	note = {Preprint},
}

@unpublished{Lie:2024,
	title = {{B}ayesian inference of covariate-parameter relationships for population modelling},
	author = {Han Cheng Lie},
	year = 2024,
	eprint = {2407.09640},
	eprinttype = {arxiv},
	note = {Preprint},
}

@book{GhoVaa:2017,
	title = {Fundamentals of Nonparametric Bayesian Inference},
	author = {Subhasis Ghosal and van der Vaart, Aad},
	year = 2017,
	publisher = {Cambridge University Press},
	series = {Cambridge Series in Statistical and Probabilistic Mathematics},
	doi = {10.1017/9781139029834},
}

@book{Nic:2023,
	title = {{B}ayesian Non-linear Statistical Inverse Problems},
	author = {Richard Nickl},
	year = 2023,
	publisher = {European Mathematical Society (EMS) Press},
	series = {Zurich Lectures in Advanced Mathematics},
	doi = {10.4171/ZLAM/30},
}

@article{Siebel:2024,
	title = {Convergence rates for the maximum a posteriori estimator in {PDE}-regression models with random design},
	author = {Siebel, Maximilian},
	fjournal = {SIAM/ASA Journal on Uncertainty Quantification},
	journal = {SIAM/ASA J. Uncertain. Quantif.},
	volume = {13},
	number = {4},
	pages = {1862-1903},
	year = {2025},
	doi = {10.1137/25M1744526},
}

@article{HarWahRasHui:2021,
	title = {Nonparametric goodness-of-fit testing for parametric covariate models in pharmacometric analyses},
	author = {Niklas Hartung and Martin Wahl and Abhishake Rastogi and Wilhelm Huisinga},
	year = 2021,
	doi = {https://doi.org/10.1002/psp4.12614},
	fjournal = {Pharmacometrics \& Systems Pharmacology},
	journal = {CPT Pharmacometrics Syst. Pharmacol.},
	volume = 10,
	number = 6,
	pages = {564--576},
}

@article{GiorNic:2020,
doi = {10.1088/1361-6420/ab7d2a},
year = {2020},
volume = {36},
number = {8},
pages = {085001},
author = {Giordano, Matteo and Nickl, Richard},
title = {Consistency of {B}ayesian inference with {G}aussian process priors in an elliptic inverse problem},
fjournal = {Inverse Problems},
journal = {Inverse Probl.}
}

@article{NicvandeGeerWang:2020,
author = {Nickl, Richard and van de Geer, Sara and Wang, Sven},
title = {Convergence Rates for Penalized Least Squares Estimators in {PDE} Constrained Regression Problems},
fjournal = {SIAM/ASA Journal on Uncertainty Quantification},
journal = {SIAM/ASA J. Uncertain. Quantif.},
volume = {8},
number = {1},
pages = {374-413},
year = {2020},
doi = {10.1137/18M1236137},
}

@article{Kekk:2022,
doi = {10.1088/1361-6420/ac4839},
year = {2022},
volume = {38},
number = {3},
pages = {035002},
author = {Kekkonen, Hanne},
title = {Consistency of {B}ayesian inference with {G}aussian process priors for a parabolic inverse problem},
fjournal = {Inverse Problems},
journal = {Inverse Probl.}
}

@article{MonNicPat:2021,
author = {Monard, Fran\c{c}ois and Nickl, Richard and Paternain, Gabriel P.},
title = {Consistent Inversion of Noisy Non-Abelian X-Ray Transforms},
fjournal = {Communications on Pure and Applied Mathematics},
journal = {Comm. Pure Appl. Math.},
volume = {74},
number = {5},
pages = {1045-1099},
doi = {https://doi.org/10.1002/cpa.21942},
year = {2021},
}

@article{vanderVaavanZan:2000,
author = {A. W. van der Vaart and J. H. van Zanten},
title = {{Rates of contraction of posterior distributions based on Gaussian process priors}},
volume = {36},
fjournal = {The Annals of Statistics},
journal = {Ann. Statist.},
number = {3},
pages = {1435 -- 1463},
year = {2008},
doi = {10.1214/009053607000000613},
}

\appendix

\section{Deferred proofs}
\label{sec_deferred_proofs}

Recall from \Cref{subsection_notation} the notation $V\hookrightarrow W$ for the continuous embedding of one normed vector space $V$ into another vector space $W$, i.e. the existence of some $C>0$ such that $\norm{v}_W\leq C\norm{v}_V$ for every $v\in V$.
Statement \cref{embedding_L2zeta_L2} of \Cref{lemma_embedding} is used in the proof of \cref{theorem_uniform_posterior_contraction_Linfty} and in \Cref{remark_uniform_L2_convergence_posterior_mean}.
Statement \cref{embedding_L2_L2zeta} is used in the proof of \cref{Whittle_Matern_prior} and \cref{ode_ivp_satisfies_ip_assumptions}.
\begin{lemma}
\label{lemma_embedding}
Let $\calZ \subset \R^{d_z}$ be a bounded domain with smooth boundary and $\zeta$ be a probability measure on $\calZ$ with Lebesgue density $p_\zeta$, and let $W$ be a normed vector space.
\begin{enumerate}
\item \label{embedding_L2zeta_L2}
	 If $\essinf_{z\in\calZ}p_\zeta(z)>0$, then $L^2_\zeta(\calZ,W)\hookrightarrow L^2(\calZ,W)$.
\item \label{embedding_L2_L2zeta}
If $\norm{p_\zeta}_{L^\infty}<\infty$, then $L^2(\calZ,W)\hookrightarrow L^2_\zeta(\calZ,W)$.
\end{enumerate}
\end{lemma}
\begin{proof}[Proof of \Cref{lemma_embedding}]
\cref{embedding_L2zeta_L2}: Let $u \in L^2_\zeta(\calZ,W)$ be arbitrary. Then
\begin{align*}
	\norm{u}_{L^2_\zeta(\calZ)}^2 &= \int_\calZ \norm{u(z)}_W^2 p_\zeta(z) \di z \ge \essinf_{z \in \calZ} \left\{p_\zeta(z)\right\} \int_\calZ \norm{u(z)}_W^2 \di z
	= \essinf_{z \in \calZ} \left\{p_\zeta(z)\right\} \norm{u}_{L^2(\calZ)}^2,
\end{align*}
thus proving $L^2_\zeta(\calZ)\hookrightarrow L^2(\calZ)$.
\\
\cref{embedding_L2_L2zeta}:  Let $u \in L^2(\calZ,W)$ be arbitrary. Then
\begin{align*}
	\norm{u}_{L^2_\zeta}^2 &= \int_\calZ \norm{u(z)}_W^2 p_\zeta(z) \di z  \le \norm{p_\zeta}_{L^\infty} \int_\calZ \norm{u(z)}_W^2 \di z=\norm{p_\zeta}_{L^\infty}\norm{u}_{L^2}^2,
\end{align*}
thus proving $L^2(\calZ,W)\hookrightarrow L^2_\zeta(\calZ,W)$.
\end{proof}

\subsection{Proof of forward posterior contraction result}
\label{sec_proof_forward_posterior_contraction_result}

In this section, we prove the forward posterior contraction result \Cref{theorem_Nickl_2_2_2}.
To do so, we shall also use \Cref{lemma_existence_of_tests}. The statement of \Cref{lemma_existence_of_tests} and a sketch of its proof are given in the text between equations (1.28) and (1.29) on \cite[p. 19]{Nic:2023}.

  We denote the Hellinger distance between two probability density functions $p_1$ and $p_2$ with respect to a common dominating measure by $h(p_1,p_2)$. 
  Recall the set $\Theta_N$ defined in \eqref{eq_Nickl_2_20}.
  \begin{lemma}[Existence of tests]
   \label{lemma_existence_of_tests}
   Suppose that \Cref{assumption_Nickl_Condition_2_1_1} holds, and let $L$ and $(\delta_N)_{N\in\N}$ be as in \eqref{eq_Nickl_2_4} and \eqref{eq_Nickl_2_19} respectively.
   Suppose $M$ and $\overline{m}=\overline{m}(M,L)$ satisfy \eqref{eq_Nickl_1_26} for every $N\in\N$ and let $\theta_0\in B_{\calR}(M)$. 
   Then there exists a universal constant $0<K<\infty$ such that for every $m>2\overline{m}$, there exists for every $N\in\N$ a test $\Psi_N:(V\times \X)^N\to\{0,1\}$ such that
   \begin{equation}
    \label{eq_Nickl_1_29}
    \begin{aligned}
   P^N_{\theta_0}(\Psi_N=1) &\leq \frac{1}{K}\exp(-(Km^2-1)N\delta_N^2),
   \\
   \sup_{\theta\in\Theta_N(M):h(p_\theta,p_{\theta_0})>m\delta_N} P^N_\theta(\Psi_N=0) &\leq \exp(-Km N\delta_N^2).
    \end{aligned}
    \end{equation}
  \end{lemma}
  For the type $1$ and type $2$ error bounds in the conclusion \eqref{eq_Nickl_1_29} of \Cref{lemma_existence_of_tests} to be nontrivial, it suffices to impose the additional condition that $Km^2-1>0$ on $m$.
  
The proof of \cref{lemma_existence_of_tests} uses the tests defined in \cite[Theorem 7.1.4]{GinNic:2016}. These tests are constructed by combining tests for ball-shaped hypotheses against ball-shaped alternatives \cite[Corollary 7.1.3]{Nic:2023}, and by covering the original alternative by such balls.
Although the definition of these tests is constructive, it is mostly of theoretical use, since it is generally not feasible to construct these tests in practice by finding a ball cover of the alternative.

   \begin{proof}[Proof of \Cref{lemma_existence_of_tests}]
		The proof consists in verifying the hypotheses of \cite[Theorem 7.1.4]{GinNic:2016}.
		Recall the definition \eqref{eq_Nickl_1_10} of $p_\theta$. 
	Since \Cref{assumption_Nickl_Condition_2_1_1} holds, we may apply \cite[Proposition 1.3.1]{Nic:2023} with the substitutions $\Theta\leftarrow\Theta_N(M)$ and $U\leftarrow U(M)$, for the envelope constant $U(M)$ from \eqref{eq_Nickl_2_3} in \Cref{assumption_Nickl_Condition_2_1_1}, to conclude that the Hellinger distance $h(p_\theta,p_{\vartheta})$ satisfies 
   \begin{equation}
    \label{eq_Nickl_1_20}
    C_U^{1/2} d_G(\theta,\vartheta)\leq h(p_\theta,p_{\vartheta})\leq \frac{1}{2} d_G(\theta,\vartheta),\quad\forall \theta,\vartheta\in\Theta_N(M)
   \end{equation}
	for $C_U$ given by \eqref{eq_C_of_U_function} evaluated at the envelope constant $U=U(M)$, and $d_G$ the semimetric defined in \eqref{semimetric_dG}.
	In particular, by the upper bound in \eqref{eq_Nickl_1_20}, $d_G(\theta,\vartheta)\leq \overline{m}\delta_N$ implies that $h(p_\theta,p_{\vartheta})\leq \tfrac{1}{2}\overline{m}\delta_N$.
	Thus, by the condition \eqref{eq_Nickl_1_26},
	\begin{equation}   
	\label{eq_intermed10a}
	 N(\{p_\theta: \theta\in\Theta_N(M)\},h,\tfrac{1}{2}\overline{m}\delta_N) \leq N(\Theta_N(M),d_G,\overline{m}\delta_N)\leq \exp(N\delta_N^2).
	\end{equation}
	Note that for every $\varepsilon_0>0$ and every $\varepsilon>\varepsilon_0$,
	\begin{align}   
	 &N(\{p_\theta: \theta\in\Theta_N(M), \varepsilon<h(p_\theta,p_{\theta_0})\leq 2\varepsilon\},h,\tfrac{\varepsilon}{4})
	 \nonumber\\
	 \leq & N(\{p_\theta: \theta\in\Theta_N(M)\},h,\tfrac{\varepsilon}{4})
	 \nonumber\\
	 \leq &N(\{p_\theta: \theta\in\Theta_N(M)\},h,\tfrac{\varepsilon_0}{4})
	 \label{eq_intermed10b}
	\end{align}
	where the first inequality follows from the fact that $\{p_\theta: \theta\in\Theta_N(M), \varepsilon<h(p_\theta,p_{\theta_0})\leq 2\varepsilon\}$ is contained in $\{p_\theta: \theta\in\Theta_N(M)\}$ and the definition of the covering number, and the second inequality follows from $\varepsilon>\varepsilon_0$ and the definition of the covering number.
	Using \eqref{eq_intermed10b} with $\varepsilon_0\leftarrow 2\overline{m}\delta_N$ and using \eqref{eq_intermed10a}, the preceding argument yields
	\begin{equation*}   
	 N(\{p_\theta: \theta\in\Theta_N(M), \varepsilon<h(p_\theta,p_{\theta_0})\leq 2\varepsilon\},h,\tfrac{\varepsilon}{4})\leq \exp(N\delta_N^2).
	\end{equation*}
	In particular, the hypotheses of \cite[Theorem 7.1.4]{GinNic:2016} hold with the set of densities $\calP\leftarrow \{p_\theta:\theta\in\Theta_N(M)\}$, the constant and thus nonincreasing function $\varepsilon\mapsto N(\varepsilon)\coloneqq \exp(N\delta_N^2)$, and $\varepsilon_0$ as above.
	By applying the conclusion of \cite[Theorem 7.1.4]{GinNic:2016} with $\varepsilon\leftarrow m\delta_N$, $m>2\overline{m}$, there exists for every $N\in\N$ a test $\Psi_N:(V\times\X)^N\to\{0,1\}$ for the hypothesis testing problem
	\begin{equation*}   
	 H_0\coloneqq\{p_{\theta_0}\}\quad\text{vs.}\quad H_1\coloneqq\{p_\theta: \theta\in\Theta_N(M), h(p,p_{\theta_0})>m\delta_N\}
	\end{equation*}
	such that for some universal constant $0<K<\infty$ and every $N\in\N$,
	\begin{equation*}   
	 P^N_{\theta_0}(\Psi_N=1)\leq \frac{1}{K}\exp((1-Km^2)N\delta_N^2),\quad \sup_{p_\theta: \theta\in\Theta_N(M), h(p,p_{\theta_0})>m\delta_N} E^N_\theta(1-\Psi_N)\leq \exp(-Km^2N\delta_N^2);
	\end{equation*}
	cf. \cite[equation (1.28)]{Nic:2023}.
	This completes the proof of \Cref{lemma_existence_of_tests}.
   \end{proof}
   
   \forwardPosteriorContraction*
   We now prove \Cref{theorem_Nickl_2_2_2}.
   We reproduce the argument given in the proof of \cite[Theorem 1.3.2]{Nic:2023} in \cite[pp. 19-20]{Nic:2023} and track the dependence on $N\delta_N^2$ of certain upper bounds, with the goal of showing that \eqref{eq_Nickl_1_27} holds.
  \begin{proof}[Proof of \Cref{theorem_Nickl_2_2_2}]
   The hypotheses of \Cref{theorem_Nickl_2_2_2} are those of \cite[Theorem 2.2.2]{Nic:2023}, and $\Theta_N$ in \eqref{eq_Nickl_2_20} is exactly the same as the `regularisation sets' defined in \cite[equation (2.20)]{Nic:2023}.
   Thus, by \cite[Theorem 2.2.2]{Nic:2023}, for $\theta_0\in\calH\cap\calR$, the statements \eqref{eq_Nickl_1_24}, \eqref{eq_Nickl_1_25} and \eqref{eq_Nickl_1_26} hold for some $A>0$, $\delta_N$ as in \eqref{eq_Nickl_2_19} and any $B>A+2$, provided that $M=M(B)$ is sufficiently large. It remains to prove \eqref{eq_Nickl_1_27}.
   	   
Recall the definition \eqref{eq_Nickl_1_23} of $\mathcal{B}_N$.
Define
	\begin{equation}
 \label{eq_Nickl_1_28}
 \mathcal{A}_N=\mathcal{A}_N(\theta_0)\coloneqq \left\{D_N : \int_{\Theta} e^{\ell_N(\theta)-\ell_N(\theta_0)}\mathrm{d}\Pi_N(\theta)\geq e^{-(A+2)N\delta_N^2} \right\}.
\end{equation}
The following implication
\begin{align*}
  & & \int_{\mathcal{B}_N(\theta_0,U)} e^{\ell_N(\theta)-\ell_N(\theta_0)}\mathrm{d}\Pi_N(\theta)\geq & e^{-2N\delta_N^2 } \Pi_N(\mathcal{B}_N(\theta_0,U))
  \\
  \Longrightarrow & & \int_{\Theta} e^{\ell_N(\theta)-\ell_N(\theta_0)}\mathrm{d}\Pi(\theta) \geq &  e^{-(A+2)N\delta_N^2}
\end{align*}
holds, because $\Theta\supset \mathcal{B}_N(\theta_0,U)$, because $e^{\ell_N(\theta)-\ell_N(\theta_0)}$ is nonnegative on $\Theta$, and because \eqref{eq_Nickl_1_24} holds by hypothesis.
Thus, by the definition \eqref{eq_Nickl_1_28} of $\mathcal{A}_N$,
\begin{align*}
 &P^N_{\theta_0}(\mathcal{A}_N)\geq P^N_{\theta_0}\left(\int_{\mathcal{B}_N(\theta_0,U)} e^{\ell_N(\theta)-\ell_N(\theta_0)}\mathrm{d}\Pi_N(\theta)\geq e^{-2N\delta_N^2 } \Pi_N(\mathcal{B}_N(\theta_0,U)) \right)
 \\
 \Longleftrightarrow &P^N_{\theta_0}(\mathcal{A}_N^\complement)\leq P^N_{\theta_0}\left(\int_{\mathcal{B}_N(\theta_0,U)} e^{\ell_N(\theta)-\ell_N(\theta_0)}\mathrm{d}\Pi_N(\theta)< e^{-2N\delta_N^2 } \Pi_N(\mathcal{B}_N(\theta_0,U)) \right).
\end{align*}
Since the hypotheses of  \Cref{theorem_Nickl_2_2_2} include the hypotheses of \Cref{lemma_1_3_3_Nickl}, we may apply \Cref{lemma_1_3_3_Nickl} with $B_N\leftarrow\mathcal{B}_N(\theta_0,U)$, $\nu(\cdot)\leftarrow \Pi_N(\cdot\cap \mathcal{B}_N(\theta_0,U))/\Pi_N(\mathcal{B}_N(\theta_0,U))$, $K\leftarrow 2$, and $U$.
The conclusion \eqref{eq_lemma_1_3_3_Nickl_conclusion} of \Cref{lemma_1_3_3_Nickl} then implies
\begin{equation*}
P^N_{\theta_0}\left(\int_{\mathcal{B}_N(\theta_0,U)} e^{\ell_N(\theta)-\ell_N(\theta_0)}\mathrm{d}\Pi_N(\theta)\leq e^{-2N\delta_N^2}\Pi_N(\mathcal{B}_N(\theta_0,U))\right)\leq \frac{8 (U^2+1) }{4N\delta_N^2}.
\end{equation*}
Combining the preceding two inequalities yields
\begin{equation}   
	\label{eq_intermed02}
	 P^N_{\theta_0}(\mathcal{A}_N^\complement)\leq \frac{2(U^2+1) }{N\delta_N^2}.
	\end{equation}
   
   Recall that \Cref{assumption_Nickl_Condition_2_1_1} holds. 
   By the hypothesis that $\theta_0\in \calH \cap\calR$ is fixed, there exists $M$ sufficiently large such that $\theta_0\in B_{\calR}(M)$. 
   By increasing $M$ if necessary, the argument in the first paragraph of the proof also ensures that there exists some $\overline{m}>0$ such that \eqref{eq_Nickl_1_26} holds for every $N\in\N$. Hence, we may apply \Cref{lemma_existence_of_tests} to conclude that for some universal constant $0<K<\infty$ and for every $m>2\overline{m}$ such that $Km^2-1>0$, there exists for every $N\in\N$ a test $\Psi_N:(V\times\X)^N\to\{0,1\}$ such that \eqref{eq_Nickl_1_29} holds. 
   Let $m$ belong to a range of values that we shall specify towards the end of the proof. Define
   \begin{equation}
   \label{eq_intermed03}
   \overline{\Theta}_N\coloneqq \Theta_N\cap \{\theta\in\Theta:h(p_\theta,p_{\theta_0})\leq m\delta_N\},\quad \overline{\Theta}_N^\complement\coloneqq \Theta\setminus\overline{\Theta}_N
   \end{equation}
   where $\Theta\subseteq L^2_\zeta$ is the set from \Cref{assumption_Nickl_Condition_2_2_1}.
	Let $0<b<B-(A+2)$ for $B$ in \eqref{eq_Nickl_1_25} and $A$ in \eqref{eq_Nickl_1_24}.
		Now
	\begin{align*}   
	 & P^N_{\theta_0}\left(\Pi_N\left(\overline{\Theta}_N^\complement\vert D_N\right)\geq e^{-bN\delta_N^2}\right)=P^N_{\theta_0}\left(\frac{\int_{\overline{\Theta}_N^\complement} e^{\ell_N(\theta)-\ell_N(\theta_0)}\mathrm{d}\Pi_N(\theta)}{\int_{\Theta} e^{\ell_N(\theta')-\ell_N(\theta_0)}\mathrm{d}\Pi_N(\theta')} \geq e^{-bN\delta_N^2}\right)
	 \\
	 \leq & P^N_{\theta_0}\left(\frac{\int_{\overline{\Theta}_N^\complement} e^{\ell_N(\theta)-\ell_N(\theta_0)}\mathrm{d}\Pi_N(\theta)}{\int_{\Theta} e^{\ell_N(\theta')-\ell_N(\theta_0)}\mathrm{d}\Pi_N(\theta')} \geq e^{-bN\delta_N^2},\mathcal{A}_N\right)+P^N_{\theta_0}(\mathcal{A}_N^\complement)
	 \\
	 \leq & P^N_{\theta_0}\left(\frac{\int_{\overline{\Theta}_N^\complement} e^{\ell_N(\theta)-\ell_N(\theta_0)}\mathrm{d}\Pi_N(\theta)}{\int_{\Theta} e^{\ell_N(\theta')-\ell_N(\theta_0)}\mathrm{d}\Pi_N(\theta')} \geq e^{-bN\delta_N^2},\Psi_N=0,\mathcal{A}_N\right)+P^N_{\theta_0}(\Psi_N=1,\mathcal{A}_N)+P^N_{\theta_0}(\mathcal{A}_N^\complement)
	 \\
	  \leq & P^N_{\theta_0}\left(\frac{\int_{\overline{\Theta}_N^\complement} e^{\ell_N(\theta)-\ell_N(\theta_0)}\mathrm{d}\Pi_N(\theta)}{\int_{\Theta} e^{\ell_N(\theta')-\ell_N(\theta_0)}\mathrm{d}\Pi_N(\theta')} \geq e^{-bN\delta_N^2},\Psi_N=0,\mathcal{A}_N\right)+P^N_{\theta_0}(\Psi_N=1)+P^N_{\theta_0}(\mathcal{A}_N^\complement)
	\end{align*}
	where the equation follows from \eqref{eq_Nickl_1_12}, and the inequalities follow from the monotonicity of probability measures.
	The second and third terms on the right-hand side of \eqref{eq_intermed05} are bounded according to the type $1$ error bound in the conclusion \eqref{eq_Nickl_1_29} of \Cref{lemma_existence_of_tests} and according to \eqref{eq_intermed02} respectively:
	\begin{align}
	P^N_{\theta_0}\left(\Pi_N\left(\overline{\Theta}_N^\complement\vert D_N\right)\geq e^{-bN\delta_N^2}\right)	 \leq &P^N_{\theta_0}\left(\frac{\int_{\overline{\Theta}_N^\complement} e^{\ell_N(\theta)-\ell_N(\theta_0)}\mathrm{d}\Pi_N(\theta)}{\int_{\Theta} e^{\ell_N(\theta')-\ell_N(\theta_0)}\mathrm{d}\Pi_N(\theta')} \geq e^{-bN\delta_N^2},\Psi_N=0,\mathcal{A}_N\right)
	\label{eq_intermed05}
	\\
	 &+\frac{1}{K}\exp(-(Km^2-1)N\delta_N^2)+\frac{2(U^2+1)}{N\delta_N^2}.
	\nonumber
	\end{align}
	We now bound the first term on the right-hand side of the inequality \eqref{eq_intermed05} according to
	\begin{align*}   
	 & P^N_{\theta_0}\left(\frac{\int_{\overline{\Theta}_N^\complement} e^{\ell_N(\theta)-\ell_N(\theta_0)}\mathrm{d}\Pi_N(\theta)}{\int_{\Theta} e^{\ell_N(\theta')-\ell_N(\theta_0)}\mathrm{d}\Pi_N(\theta')} \geq e^{-bN\delta_N^2},\Psi_N=0,\mathcal{A}_N\right)
	 \\
	 =&P^N_{\theta_0}\left(\int_{\overline{\Theta}_N^\complement} e^{\ell_N(\theta)-\ell_N(\theta_0)}\mathrm{d}\Pi_N(\theta) \geq e^{-bN\delta_N^2}\int_{\Theta} e^{\ell_N(\theta')-\ell_N(\theta_0)}\mathrm{d}\Pi_N(\theta'),\Psi_N=0,\mathcal{A}_N\right)
	 \\
	 \leq & P^N_{\theta_0}\left(\int_{\overline{\Theta}_N^\complement} e^{\ell_N(\theta)-\ell_N(\theta_0)}\mathrm{d}\Pi_N(\theta) \geq e^{-bN\delta_N^2}e^{-(A+2)N\delta_N^2},\Psi_N=0,\mathcal{A}_N\right)
	 \\
	 \leq & P^N_{\theta_0}\left(\int_{\overline{\Theta}_N^\complement} e^{\ell_N(\theta)-\ell_N(\theta_0)}\mathrm{d}\Pi_N(\theta)(1-\Psi_N) \geq e^{-bN\delta_N^2}e^{-(A+2)N\delta_N^2},\Psi_N=0\right)
	 \\
	 \leq &  P^N_{\theta_0}\left(\int_{\overline{\Theta}_N^\complement} e^{\ell_N(\theta)-\ell_N(\theta_0)}\mathrm{d}\Pi_N(\theta)(1-\Psi_N) \geq e^{-(b+A+2)N\delta_N^2}\right)
	 \\
	 \leq & E^N_{\theta_0}\left[ \int_{\overline{\Theta}_N^\complement} e^{\ell_N(\theta)-\ell_N(\theta_0)}\mathrm{d}\Pi_N(\theta)(1-\Psi_N)\right] e^{(b+A+2)N\delta_N^2}
	 \\
	 =&  \int_{\overline{\Theta}_N^\complement}E^N_{\theta_0}\left[ e^{\ell_N(\theta)-\ell_N(\theta_0)}(1-\Psi_N)\right]\mathrm{d}\Pi_N(\theta) e^{(b+A+2)N\delta_N^2},
	\end{align*}
	where the first inequality follows from the definition \eqref{eq_Nickl_1_28} of $\mathcal{A}_N$,
	the second and third inequalities follow by the monotonicity of probability measures, the fourth inequality follows from Markov's inequality, and the final equation follows from Tonelli's theorem.
	
	Next, we bound $\int_{\overline{\Theta}_N^\complement}E^N_{\theta_0}[ e^{\ell_N(\theta)-\ell_N(\theta_0)}(1-\Psi_N)\mathrm{d}\Pi_N(\theta)$.
	By the definitions of $\overline{\Theta}_N$ and $\overline{\Theta}_N^\complement$ in \eqref{eq_intermed03},
	\begin{equation}   
	\label{eq_intermed04}
	 \overline{\Theta}_N^\complement=\Theta_N^\complement\cup \{\theta\in\Theta_N~\vert~h(p_\theta,p_{\theta_0})>m\delta_N\}.
	\end{equation}
	By \eqref{eq_Nickl_1_10} and \eqref{eq_Nickl_1_13},
	\begin{equation}   
	\label{eq_Nickl_1_30}
	 E^N_{\theta_0}\left[ e^{\ell_N(\theta)-\ell_N(\theta_0)}(1-\Psi_N)\right]=E^N_{\theta_0}\left[ \prod_{i=1}^{N}\frac{p_\theta}{p_{\theta_0}}(1-\Psi_N)\right]=E^N_\theta\left[1-\Psi_N\right].
	\end{equation}
	Now \eqref{eq_intermed04} and \eqref{eq_Nickl_1_30} imply that
	\begin{align*}   
	 &\int_{\overline{\Theta}_N^\complement}E^N_{\theta_0}\left[ e^{\ell_N(\theta)-\ell_N(\theta_0)}(1-\Psi_N)\right]\mathrm{d}\Pi_N(\theta) 
	 \\
	 =&\int_{\Theta_N^\complement} E^N_{\theta}\left[ 1-\Psi_N\right]\mathrm{d}\Pi_N(\theta)+\int_{\{\theta\in\Theta_N:h(p_\theta,p_{\theta_0})>m\delta_N\}} E^N_{\theta}\left[1-\Psi_N\right]\mathrm{d}\Pi_N(\theta)
	 \\
	 \leq& \Pi_N(\Theta_N^\complement)+\sup_{\{\theta\in\Theta_N:h(p_\theta,p_{\theta_0})>m\delta_N\}} E^N_{\theta}\left[1-\Psi_N\right],
	\end{align*}
	where the first term on the right-hand side of the inequality follows since $\Psi_N$ takes values in $\{0,1\}$, and the second term follows since $\Pi_N(\Theta_N)\leq 1$.
	By \eqref{eq_Nickl_1_25}, $\Pi_N(\Theta_N^\complement)\leq e^{-BN\delta_N^2}$, provided the radius parameter $M$ that defines $\Theta_N(M)$ in \eqref{eq_Nickl_2_20} is large enough. 
	Furthermore, we have 
	\begin{align*}   
	 \sup_{\{\theta\in\Theta_N:h(p_\theta,p_{\theta_0})>m\delta_N\}} E^N_{\theta}\left[1-\Psi_N\right]= \sup_{\{\theta\in\Theta_N:h(p_\theta,p_{\theta_0})>m\delta_N\}} P^N_{\theta}(\Psi_N=0)\leq  \exp(-Km^2N\delta_N^2)
	\end{align*}
	where the equation follows since $\Psi_N$ is $\{0,1\}$-valued.
	The inequality follows from the bound on the type $2$ errors in \eqref{eq_Nickl_1_29}. 
	Thus, it follows that the first term on the right-hand side of \eqref{eq_intermed05} satisfies
	\begin{align}   
	 &P^N_{\theta_0}\left(\frac{\int_{\overline{\Theta}_N^\complement} e^{\ell_N(\theta)-\ell_N(\theta_0)}\mathrm{d}\Pi_N(\theta)}{\int_{\Theta} e^{\ell_N(\theta')-\ell_N(\theta_0)}\mathrm{d}\Pi_N(\theta')} \geq e^{-bN\delta_N^2},\Psi_N=0,\mathcal{A}_N\right)
	 \nonumber\\
	 \leq & (e^{-BN\delta_N^2}+e^{-Km^2N\delta_N^2})e^{(b+A+2)N\delta_N^2}=e^{(b+A+2-B)N\delta_N^2}+e^{(b+A+2-Km^2)N\delta_N^2},
	 \label{eq_intermed04a}
	\end{align}
	Note that $\Pi_N(\overline{\Theta}_N\vert D_N)=1-\Pi_N(\overline{\Theta}_N^\complement\vert D_N)$, since  \Cref{assumption_Nickl_Condition_2_2_1} implies $\Pi_N(\Theta)=1$, and since $\overline{\Theta}_N^\complement\coloneq\Theta\setminus\overline{\Theta}_N$ by \eqref{eq_intermed03}. 
	Hence, 
    \begin{align}
     &P^N_{\theta_0}\left( \Pi_N\left(\overline{\Theta}_N\vert D_N\right)\leq 1-e^{-bN\delta_N^2}\right)
     \nonumber
     \\
     =&P^N_{\theta_0}\left(\Pi_N\left(\overline{\Theta}_N^\complement\vert D_N\right)\geq e^{-bN\delta_N^2}\right)	
     \nonumber
     \\
     \leq &P^N_{\theta_0}\left(\frac{\int_{\overline{\Theta}_N^\complement} e^{\ell_N(\theta)-\ell_N(\theta_0)}\mathrm{d}\Pi_N(\theta)}{\int_{\Theta} e^{\ell_N(\theta')-\ell_N(\theta_0)}\mathrm{d}\Pi_N(\theta')} \geq e^{-bN\delta_N^2},\Psi_N=0,\mathcal{A}_N\right)+\frac{1}{K}e^{-(Km^2-1)N\delta_N^2}+\frac{2(U^2+1)}{N\delta_N^2}
	 \nonumber
	 \\
     \leq & e^{(b+A+2-B)N\delta_N^2}+e^{(b+A+2-Km^2)N\delta_N^2}+\frac{1}{K}e^{-(Km^2-1)N\delta_N^2}+\frac{2(U^2+1)}{N\delta_N^2}
     \label{eq_intermed04b}
    \end{align}
	where the first and second inequalities follow by \eqref{eq_intermed05} and \eqref{eq_intermed04a} respectively.	
	
	By the hypothesis that $0<b<B-(A+2)$ in \Cref{theorem_Nickl_2_2_2}, the first term on the right-hand side of the second inequality decreases exponentially in $N\delta_N^2$ to zero as $N\to\infty$.
	We now define the range of values of $m$ for \eqref{eq_intermed03}: in addition to $m>2\overline{m}$ and $Km^2-1>0$, we further require $b+A+2-Km^2<0$. These conditions are all satisfied for 
	\begin{equation*}
	 m>\max\left\{2\overline{m},\frac{1}{\sqrt{K}},\sqrt{\frac{b+A+2}{K}}\right\}.
	\end{equation*}
	For any such $m$, the second term on the right-hand side of \eqref{eq_intermed04b} decreases exponentially in $N\delta_N^2$ to zero.
	Thus, for every $c>0$, there exists $N^*=N^*(c)\in\N$ such that for every $N\geq N^*$,
	\begin{align}
	\label{eq_intermed06}
     P^N_{\theta_0}\left(\Pi_N\left(\overline{\Theta}_N\vert D_N\right)\leq 1- e^{-bN\delta_N^2}\right)
     \leq (1+c)\frac{2(U^2+1)}{N\delta_N^2}.
	\end{align}	
	To complete the proof of \Cref{theorem_Nickl_2_2_2}, we recall from  \eqref{eq_intermed03} that 
	$\overline{\Theta}_N=\{\theta\in\Theta_N:h(p_\theta,p_{\theta_0})\leq m\delta_N\}$.
	Next, recall the inequality $ C_U^{1/2} d_G(\theta,\vartheta)\leq h(p_\theta,p_{\vartheta})$ from \eqref{eq_Nickl_1_20}, where $C_U$ is given by \eqref{eq_C_of_U_function} with $u\leftarrow U=U(M)$, and $d_G$ is the semimetric defined in \eqref{semimetric_dG}.
	By this inequality,
	\begin{align*}   
	 \overline{\Theta}_N=\{\theta\in\Theta_N:h(p_\theta,p_{\theta_0})\leq m\delta_N\}\subseteq & \{\theta\in\Theta_N:d_G(\theta,\theta_0)\leq m\delta_N C_U^{-1/2}\}.
	\end{align*}
	The set relation above implies that 
	\begin{align*}   
		&\left\{D_N:\Pi_N\left( \{\theta\in\Theta_N:d_G(\theta,\theta_0)\leq m\delta_N C_U^{-1/2}\}\biggr\vert D_N\right)\leq 1-e^{-bN\delta_N^2}\right\}
		\\
		\subseteq &\left\{ D_N:\Pi_N(\overline{\Theta}_N\vert D_N)\leq 1-e^{-bN\delta_N^2}\right\}.
	\end{align*}
	Combining the set relation above with \eqref{eq_intermed06}, we conclude that for every $0<c<1$, there exists $N^*(c)\in\N$ such that for every $N\geq N^*$,
	\begin{equation*}   
	 P^N_{\theta_0}\left(\Pi_N\left(\{\theta\in\Theta_N:d_G(\theta,\theta_0)\leq m\delta_N C_U^{-1/2}\vert D_N\right)\leq 1-e^{-bN\delta_N^2}\right)	\leq (1+c)\frac{2(U^2+1)}{N\delta_N^2}.
	\end{equation*}
	This proves \eqref{eq_Nickl_1_27}. 
  \end{proof}

\subsection{Proof that uniform posterior contraction implies uniform convergence of posterior mean}
\label{sec_proof_uniform_posterior_contraction_implies_uniform_convergence_posterior_mean}

In this section, we prove \Cref{theorem_uniform_posterior_contraction_implies_uniform_convergence_posterior_mean}.
For the proof of \Cref{theorem_uniform_posterior_contraction_implies_uniform_convergence_posterior_mean}, we will use \Cref{lemma_probability_of_A_N_events_converges_to_1_uniformly} below.
To prove \Cref{lemma_probability_of_A_N_events_converges_to_1_uniformly}, we will use \Cref{lemma_1_3_3_Nickl}.
Recall the definitions \eqref{eq_Nickl_1_13} and \eqref{eq_Nickl_1_23} of $\ell_N$ and $\mathcal{B}_N$ respectively.
\begin{lemma}[{\cite[Lemma 1.3.3]{Nic:2023}}]
 \label{lemma_1_3_3_Nickl}
Suppose \eqref{eq_Nickl_2_3} from \Cref{assumption_Nickl_Condition_2_1_1} holds with envelope constant $U=U(M)$, where $M$ is such that $\theta_0\in\mathcal{B}_{\calR}(M)$.
Let $\nu$ be a probability measure on a measurable subset $B_N$ of the set $\mathcal{B}_N(\theta_0,U)$ defined in \eqref{eq_Nickl_1_23} with $\mathcal{U}\leftarrow U$. Then for every $K\geq 2$ and $N\in\N$,
 \begin{equation}
 \label{eq_lemma_1_3_3_Nickl_conclusion}
 P^N_{\theta_0}\left(\int_{B_N} e^{\ell_N(\theta)-\ell_N(\theta_0)}\mathrm{d}\nu(\theta)\leq e^{-KN\delta_N^2}\right)\leq \frac{8 (U^2+1) }{K^2N\delta_N^2}.
 \end{equation}
\end{lemma}

The statement of \Cref{lemma_1_3_3_Nickl} differs from the statement of \cite[Lemma 1.3.3]{Nic:2023} in certain aspects.
First, in the statement of \cite[Lemma 1.3.3]{Nic:2023}, the constant $U$ is not explicitly described, whereas above we specify $U$ to be the envelope constant from \eqref{eq_Nickl_2_3} in \Cref{assumption_Nickl_Condition_2_1_1} for a suitably large $M$.
Second, \cite[Lemma 1.3.3]{Nic:2023} does not assume that \Cref{assumption_Nickl_Condition_2_1_1} holds.
We shall apply \Cref{lemma_1_3_3_Nickl} to prove theorems for which \Cref{assumption_Nickl_Condition_2_1_1} is stated as a hypothesis, so there is no loss of generality.

Recall that $\ell_N$ defined in \eqref{eq_Nickl_1_13} depends on the data $D_N=(Y_i,X_i)_{i=1}^{N}$.
\begin{lemma}
 \label{lemma_probability_of_A_N_events_converges_to_1_uniformly}
Suppose that the hypotheses of \Cref{theorem_uniform_posterior_contraction_implies_uniform_convergence_posterior_mean} hold, and define
\begin{equation}
 \label{eq_Nickl_1_28_uniform}
A_N=A_N(\theta_0)\coloneqq \left\{D_N : \int_{\Theta} e^{\ell_N(\theta)-\ell_N(\theta_0)}\mathrm{d}\Pi(\theta)\geq e^{-(\mathbf{A}+2)N\delta_N^2} \right\}
\end{equation}
for $\mathbf{A}$ as in \eqref{eq_uniform_constant_A}.
Then for sufficiently large $N$,
\begin{equation*}
 P^N_{\theta_0}(A_N^\complement)\leq \frac{2(U^2+1) }{N\delta_N^2},
\end{equation*}
where $U=U(M)$ is the envelope constant from \eqref{eq_Nickl_2_3} of \Cref{assumption_Nickl_Condition_2_1_1} corresponding to any $M\geq \sup_{\theta_0\in S}\norm{\theta_0}_{\calR}$.
\end{lemma}
  The event $A_N$ in \eqref{eq_Nickl_1_28_uniform} is the same as the event $A_N$ defined in \cite[equation (1.28)]{Nic:2023}, except that the lower bound is $e^{-(\mathbf{A}+2)N\delta_N^2}$ instead of $e^{-(A+2)N\delta_N^2}$ for $A=A(\theta_0)$ satisfying \eqref{eq_Nickl_1_24}.
\begin{proof}[Proof of \Cref{lemma_probability_of_A_N_events_converges_to_1_uniformly}]
Recall the definition \eqref{eq_Nickl_1_23} of $\mathcal{B}_N$.
If 
\begin{equation*}
  \int_{\mathcal{B}_N(\theta_0,U)} e^{\ell_N(\theta)-\ell_N(\theta_0)}\mathrm{d}\Pi_N(\theta)\geq e^{-2N\delta_N^2 } \Pi_N(\mathcal{B}_N(\theta_0,U))
\end{equation*}
then
\begin{equation*}
 \int_{\Theta} e^{\ell_N(\theta)-\ell_N(\theta_0)}\mathrm{d}\Pi(\theta) \geq  e^{-(\mathbf{A}+2)N\delta_N^2},
\end{equation*}
because $\Theta\supset \mathcal{B}_N(\theta_0,U)$, $e^{\ell_N(\theta)-\ell_N(\theta_0)}$ is nonnegative on $\Theta$, and because \eqref{eq_uniform_constant_A} holds by hypothesis.
Thus, by the definition \eqref{eq_Nickl_1_28_uniform} of $A_N$,
\begin{align*}
 &P^N_{\theta_0}(A_N)\geq P^N_{\theta_0}\left(\int_{\mathcal{B}_N(\theta_0,U)} e^{\ell_N(\theta)-\ell_N(\theta_0)}\mathrm{d}\Pi_N(\theta)\geq e^{-2N\delta_N^2 } \Pi_N(\mathcal{B}_N(\theta_0,U)) \right)
 \\
 \Longleftrightarrow\ &P^N_{\theta_0}(A_N^\complement)\leq P^N_{\theta_0}\left(\int_{\mathcal{B}_N(\theta_0,U)} e^{\ell_N(\theta)-\ell_N(\theta_0)}\mathrm{d}\Pi_N(\theta)< e^{-2N\delta_N^2 } \Pi_N(\mathcal{B}_N(\theta_0,U)) \right).
\end{align*}
Let $\sup_{\theta_0\in S}\norm{\theta_0}_{\calR}\leq M<\infty$ and $U=U(M)$, as stated in the hypotheses.
Since the hypotheses of \Cref{theorem_uniform_posterior_contraction_implies_uniform_convergence_posterior_mean} include the hypotheses of \Cref{lemma_1_3_3_Nickl}, we may apply \Cref{lemma_1_3_3_Nickl} with $B_N\leftarrow\mathcal{B}_N(\theta_0,U)$, $\nu(\cdot)\leftarrow \Pi_N(\cdot\cap \mathcal{B}_N(\theta_0,U))/\Pi_N(\mathcal{B}_N(\theta_0,U))$, $K\leftarrow 2$, and $U$.
The conclusion \eqref{eq_lemma_1_3_3_Nickl_conclusion} of \Cref{lemma_1_3_3_Nickl} then implies
\begin{equation*}
P^N_{\theta_0}\left(\int_{\mathcal{B}_N(\theta_0,U)} e^{\ell_N(\theta)-\ell_N(\theta_0)}\mathrm{d}\Pi_N(\theta)\leq e^{-2N\delta_N^2}\Pi_N(\mathcal{B}_N(\theta_0,U))\right)\leq \frac{8 (U^2+1) }{4N\delta_N^2},
\end{equation*}
which in turn implies the desired conclusion.
This completes the proof of \Cref{lemma_probability_of_A_N_events_converges_to_1_uniformly}.
\end{proof}

\begin{proof}[Proof of \Cref{theorem_uniform_posterior_contraction_implies_uniform_convergence_posterior_mean}]
 The idea of the proof is to modify the proof of \cite[Theorem 2.3.2]{Nic:2023} by replacing $\theta_0$-dependent constants with constants that depend only on the suitably bounded set $S$ of candidate truths, and by substituting the $\norm{\cdot}_{L^2_\zeta}$ norm with the $\norm{\cdot}_{\mathbb{B}}$ norm.
 
 The first part of the proof is analogous to the first part of the proof of \cite[Theorem 2.3.2]{Nic:2023}: define the event
 \begin{equation}
 \label{eq_F_N}
   F_N\coloneqq \{\theta\in\Theta\ :\ \norm{\theta-\theta_0}_{\mathbb{B}}\leq C\delta_N^\eta\},
 \end{equation}
 for $C$ and $\delta_N^\eta$ as in \eqref{eq_hypothesis_uniform_posterior_contraction_in_Banach_norm}.
Then for arbitrary $\theta_0\in S$,
 \begin{align}
  \norm{\Exp^{\Pi_N}[\theta\vert D_N]-\theta_0}_{\mathbb{B}} &\leq \Exp^{\Pi_N}[\norm{\theta-\theta_0}_{\mathbb{B}}\vert D_N]=\Exp^{\Pi_N}[\norm{\theta-\theta_0}_{\mathbb{B}}(\mathbf{1}_{F_N}+\mathbf{1}_{F_N^\complement})\vert D_N]
  \nonumber\\
  &\leq C\delta_N^\eta+\Exp^{\Pi_N}[\norm{\theta-\theta_0}_{\mathbb{B}}\mathbf{1}_{F_N^\complement}\vert D_N]
  \nonumber\\
  &\leq C\delta_N^\eta+\Exp^{\Pi_N}[\norm{\theta-\theta_0}^2_{\mathbb{B}}\vert D_N]^{1/2}\Pi_N(F_N^\complement\vert D_N)^{1/2},
  \label{eq_intermed00}
 \end{align}
where the first, second, and third inequalities follow by Jensen's inequality, the definition of $F_N$, and the Cauchy--Schwarz inequality respectively.

We now modify the remainder of the proof of \cite[Theorem 2.3.2]{Nic:2023} in order to replace some $\theta_0$-dependent quantities with analogous quantities that depend only on $S$. 
We do this in several steps.
First, we show that the second term on the right-hand side of \eqref{eq_intermed00} satisfies
 \begin{align*}
 &P^N_{\theta_0}\left(\Exp^{\Pi_N}[\norm{\theta-\theta_0}^2_{\mathbb{B}}\vert D_N]\Pi_N(F_N^\complement\vert D_N)>(C\delta_N^\eta)^2\right)
 \\
 \leq& o(1)+P^N_{\theta_0}\left(\Exp^{\Pi_N}[\norm{\theta-\theta_0}^2_{\mathbb{B}}\vert D_N]e^{-bN\delta_N^2}>(C\delta_N^\eta)^2\right),
\end{align*}
where the $o(1)$ term arises from the hypothesis \eqref{eq_hypothesis_uniform_posterior_contraction_in_Banach_norm} and is thus uniform over $\theta_0\in S$.
Second, we further bound the second term on the right-hand side of the inequality above by
\begin{align*}
 &P^N_{\theta_0}\left(\Exp^{\Pi_N}[\norm{\theta-\theta_0}^2_{\mathbb{B}}\vert D_N]e^{-bN\delta_N^2}>(C\delta_N^\eta)^2\right)
 \\
\leq &P^N_{\theta_0}\left(\int_{\Theta} \norm{\theta-\theta_0}^2_{\mathbb{B}} e^{\ell_N(\theta)-\ell_N(\theta_0)} e^{(\mathbf{A}+2-b)N\delta_N^2}\mathrm{d}\Pi_N(\theta)>(C\delta_N^\eta)^2\right)+P^N_{\theta_0}\left(A_N^\complement\right),
\end{align*}
for $A_N$ as in \eqref{eq_Nickl_1_28_uniform}. 
By \Cref{lemma_probability_of_A_N_events_converges_to_1_uniformly}, $P(A_N^\complement)\leq \tfrac{2(U^2+1)}{N\delta_N^2}$.
In the third step, we show that 
\begin{align*}
&P^N_{\theta_0}\left(\int_{\Theta} \norm{\theta-\theta_0}^2_{\mathbb{B}} e^{\ell_N(\theta)-\ell_N(\theta_0)} e^{(\mathbf{A}+2-b)N\delta_N^2}\mathrm{d}\Pi_N(\theta)>(C\delta_N^\eta)^2\right)
\\
\leq & \left(\int_{\Theta} \norm{\theta'}^2_{\mathbb{B}}\mathrm{d}\Pi^\prime (\theta') +\sup_{\theta_0\in S}\norm{\theta_0}_{\mathbb{B}}^2\right)e^{(\mathbf{A}+2-b)N\delta_N^2}(C\delta_N^\eta)^{-2},
\end{align*}
where the upper bound converges to zero for $b>\mathbf{A}+2$.
We then conclude by combining the preceding steps. 

\paragraph{Step 1} The second term on the right-hand side of \eqref{eq_intermed00} satisfies
\begin{align}
 &P^N_{\theta_0}\left(\Exp^{\Pi_N}[\norm{\theta-\theta_0}^2_{\mathbb{B}}\vert D_N]\Pi_N(F_N^\complement\vert D_N)>(C\delta_N^\eta)^2\right)
 \nonumber
 \\
 =&P^N_{\theta_0}\left(\Exp^{\Pi_N}[\norm{\theta-\theta_0}^2_{\mathbb{B}}\vert D_N]\Pi_N(F_N^\complement\vert D_N)>(C\delta_N^\eta)^2,\Pi_N(F_N^\complement\vert D_N)\leq e^{-bN\delta_N^2}\right)
 \label{eq_intermed07}
 \\
 &+P^N_{\theta_0}\left(\Exp^{\Pi_N}[\norm{\theta-\theta_0}^2_{\mathbb{B}}\vert D_N]\Pi_N(F_N^\complement\vert D_N)>(C\delta_N^\eta)^2,\Pi_N(F_N^\complement\vert D_N)> e^{-bN\delta_N^2}\right)
 \nonumber
\end{align}
for all $b > 0$.
We bound the first and second terms on the right-hand side of \eqref{eq_intermed07} separately. 
Recall the definition \eqref{eq_Nickl_2_20} of $\Theta_N(M)$.
For the second term on the right-hand side of \eqref{eq_intermed07}, note that if $\Pi_N(F_N\vert D_N)\leq a$ for some $a\geq 0$, then $\Pi_N(\Theta_N(M)\cap F_N\vert D_N)\leq a$.
This implies the first inequality below:
\begin{align*}
   P^N_{\theta_0}\left(\Pi_N(F_N\vert D_N)\leq 1-e^{-bN\delta_N^2}\right)\leq & P^N_{\theta_0}\left(\Pi_N(\Theta_N(M)\cap F_N\vert D_N)\leq 1-e^{-bN\delta_N^2}\right)
   \\
   \leq & \frac{C_1}{(N\delta_N^2)^{\mathbf{q}}}.
\end{align*}
The second inequality above follows by the hypothesis \eqref{eq_hypothesis_uniform_posterior_contraction_in_Banach_norm} and by the definition \eqref{eq_F_N} of $F_N$.
Thus, the second term on the right-hand side of \eqref{eq_intermed07} satisfies
\begin{align*}
 &P^N_{\theta_0}\left(\Exp^{\Pi_N}[\norm{\theta-\theta_0}^2_{\mathbb{B}}\vert D_N]\Pi_N(F_N^\complement\vert D_N)>(C\delta_N^\eta)^2,\Pi_N(F_N^\complement\vert D_N)> e^{-bN\delta_N^2}\right)
 \\
 \leq & P^N_{\theta_0}\left(\Pi_N(F_N^\complement\vert D_N)> e^{-bN\delta_N^2}\right)
 \\
 =& P^N_{\theta_0}\left(\Pi_N(F_N\vert D_N)\leq 1-e^{-bN\delta_N^2}\right)
 \\
 \leq & \frac{C_1}{(N\delta_N^2)^{\mathbf{q}}}.
\end{align*}
For the first term on the right-hand side of \eqref{eq_intermed07},
\begin{align*}
 &P^N_{\theta_0}\left(\Exp^{\Pi_N}[\norm{\theta-\theta_0}^2_{\mathbb{B}}\vert D_N]\Pi_N(F_N^\complement\vert D_N)>(C\delta_N^\eta)^2,\Pi_N(F_N^\complement\vert D_N)\leq e^{-bN\delta_N^2}\right)
 \\
 \leq &P^N_{\theta_0}\left(\Exp^{\Pi_N}[\norm{\theta-\theta_0}^2_{\mathbb{B}}\vert D_N]e^{-bN\delta_N^2}>(C\delta_N^\eta)^2,\Pi_N(F_N^\complement\vert D_N)\leq e^{-bN\delta_N^2}\right)
 \\
 \leq &P^N_{\theta_0}\left(\Exp^{\Pi_N}[\norm{\theta-\theta_0}^2_{\mathbb{B}}\vert D_N]e^{-bN\delta_N^2}>(C\delta_N^\eta)^2\right).
\end{align*}
Substituting these bounds into \eqref{eq_intermed07}, we obtain
\begin{align}
 &P^N_{\theta_0}\left(\Exp^{\Pi_N}[\norm{\theta-\theta_0}^2_{\mathbb{B}}\vert D_N]\Pi_N(F_N^\complement\vert D_N)>(C\delta_N^\eta)^2\right)
 \nonumber
 \\
 \leq & \frac{C_1}{(N\delta_N^2)^{\mathbf{q}}}+P^N_{\theta_0}\left(\Exp^{\Pi_N}[\norm{\theta-\theta_0}^2_{\mathbb{B}}\vert D_N]e^{-bN\delta_N^2}>(C\delta_N^\eta)^2\right).
 \label{eq_intermed07a}
\end{align}

\paragraph{Step 2} In this step, we further bound the right-hand side of \eqref{eq_intermed07a}.
Recall the definition \eqref{eq_Nickl_1_28_uniform} of the set $A_N=A_N(\theta_0)$.
Observe that 
\begin{align}
 &P^N_{\theta_0}\left(\Exp^{\Pi_N}[\norm{\theta-\theta_0}^2_{\mathbb{B}}\vert D_N]e^{-bN\delta_N^2}>(C\delta_N^\eta)^2\right)
 \nonumber
 \\
\leq &P^N_{\theta_0}\left(\Exp^{\Pi_N}[\norm{\theta-\theta_0}^2_{\mathbb{B}}\vert D_N]e^{-bN\delta_N^2}>(C\delta_N^\eta)^2,A_N\right)+P^N_{\theta_0}\left(A_N^\complement\right).
\label{eq_intermed08}
\end{align}
By \Cref{lemma_probability_of_A_N_events_converges_to_1_uniformly}, we have
\begin{equation*}
 P^N_{\theta_0}(A_N^\complement)\leq \frac{2(U^2+1)}{N\delta_N^2},
\end{equation*}
uniformly over $\theta_0\in S$.
Thus it remains to bound the first term on the right-hand side of \eqref{eq_intermed08}.
Using the definition \eqref{eq_Nickl_1_12} of the posterior with $\Pi\leftarrow \Pi_N$, and using the fact that $\ell_N(\theta_0)$ is constant with respect to $\theta$, we have
\begin{align*}
 \Exp^{\Pi_N}[\norm{\theta-\theta_0}^2_{\mathbb{B}}\vert D_N]=\int_{\Theta} \norm{\theta-\theta_0}^2_{\mathbb{B}} \frac{e^{\ell_N(\theta)-\ell_N(\theta_0)}}{\int_\Theta e^{\ell_N(\theta')-\ell_N(\theta_0)}\mathrm{d}\Pi_N(\theta')}\mathrm{d}\Pi_N(\theta).
\end{align*}
Thus, by the definition \eqref{eq_Nickl_1_28_uniform} of $A_N$, the following inequality holds on $A_N$:
\begin{align*}
  \Exp^{\Pi_N}[\norm{\theta-\theta_0}^2_{\mathbb{B}}\vert D_N]e^{-bN\delta_N^2}\leq \int_{\Theta} \norm{\theta-\theta_0}^2_{\mathbb{B}} e^{\ell_N(\theta)-\ell_N(\theta_0)} e^{(\mathbf{A}+2-b)N\delta_N^2}\mathrm{d}\Pi_N(\theta).
\end{align*}
Thus the first term on the right-hand side of \eqref{eq_intermed08} satisfies
\begin{align*}
 &P^N_{\theta_0}\left(\Exp^{\Pi_N}[\norm{\theta-\theta_0}^2_{\mathbb{B}}\vert D_N]e^{-bN\delta_N^2}>(C\delta_N^\eta)^2,A_N\right)
 \\
 \leq & P^N_{\theta_0}\left(\int_{\Theta} \norm{\theta-\theta_0}^2_{\mathbb{B}} e^{\ell_N(\theta)-\ell_N(\theta_0)} e^{(\mathbf{A}+2-b)N\delta_N^2}\mathrm{d}\Pi_N(\theta)>(C\delta_N^\eta)^2,A_N\right)
 \\
 \leq & P^N_{\theta_0}\left(\int_{\Theta} \norm{\theta-\theta_0}^2_{\mathbb{B}} e^{\ell_N(\theta)-\ell_N(\theta_0)} e^{(\mathbf{A}+2-b)N\delta_N^2}\mathrm{d}\Pi_N(\theta)>(C\delta_N^\eta)^2\right).
\end{align*}
By applying these bounds to \eqref{eq_intermed07a}, we obtain
\begin{align}
&P^N_{\theta_0}\left(\Exp^{\Pi_N}[\norm{\theta-\theta_0}^2_{\mathbb{B}}\vert D_N]\Pi_N(F_N^\complement\vert D_N)>(C\delta_N^\eta)^2\right)
 \nonumber
 \\
 \leq &\frac{C_1}{(N\delta_N^2)^{\mathbf{q}}}+\frac{2(U^2+1)}{N\delta_N^2}
 \label{eq_intermed07b}
 \\
 &+P^N_{\theta_0}\left(\int_{\Theta} \norm{\theta-\theta_0}^2_{\mathbb{B}} e^{\ell_N(\theta)-\ell_N(\theta_0)} e^{(\mathbf{A}+2-b)N\delta_N^2}\mathrm{d}\Pi_N(\theta)>(C\delta_N^\eta)^2\right).
 \nonumber
 \end{align}

\paragraph{Step 3} We further bound the right-hand side of \eqref{eq_intermed07b}.
By Markov's inequality and Tonelli's theorem, the third term on the right-hand side of \eqref{eq_intermed07b} satisfies
\begin{align*}
 &P^N_{\theta_0}\left(\int_{\Theta} \norm{\theta-\theta_0}^2_{\mathbb{B}} e^{\ell_N(\theta)-\ell_N(\theta_0)} e^{(\mathbf{A}+2-b)N\delta_N^2}\mathrm{d}\Pi_N(\theta)>(C\delta_N^\eta)^2\right)
 \\
 \leq & E^N_{\theta_0} \left[\int_{\Theta} \norm{\theta-\theta_0}^2_{\mathbb{B}} e^{\ell_N(\theta)-\ell_N(\theta_0)} e^{(\mathbf{A}+2-b)N\delta_N^2}\mathrm{d}\Pi_N(\theta)\right] (C\delta_N^\eta)^{-2}
 \\
 =& \int_{\Theta} \norm{\theta-\theta_0}^2_{\mathbb{B}} E^N_{\theta_0} \left[e^{\ell_N(\theta)-\ell_N(\theta_0)}\right]\mathrm{d}\Pi_N(\theta)  e^{(\mathbf{A}+2-b)N\delta_N^2}(C\delta_N^\eta)^{-2}.
\end{align*}
By \eqref{eq_Nickl_1_10}, $e^{\ell_N(\theta)-\ell_N(\theta_0)}$ is the Radon--Nikodym derivative $\tfrac{\mathrm{d}P^N_\theta}{\mathrm{d}P^N_{\theta_0}}$, and thus $E^N_{\theta_0}[e^{\ell_N(\theta)-\ell_N(\theta_0)}]=1$ for every $\theta\in\Theta$.

To bound $\int_{\Theta} \norm{\theta-\theta_0}^2_{\mathbb{B}}\mathrm{d}\Pi_N(\theta)$, note that $\norm{\theta-\theta_0}^2_{\mathbb{B}}\leq 2(\norm{\theta}^2_{\mathbb{B}}+\norm{\theta_0}^2_{\mathbb{B}})$, by the triangle inequality and the elementary inequality $(a+b)^2\leq 2(a^2+b^2)$ for $a,b\in\mathbb{R}$.
Recall from \Cref{definition_rescaled_Gaussian_priors} that $\Pi_N$ is the law of $\theta\coloneqq N^{-d_z/(4\alpha+4\kappa+2d_z)}\theta'$ for $\theta'\sim\Pi^\prime $, where $\Pi^\prime (\calR)=1$ by \Cref{assumption_Nickl_Condition_2_2_1}.
Thus,
\begin{equation*}
 \int_\Theta \norm{\theta}^2_{\mathbb{B}}\mathrm{d}\Pi_N(\theta)=\int_\Theta \norm{N^{-d/(4\alpha+4\kappa+2d)}\theta'}^2_{\mathbb{B}}\mathrm{d}\Pi^\prime (\theta')\leq \int_{\Theta} \norm{\theta'}^2_{\mathbb{B}}\mathrm{d}\Pi^\prime (\theta'),
\end{equation*}
where the rightmost term is the second $\norm{\cdot}_{\mathbb{B}}$-moment of $\Pi^\prime $ and is thus finite, by the hypothesis that $\Pi^\prime $ is a Gaussian measure that satisfies $\Pi^\prime (\mathbb{B})=1$.
Next, by the hypothesis that $S\subset\mathbb{B}$ is bounded in the $\mathbb{B}$ norm, $\sup_{\theta_0\in S}\norm{\theta_0}_{\mathbb{B}}$ is finite.
Thus the third term on the right-hand side of \eqref{eq_intermed07b} satisfies
\begin{align}
 &P^N_{\theta_0}\left(\int_{\Theta} \norm{\theta-\theta_0}^2_{\mathbb{B}} e^{\ell_N(\theta)-\ell_N(\theta_0)} e^{(\mathbf{A}+2-b)N\delta_N^2}\mathrm{d}\Pi_N(\theta)>(C\delta_N^\eta)^2\right)
 \nonumber
 \\
 \leq & 2\left(\int_{\Theta} \norm{\theta'}^2_{\mathbb{B}}\mathrm{d}\Pi^\prime (\theta') +\sup_{\theta_0\in S}\norm{\theta_0}_{\mathbb{B}}^2\right)e^{(\mathbf{A}+2-b)N\delta_N^2}(C\delta_N^\eta)^{-2}.
 \label{eq_intermed09}
 \end{align}
 For $b>\mathbf{A}+2$, the upper bound converges to zero uniformly over $\theta_0\in S$.
 
 \paragraph{Conclusion} Recall \eqref{eq_intermed00}:
 \begin{align*}
 \norm{\Exp^{\Pi_N}[\theta\vert D_N]-\theta_0}_{\mathbb{B}}  \leq & C\delta_N^\eta+\Exp^{\Pi_N}[\norm{\theta-\theta_0}^2_{\mathbb{B}}\vert D_N]^{1/2}\Pi_N(F_N^\complement\vert D_N)^{1/2}.
\end{align*}
The inequality above implies
\begin{equation*}
 \norm{\Exp^{\Pi_N}[\theta\vert D_N]-\theta_0}_{\mathbb{B}} >2C\delta^\eta_N\Longrightarrow \Exp^{\Pi_N}[\norm{\theta-\theta_0}^2_{\mathbb{B}}\vert D_N]^{1/2}\Pi_N(F_N^\complement\vert D_N)^{1/2}>C\delta^\eta_N,
\end{equation*}
and hence for every $\theta_0\in S$,
\begin{equation*}
P^N_{\theta_0}\left( \norm{\Exp^{\Pi_N}[\theta\vert D_N]-\theta_0}_{\mathbb{B}} >2C\delta^\eta_N\right)\leq P^N_{\theta_0}\left( \Exp^{\Pi_N}[\norm{\theta-\theta_0}^2_{\mathbb{B}}\vert D_N]^{1/2}\Pi_N(F_N^\complement\vert D_N)^{1/2}>C\delta^\eta_N\right).
\end{equation*}
By combining \eqref{eq_intermed07b} and \eqref{eq_intermed09}, the right-hand side of the inequality above satisfies
\begin{align*}
 &P^N_{\theta_0}\left( \Exp^{\Pi_N}[\norm{\theta-\theta_0}^2_{\mathbb{B}}\vert D_N]^{1/2}\Pi_N(F_N^\complement\vert D_N)^{1/2}>C\delta^\eta_N\right)
 \\
 \leq & \frac{C_1}{(N\delta_N^2)^{\mathbf{q}}} +\frac{2(U^2+1)}{N\delta_N^2}+2\left(\int_{\Theta} \norm{\theta'}^2_{\mathbb{B}}\mathrm{d}\Pi^\prime (\theta') +\text{diam}_{\calH}(S)^2\right)e^{(\mathbf{A}+2-b)N\delta_N^2}(C\delta_N^\eta)^{-2}.
\end{align*}
By hypothesis, the scalar $C_1$ in the first term on the right-hand side does not depend on $\theta_0$.
Thus, for sufficiently large $N$, the third term on the right-hand side is smaller than the first and second terms. 
Since 
\begin{equation*}
 \frac{C_1}{(N\delta_N^2)^{\mathbf{q}}} +\frac{2(U^2+1)}{N\delta_N^2} \leq \frac{C_1+2(U^2+1)}{(N\delta_N^2)^{1\wedge \mathbf{q}}} 
\end{equation*}
we obtain \eqref{eq_concentration_for_posterior_mean}.
This completes the proof of \Cref{theorem_uniform_posterior_contraction_implies_uniform_convergence_posterior_mean}. 
\end{proof}

\subsubsection{Proof of existence of uniform constants}
\label{sec_proof_of_existence_of_uniform_U_and_uniform_A}

\existenceOfUniformUandA*
\begin{proof}[Proof of \Cref{proposition_existence_of_uniform_U_and_uniform_A}]
 Given that the hypotheses of \Cref{theorem_Nickl_2_2_2} hold, it follows from \Cref{assumption_Nickl_Condition_2_2_1} that $\calH$ is the RKHS of the base prior $\Pi^\prime $ and $\Pi^\prime (\calR)=1$.
 Recall that the RKHS of a Gaussian measure is continuously embedded in the support of the Gaussian measure; see e.g. \cite[Proposition 2.6.9]{GinNic:2016}. 
 Thus, there exists some $0<C<\infty$ such that $\norm{h}_{\calR}\leq C\norm{h}_{\calH}$ for every $h\in\calH$, and since $S$ is bounded in the $\calH$ norm, it follows that $S$ is bounded in the $\calR$ norm, i.e. $\sup_{\theta_0\in S}\norm{\theta_0}_{\calR}$ is finite. This proves the first conclusion.

Next, we prove the second conclusion of the proposition, by adapting step 2 of the proof of \cite[Theorem 2.2.2]{Nic:2023}, cf. \cite[pp. 33-34]{Nic:2023}. 
This is justified, because the hypotheses of \Cref{theorem_Nickl_2_2_2} are stronger than the hypotheses of \cite[Theorem 2.2.2]{Nic:2023}.

Let $0<M_0<\infty$ and $U$ be as in the statement of the proposition, and let $\theta_0\in S$.
We claim that 
\begin{align*}
\Pi_N(\mathcal{B}_N(\theta_0,U))=&\Pi_N(\theta\in\Theta : d_G(\theta,\theta_0)\leq \delta_N, \norm{G(\theta)}_\infty \leq U)
\\
\geq & \Pi_N(\theta\in\Theta : d_G(\theta,\theta_0)\leq \delta_N,\norm{\theta-\theta_0}_{\calR}\leq M_0)
\\
\geq & \Pi_N(\theta\in\Theta : \norm{\theta-\theta_0}_{(H^\kappa)^*}\leq \tfrac{\delta_N}{L},\norm{\theta-\theta_0}_{\calR}\leq M_0),
\end{align*}
for some scalar $L$ to be determined below.
The equation follows from the definition \eqref{eq_Nickl_1_23} of $\mathcal{B}_N$.
For the first inequality, note first that
\begin{equation}
\label{step00}
 \norm{\theta-\theta_0}_{\calR}\leq M_0\Longrightarrow \norm{\theta}_{\calR}\leq M_0+\sup_{\theta_0\in S}\norm{\theta_0}_{\calR},
\end{equation}
by the triangle inequality.
Now \eqref{eq_Nickl_2_3} in \Cref{assumption_Nickl_Condition_2_1_1} implies that for any $\theta$ such that $\norm{\theta-\theta_0}_{\calR}\leq M_0$, it holds that $\norm{G(\theta)}_\infty\leq U$, thus proving the first inequality.
For the second inequality, note that \eqref{eq_Nickl_2_4} from \Cref{assumption_Nickl_Condition_2_1_1} can be written as follows: for any $0<M<\infty$, there exists $L=L(M)$ such that
\begin{equation*}
 d_G(\theta_1,\theta_2)\leq L \norm{\theta_1-\theta_2}_{(H^\kappa(\calZ))^*},\quad \forall \theta_1,\theta_2\in\Theta\cap B_{\calR}(M).
\end{equation*}
If $\norm{\theta-\theta_0}_{\calR}\leq M_0$, then by \eqref{step00}, we may replace $M$ in the above display with $M_0+\sup_{\theta_0\in S}\norm{\theta_0}_{\calR}$ and let $L=L(M_0+\sup_{\theta_0\in S}\norm{\theta_0}_{\calR})$, to conclude that $\norm{\theta-\theta_0}_{(H^\kappa(\calZ))^*}\leq \tfrac{\delta_N}{L}$ implies $d_G(\theta,\theta_0)\leq \delta_N$. 
This proves the second inequality.

Next, we rewrite the right-hand side of the second inequality.
Note that for any $\theta_0\in\Theta$ and $r>0$, $B_{(H^\kappa)^*}(\theta_0,r)\coloneqq\{\theta\in\Theta: \norm{\theta-\theta_0}_{(H^\kappa)^*} \leq r)$ and $B_{\calR}(\theta_0,r)\coloneqq \{\theta\in\Theta: \norm{\theta-\theta_0}_{\calR} \leq r\}$ are symmetric Borel sets if and only if $\theta_0=0$. 
Thus
\begin{align*}
 &\Pi_N(\theta\in\Theta : \norm{\theta-\theta_0}_{(H^\kappa)^*}\leq \tfrac{\delta_N}{L},\norm{\theta-\theta_0}_{\calR}\leq M_0)
 \\
 =&\Pi_N( \theta\in B_{(H^\kappa)^*}(\theta_0,\tfrac{\delta_N}{L})\cap B_{\calR}(\theta_0,M_0))
 \\
 =&\Pi_N(\theta-\theta_0\in  B_{(H^\kappa)^*}(0,\tfrac{\delta_N}{L})\cap B_{\calR}(0,M_0)).
\end{align*}
We now wish to bound the right-hand side of the last equation from below. 
Since $S$ is bounded in the $\calH$ norm, it follows that $S\subset \calH$.
In particular, $\theta_0\in\calH$.
Since $ B_{(H^\kappa)^*}(0,\tfrac{\delta_N}{L})\cap B_{\calR}(0,M)\subset\calR$ is a symmetric Borel set and $\calR$ is separable by \Cref{assumption_Nickl_Condition_2_2_1}, we may apply the small ball probability estimate from \cite[Corollary 2.6.18]{GinNic:2016} to obtain
\begin{align*}
 &\Pi_N(\theta-\theta_0\in  B_{(H^\kappa)^*}(0,\tfrac{\delta_N}{L})\cap B_{\calR}(0,M_0))
 \nonumber\\
 \geq& \exp(-\tfrac{1}{2}\norm{\theta_0}_{\calH}^2) \Pi_N(\theta\in B_{(H^\kappa)^*}(0,\tfrac{\delta_N}{L})\cap B_{\calR}(0,M_0))
 \nonumber
 \\
 \geq &\exp\left(-\frac{1}{2}\sup_{\theta_0\in S}\norm{\theta_0}^2_{\calH}\right)\Pi_N(\theta\in B_{(H^\kappa)^*}(0,\tfrac{\delta_N}{L})\cap B_{\calR}(0,M_0)),
\end{align*}
where in the second inequality we used the hypothesis that the set $S$ is bounded in the $\calH$ norm.

By the Gaussian correlation inequality, see e.g. \cite[Theorem B.1.2]{Nic:2023}, 
\begin{align*}
 &\Pi_N(\theta\in B_{(H^\kappa)^*}(0,\tfrac{\delta_N}{L})\cap B_{\calR}(0,M_0))\geq \Pi_N(\theta\in B_{(H^\kappa)^*}(0,\tfrac{\delta_N}{L}))\Pi_N(\theta\in B_{\calR}(0,M_0)).
\end{align*}
By choosing $M_0$ sufficiently large, we may ensure that $\Pi_N(\theta\in B_{\calR}(0,M_0))\geq \tfrac{1}{2}$.
By the argument between equations (2.22) and (2.24) on \cite[p. 34]{Nic:2023}, there exists some $0<a<\infty$ not depending on $S$, such that 
\begin{equation*}
 \Pi_N(\theta\in B_{(H^\kappa)^*}(0,\tfrac{\delta_N}{L}))\geq \exp\left(-aN\delta_N^2\right),
\end{equation*}
for $L$ as in \eqref{step00}.
Combining the preceding steps implies that
\begin{align*}
 \Pi_N(\mathcal{B}_N(\theta_0))\geq  \exp\left(-\log 2-\frac{1}{2}\sup_{\theta_0\in S}\norm{\theta_0}^2_{\calH}-aN\delta_N^2\right).
\end{align*}
By \eqref{eq_Nickl_2_19}, it follows that $N\delta_N^2\to\infty$ as $N\to\infty$. Thus, there exists some $\mathbf{A}$ that satisfies \eqref{eq_uniform_constant_A}.
This completes the proof of \Cref{proposition_existence_of_uniform_U_and_uniform_A}.
\end{proof}

If the set $S$ of candidate true parameters is not bounded in the $\calH$ norm, then the proof technique of \Cref{proposition_existence_of_uniform_U_and_uniform_A} cannot be applied as written.
 This is because of the intermediate step that applies \cite[Corollary 2.6.18]{GinNic:2016}, which uses the property that $\theta_0$ belongs to the RKHS $\calH$ of $\Pi_N$.
For example, if $S= B_{\calR}(C)$ for some $0<C<\infty$, then $S$ is not bounded in the $\calH$ norm. If it were, then by the continuous embedding of $\calH$ in $\calR$, the $\calH$ norm and $\calR$ norm would be equivalent, which is not true when $\calH$ is infinite-dimensional.

\subsection{Proofs for uniform posterior contraction and distinguishability in \footnotesize{$L^2$}}
\label{sec_proof_uniform_posterior_contraction_distinguishability_L2}

\begin{proof}[Proof of \Cref{theorem_uniform_posterior_contraction_L2}]
We prove the first statement.
Given the hypothesis that $S\subset\calR$ is bounded in the $\calH$ norm, we may apply \Cref{proposition_existence_of_uniform_U_and_uniform_A} to conclude that $\sup_{\theta_0\in S}\norm{\theta_0}_{\calR}$ is finite.
Again by \Cref{proposition_existence_of_uniform_U_and_uniform_A}, for for every $M>0$ and for $U=U(M+\sup_{\theta_0\in S}\norm{\theta_0}_{\calR})$, there exists some $0<\mathbf{A}=\mathbf{A}(S,U)<\infty$ such that \eqref{eq_uniform_constant_A} holds, i.e.
\begin{equation*}
 \forall \theta_0\in S,N\in\N, \quad \Pi_N(\mathcal{B}_N(\theta_0,U(M+\sup_{\theta_0\in S}\norm{\theta_0}_{\calR})))\geq e^{-\mathbf{A}N\delta_N^2}.
\end{equation*}
By \eqref{eq_Nickl_2_3}, $U(M+\sup_{\theta_0\in S}\norm{\theta_0}_{\calR})$ increases with $M$, and by the discussion below \eqref{eq_Nickl_1_23}, $\mathcal{B}_N(\theta_0,\mathcal{U})$ is increasing in $\mathcal{U}$.
Thus, for the same $\mathbf{A}$ as in the preceding display, the following holds for every $M_1\geq M$:
\begin{equation*}
 \forall \theta_0\in S,N\in\N, \quad \Pi_N(\mathcal{B}_N(\theta_0,U(M_1+\sup_{\theta_0\in S}\norm{\theta_0}_{\calR})))\geq e^{-\mathbf{A}N\delta_N^2}.
\end{equation*}
This implies that we may increase $M$ further if necessary, to ensure that there exists some $\mathbf{B}>\mathbf{A}+2$ and some $\overline{m}=\overline{m}(M,L)$ that do not depend on $\theta_0$, such that the sets $(\Theta_N(M))_{N\in\N}$ also satisfy \eqref{eq_Nickl_1_25} and \eqref{eq_Nickl_1_26}.
In particular, for any such $M$, and for any $0<b<\mathbf{B}-(\mathbf{A}+2)$, the final conclusion of \Cref{theorem_Nickl_2_2_2} holds: for any $0<c<1$, there exists $m=m(\mathbf{A},b,\overline{m})$ sufficiently large and independent of $c$, such that for $U=U(M)$, for $C_U$ given by \eqref{eq_C_of_U_function} with $u\leftarrow U$, for some $N^*(c)\in\N$, and for every $N\geq N^*(c)$, \eqref{eq_Nickl_1_27} holds:
\begin{equation*}
  P^N_{\theta_0}\left(\Pi_N(\{\theta\in\Theta_N(M): d_G(\theta,\theta_0)\leq m\delta_N C_{U}^{-1/2}\}\vert D_N)\leq 1-e^{-bN\delta_N^2}\right)\leq (1+c)\frac{2(U^2+1)}{N\delta_N^2}.
\end{equation*}
In particular, neither $m$, $\delta_N$, $C_U$, $b$, $c$, or $U$ depend on $\theta_0$.
By fixing an arbitrary $0<c<1$, we obtain the first conclusion.

Next, we prove the second conclusion \eqref{eq_Nickl_2_26}, using the same argument as the proof of \cite[Theorem 2.3.1]{Nic:2023}.
If \Cref{assumption_local_Lipschitz_continuity_of_restricted_inverse_forward_map} holds, then there exists some $\eta>0$ and some $L^\prime=L^\prime(M)$, neither of which depend on $\theta_0$, such that
\begin{equation*}
\norm{\theta-\theta_0}_{L^2_\zeta(\calZ,W)}\leq L^\prime d_G(\theta,\theta_0)^\eta,\quad\forall (\theta,\theta_0)\in\Theta_N(M)\times S.
\end{equation*}
Thus,
\begin{align*}
  \left\{\theta\in\Theta_N(M):d_G(\theta,\theta_0)\leq m C_U^{-1/2}\delta_N\right\}\subseteq \left\{\theta\in\Theta_N(M):\norm{\theta-\theta_0}_{L^2_\zeta(\calZ,W)}\leq L^\prime (mC_U^{-1/2}\delta_N)^\eta\right\}.
 \end{align*}
The set inclusion above and the first conclusion then imply \eqref{eq_Nickl_2_26}.
\end{proof}

\subsection{Proofs for uniform posterior contraction and distinguishability in Sobolev and supremum norms}
\label{sec_proof_uniform_posterior_contraction_distinguishability_Sobolev_Linfty}

\begin{proof}[Proof of \Cref{theorem_uniform_posterior_contraction_Linfty}]
 The proof is essentially the same as the proof of \cite[Proposition 4.1.3]{Nic:2023}. 
Recall the notation `$\lesssim$' from \Cref{subsection_notation}, the notation $d_z$ for the dimension of the domain $\calZ$, and recall the Sobolev embedding and the interpolation inequality for Sobolev norms, see e.g. \cite[equation (A.4), (A.5)]{Nic:2023}: 
\begin{subequations}
 \begin{align}
  \label{eq_Nickl_A_4}
\forall\alpha> d_z/2,\ 0<\eta<\alpha-d_z/2,\ f\in H^\alpha&:\quad \norm{f}_\infty \lesssim \norm{f}_{C^\eta}\lesssim\norm{f}_{H^\alpha}
\\
\label{eq_Nickl_A_5}
\forall \beta_1,\beta_2\geq 0,\ \theta\in[0,1],\ u\in H^{\beta_1}\cap H^{\beta_2}&:\quad \norm{u}_{H^{\theta\beta_1+(1-\theta)\beta_2}}\lesssim \norm{u}^{\theta}_{H^{\beta_1}}\norm{u}^{1-\theta}_{H^{\beta_2}}.
 \end{align}
\end{subequations}
We first prove \eqref{eq_Sobolev_posterior_contraction_uniform_version}.
Fix an arbitrary $\beta'\in [0,\beta)$.
For the substitutions $\beta_1\leftarrow 0$, $\beta_2\leftarrow \beta$ and $\theta\leftarrow(\beta- \beta')/\beta$, it follows that $(1-\theta)\leftarrow \beta'/\beta$ and $\theta\beta_1+(1-\theta)\beta_2\leftarrow \beta'$.
Applying \eqref{eq_Nickl_A_5} with these substitutions yields
\begin{equation}
\label{eq_intermed01}
 \norm{u}_{H^{\beta^\prime}}\lesssim \norm{u}^{(\beta-\beta')/\beta}_{L^2}\norm{u}_{H^{\beta}}^{\beta'/\beta},\quad \forall u\in H^{\beta}.
\end{equation}
By the hypothesis that $\zeta$ is a probability measure on $\calZ$ with a Lebesgue density $p_\zeta$ that has strictly positive essential infimum, we may apply statement \cref{embedding_L2zeta_L2} of \cref{lemma_embedding} to conclude that
\begin{align*}
	\norm{u}_{L^2} \lesssim \norm{u}_{L^2_\zeta}, \quad \forall u \in L^2_\zeta(\calZ).
\end{align*}
By \Cref{assumption_Nickl_Condition_2_2_1}, $\Pi^\prime (\calR)=1$ and hence by \Cref{definition_rescaled_Gaussian_priors}, $\Pi_N(\calR)=1$ for every $N$.
By the hypothesis that $\calR$ is continuously embedded in $H^\beta(\calZ)$ for $\beta>d_z/2$, it follows that $\theta-\theta_0 \in H^{\beta}$ for $\Pi_N$-a.e. $\theta$ and every $\theta_0\in S$, and it follows that $\norm{\theta-\theta_0}_{H^\beta}\lesssim \norm{\theta}_{\calR}+\norm{\theta_0}_{\calR}$, by the triangle inequality.
Recall that in \Cref{theorem_uniform_posterior_contraction_L2}, $M$ is chosen to satisfy $\sup_{\theta_0\in S}\norm{\theta_0}_{\calR}<M<\infty$, i.e. $S\subset B_{\calR}(M)$. Given the restriction to $\theta$ such that $\norm{\theta}_{\calR}\leq M$ in \eqref{eq_Linfty_posterior_contraction_uniform_version}, it follows that
\begin{equation*}
\forall( \theta,\theta_0)\in B_{\calR}(M)\times S,\quad \norm{\theta-\theta_0}_{H^\beta}^{\beta'/\beta}\lesssim (2M)^{\beta'/\beta}\eqqcolon C_1=C_1(\beta,\beta',M).
\end{equation*}
By \eqref{eq_intermed01} with the substitution $u\leftarrow \theta-\theta_0$, and by the preceding inequality, we have
\begin{equation*}
\forall (\theta,\theta_0)\in B_{\calR}(M)\times S,\quad  \norm{\theta-\theta_0}_{H^{\beta^\prime}}\lesssim C_1 \norm{\theta-\theta_0}_{L^2_\zeta}^{(\beta - \beta')/\beta}.
\end{equation*}
Next, note that $\Theta_N(M)\subset \{\theta\in\Theta:\norm{\theta}_{\calR}\leq M\}$, by the definition \eqref{eq_Nickl_2_20} of $\Theta_N(M)$. By the preceding arguments,
\begin{align*}
 &\left\{\theta\in\Theta_N(M):\norm{\theta-\theta_0}_{L^2_\zeta}\leq L^\prime(mC_U^{-1/2}\delta_N)^\eta \right\}
 \\
 &\subseteq\left\{ \theta\in\Theta:\norm{\theta}_{\calR}\leq M, \norm{\theta-\theta_0}_{H^{\beta^\prime}}\lesssim C_1 \left(L^\prime(mC_U^{-1/2}\delta_N)^\eta\right)^{(\beta-\beta')/\beta}\right\}.
\end{align*}
This implies that for $P^N_{\theta_0}$-almost every realisation of $D_N$ that
\begin{align*}
  &\Pi_N\left(\left\{\theta\in\Theta_N(M):\norm{\theta-\theta_0}_{L^2_\zeta}\leq L^\prime(mC_U^{-1/2}\delta_N)^\eta \right\}\biggr\vert D_N\right)
 \\
 &\leq\Pi_N\left(\left\{ \theta\in\Theta:\norm{\theta}_{\calR}\leq M, \norm{\theta-\theta_0}_{H^{\beta^\prime}}\lesssim C_1 \left(L^\prime(mC_U^{-1/2}\delta_N)^\eta\right)^{(\beta-\beta')/\beta}\right\}\biggr\vert D_N\right).
\end{align*}
In particular, if the right-hand side of the inequality above is less than or equal to some $a>0$, then so is the left-hand side of the inequality.
Thus, for $C_0\coloneqq L^\prime (m C_U^{-1/2})^\eta$, it follows that
\begin{align*}
&P^N_{\theta_0}\left(\Pi_N\left(\left\{\theta\in\Theta_N(M):\norm{\theta-\theta_0}_{L^2_\zeta}\leq C\delta_N^\eta \right\}\biggr\vert D_N\right)\leq 1-e^{-bN\delta_N^2}\right)
\\
\geq &
P^N_{\theta_0}\left(\Pi_N\left(\left\{ \theta\in\Theta:\norm{\theta}_{\calR}\leq M, \norm{\theta-\theta_0}_{H^{\beta^\prime}}\lesssim C_1 \left(C_0\delta_N)^\eta\right)^{(\beta-\beta')/\beta}\right\}\biggr\vert D_N\right)\leq 1-e^{-bN\delta_N^2}\right).
\end{align*}
The desired conclusion \eqref{eq_Sobolev_posterior_contraction_uniform_version} of uniform posterior contraction in the $H^{\beta^\prime}$ metric now follows from the result \eqref{eq_Nickl_2_26} of uniform posterior contraction in the $L^2_\zeta$ metric, for some suitable $C$.

To prove the second conclusion  \eqref{eq_Linfty_posterior_contraction_uniform_version} of uniform posterior contraction in the metric induced by the supremum norm, we use a similar argument. By applying \eqref{eq_Nickl_A_4} with $\alpha\leftarrow \beta'$, we may conclude that there exists some constant $D$ such that for every $u\in H^{\beta^\prime}$, $\norm{u}_\infty\leq D\norm{u}_{H^{\beta^\prime}}$. Thus,
\begin{align*}
 \{\theta\in\Theta:\norm{\theta}_{\calR}\leq M,\norm{\theta-\theta_0}_{H^{\beta^\prime}}\leq C\delta_N^{\eta(\beta-\beta')/\beta}\}
 \subseteq \{\theta\in\Theta:\norm{\theta}_{\calR}\leq M,\norm{\theta-\theta_0}_{\infty}\leq CD\delta_N^{\eta(\beta-\beta')/\beta}\}.
\end{align*}
Thus, for $P^N_{\theta_0}$-almost every realisation of $D_N$,
\begin{align*}
&\Pi_N( \{\theta\in\Theta:\norm{\theta}_{\calR}\leq M,\norm{\theta-\theta_0}_{H^{\beta^\prime}}\leq C\delta_N^{\eta(\beta-\beta')/\beta}\}\vert D_N)
 \\
 \leq&\Pi_N(\{\theta\in\Theta:\norm{\theta}_{\calR}\leq M,\norm{\theta-\theta_0}_{\infty}\leq CD\delta_N^{\eta(\beta-\beta')/\beta}\}\vert D_N).
\end{align*}
In particular, if the right-hand side of the inequality above is less than or equal to some $a>0$, then so is the left-hand side of the inequality.
Thus, for $C^\prime\coloneqq CD$, it follows that
\begin{align*}
&P^N_{\theta_0}\left(\Pi_N( \{\theta\in\Theta:\norm{\theta}_{\calR}\leq M,\norm{\theta-\theta_0}_{H^{\beta^\prime}}\leq C\delta_N^{\eta(\beta-\beta')/\beta}\}\vert D_N)\leq 1-e^{-bN\delta_N^2}\right)
 \\
 \geq& P^N_{\theta_0}\left(\Pi_N(\{\theta\in\Theta:\norm{\theta}_{\calR}\leq M,\norm{\theta-\theta_0}_{\infty}\leq C^\prime\delta_N^{\eta(\beta-\beta')/\beta}\}\vert D_N)\leq 1-e^{-bN\delta_N^2}\right)
 \end{align*}
 and thus the second conclusion \eqref{eq_Linfty_posterior_contraction_uniform_version} of uniform posterior contraction in the metric induced by the supremum norm follows from the first conclusion \eqref{eq_Sobolev_posterior_contraction_uniform_version} of uniform posterior contraction in the $H^{\beta^\prime}$ metric.
This completes the proof of \Cref{theorem_uniform_posterior_contraction_Linfty}.
\end{proof}

\section{Distinguishability for simple hypotheses}
\label{sec:dist_simple_hyp}

The proof of the $L^2_\zeta$-posterior contraction result \cite[Theorem 2.3.1]{Nic:2023} for a single unknown $\theta_0 \in \calH\cap\calR$ immediately yields the existence of $L^2_\zeta$ hypothesis tests that are able to distinguish between a simple hypothesis and an alternative given by the complement of a ball, where the complement is separated from the null hypothesis by the posterior contraction rate. Recall from the discussion after \eqref{eq:def_dist} that the rate of minimax distinguishability represents an upper bound for the minimax separation rate. Given the preceding observations, it follows that the minimax $L^2_\zeta$-separation rate is bounded from above by the posterior contraction rate. Recall the sequence $(\delta_N)_{N\in\N}$, the sets $(\Theta_N)_{N\in\N}$, and the mapping $U\mapsto C_U$ defined in \eqref{eq_Nickl_2_19}, \eqref{eq_Nickl_2_20}, and \eqref{eq_C_of_U_function} respectively.
\begin{theorem}
	\label{simple_tests_exist}
	Suppose that \Cref{assumption_Nickl_Condition_2_1_1} holds with $U=U(M)$ and $L=L(M)$ for every $M>0$.
   Suppose $M$ and $\overline{m}=\overline{m}(M,L)$ satisfy \eqref{eq_Nickl_1_26} for every $N\in\N$ and $\theta_0\in B_{\calR}(M)$. 
    In addition, suppose that \Cref{assumption_local_Lipschitz_continuity_of_restricted_inverse_forward_map} holds with $\eta>0$, $L^\prime$, and $\delta_0$.  
   Then there exists a universal constant $0<K<\infty$ such that for every $m > 2\overline{m}$, there exist tests $(\Psi_N)_{N \in \N}$ such that the type $1$ and type $2$ errors for testing
\begin{equation}
	\label{eq:def_hyp_alt}
	H_0: \theta = \theta_0 \quad \text{vs.} \quad H_{1,N} \coloneqq \left\{\theta \in \Theta_N: \norm{\theta - \theta_0}_{L_\zeta^2} > L^\prime(m C_U^{-1/2})^\eta\delta_N^\eta\right\}
\end{equation}
	are bounded by 
	\begin{equation}
	\label{eq:simple_tests_exist_type_1_type_2_bounds}
		P_{\theta_0}^N(\Psi_N = 1) \le  \frac{1}{K}\exp(-(Km^2-1)N\delta_N^2),\qquad	\sup_{\theta \in H_{1,N}} P_\theta^N(\Psi_N = 0) \le  \exp(-Km N\delta_N^2).
\end{equation}
\end{theorem}
\Cref{simple_tests_exist} includes the limiting case in which the alternative is separated from the null at exactly the rate $\delta_N^\eta$, because the prefactor $L^\prime(mC_U^{-1/2})^\eta$ of $\delta_N^\eta$ in the definition of the alternative $H_{1,N}$ in \eqref{eq:def_hyp_alt} is constant with respect to $N$; cf. \cref{remark_limit_case}.

\begin{proof}[Proof of \Cref{simple_tests_exist}]
Since the stated hypotheses include the hypotheses of  \Cref{lemma_existence_of_tests}, we may apply the conclusion and any statement in the proof of \Cref{lemma_existence_of_tests}.
By the conclusion of \Cref{lemma_existence_of_tests}, there exist a universal constant $0<K<\infty$ and tests $(\Psi_N)_{N\in\N}$ such that for every $m>2\overline{m}$, \eqref{eq_Nickl_1_29} holds.
The bound on the type $1$ error in \eqref{eq_Nickl_1_29} yields the bound on the type $1$ error in \eqref{eq:simple_tests_exist_type_1_type_2_bounds}.

To obtain the bound on the type $2$ error in \eqref{eq:simple_tests_exist_type_1_type_2_bounds}, we first note that the alternative hypothesis in \eqref{eq_Nickl_1_29} is of the form $\{ \theta \in \Theta_N: h(p_\theta,p_{\theta_0}) > m\delta_N\}$.
Recall the definition \eqref{semimetric_dG} of the semimetric $d_G$ on $\Theta$.
By the lower bound $C_U^{1/2}d_G(\theta,\vartheta)\leq h(p_\theta,p_{\vartheta})$ from \eqref{eq_Nickl_1_20}, we conclude that
   \begin{equation}
     \label{eq_intermed11a}
     \{ \theta\in\Theta_N:d_G(\theta,\theta_0)>m\delta_N C_U^{-1/2}\}\subset \{\theta\in\Theta_N:h(p_\theta,p_{\theta_0})>m\delta_N\}.
   \end{equation}
By the hypothesis that \Cref{assumption_local_Lipschitz_continuity_of_restricted_inverse_forward_map} holds with $\eta>0$, $L^\prime$, and $\delta_0$, it follows that for every $\delta\leq \delta_0$, $\{\theta\in\Theta_N: d_G(\theta_1,\theta_2)\leq \delta\} \subset \{\theta\in\Theta_N:\norm{\theta_1-\theta_2}_{L^2_\zeta(\calZ,W)}\leq L^\prime\delta^\eta\}$,
see e.g. the proof of \cite[Theorem 2.3.1]{Nic:2023}.
Thus, by taking complements in $\Theta_N$, it follows that for every $\delta\leq \delta_0$,
\begin{equation*}
 \{\theta\in\Theta_N:\norm{\theta_1-\theta_2}_{L^2_\zeta(\calZ,W)}> L^\prime\delta^\eta\}\subset \{\theta\in\Theta_N: d_G(\theta_1,\theta_2)> \delta\} .
\end{equation*}
In particular, for all sufficiently large $N\in\N$ such that $m\delta_N C_U^{-1/2}\leq \delta_0$, it also holds that
\begin{equation}
\label{eq_intermed11b}
 H_{1,N}=\{\theta\in\Theta_N:\norm{\theta_1-\theta_2}_{L^2_\zeta(\calZ,W)}> L^\prime(m\delta_N C_U^{-1/2})^\eta\}\subset \{\theta\in\Theta_N: d_G(\theta_1,\theta_2)> m\delta_N C_U^{-1/2}\} ,
\end{equation}
where we used the definition of $H_{1,N}$ in \eqref{eq:def_hyp_alt}.
Combining \eqref{eq_intermed11a} and \eqref{eq_intermed11b}, we obtain $H_{1,N}\subset \{\theta\in\Theta_N:h(p_\theta,p_{\vartheta})>m\delta_N\}$.
Thus, by the bound on the type $2$ error in \eqref{eq_Nickl_1_29},
\begin{equation*}
 \sup_{\theta\in H_{1,N}} P^N_\theta(\Psi_N=0)
  \leq \sup_{\theta\in\Theta_N:h(p_\theta,p_{\vartheta})>m\delta_N}P^N_\theta(\Psi_N=0) \leq e^{-KmN\delta_N^2},
\end{equation*}
which proves the bound on the type $2$ error in \eqref{eq:simple_tests_exist_type_1_type_2_bounds}. This completes the proof of \Cref{simple_tests_exist}.
\end{proof}

\end{document}